\documentclass[10pt,a4paper]{article}
\usepackage[a4paper]{geometry}
\usepackage{amssymb,latexsym,amsmath,amsfonts,amsthm}
\usepackage{bm}
\usepackage{graphicx}
\usepackage{epsfig}
\usepackage{tikz}
\usepackage{subfigure}
\usepackage{comment}
\usepgflibrary{decorations.markings}
\usetikzlibrary{decorations.markings}

\newcommand{\Tr}{\mathrm{Tr}\,}

\newcommand{\supp}{\mathrm{supp}}
\newcommand{\ir}{\text{\rm{i}}}
\newcommand{\diag}{\mathrm{diag}\,}
\newcommand{\Imag}{\mathrm{Im}\,}
\newcommand{\Real}{\mathrm{Re}\,}
\newcommand{\area}{\text{\rm{area}}}

\newcommand{\ud}{\,\mathrm{d}}

\newcommand{\crit}{\text{\rm{crit}}}

\newcommand{\Res}{\mathrm{Res}}

\newtheorem{theorem}{Theorem}[section]
\newtheorem{lemma}[theorem]{Lemma}
\newtheorem{proposition}[theorem]{Proposition}

\newtheorem{corollary}[theorem]{Corollary}

\theoremstyle{definition}
\newtheorem{definition}[theorem]{Definition}
\newtheorem{rhp}[theorem]{RH problem}

\theoremstyle{remark}

\newtheorem{remark}[theorem]{Remark}

\numberwithin{equation}{section}

\title{The normal matrix model with a monomial potential, a vector equilibrium problem, 
and multiple orthogonal polynomials on a star}

\date{\today}

\author{Arno B.J. Kuijlaars\footnotemark[1]\, and Abey L\'{o}pez-Garc\'{i}a\footnotemark[2]}

\begin{document}

\maketitle
\renewcommand{\thefootnote}{\fnsymbol{footnote}}
\footnotetext[1]{Department of Mathematics, University of Leuven (KU Leuven), Celestijnenlaan 200B, B-3001 Leuven, Belgium. email: arno.kuijlaars\symbol{'100}wis.kuleuven.be}
\footnotetext[2]{Department of Mathematics and Statistics, University of South Alabama, 411 University Blvd North, ILB 325, Mobile, AL 36688, USA, email: lopezgarcia@southalabama.edu}

\begin{abstract}
We investigate the asymptotic behavior of a family of multiple orthogonal polynomials that is naturally linked with the 
normal matrix model with a monomial potential of arbitrary degree $d+1$. 
The polynomials that we investigate are multiple orthogonal with respect to a system of $d$ analytic weights 
defined on a symmetric $(d+1)$-star centered at the origin. 
In the first part we analyze in detail a vector equilibrium problem involving 
a system of $d$ interacting measures $(\mu_{1},\ldots,\mu_{d})$ supported on star-like 
sets in the plane. We show that in the subcritical regime, the first component $\mu_{1}^{*}$  of the solution
to this problem is the asymptotic zero distribution of the multiple orthogonal polynomials. 
It also characterizes the domain where the eigenvalues in the normal matrix model accumulate, 
in the sense that the Schwarz function associated with the boundary of this domain can be 
expressed explicitly in terms of $\mu_{1}^{*}$. The second part of the paper is devoted to the asymptotic 
analysis of the multiple orthogonal polynomials. The asymptotic results are obtained again in the 
subcritical regime, and they follow from the Deift/Zhou steepest descent analysis of 
a Riemann-Hilbert problem of size $(d+1)\times (d+1)$. 
The vector equilibrium problem and the Riemann-Hilbert problem that we investigate are generalizations 
of those studied recently by Bleher-Kuijlaars in the case $d=2$. 

\smallskip




\end{abstract}

\tableofcontents

\section{Introduction}\label{section:introduction}

\subsection{Normal matrix model and Laplacian growth}

The starting point of this work is the analysis of the \textit{normal matrix model}. 
This is a probability measure on the space of $n\times n$ normal matrices $M$ of the form
\begin{equation}\label{eq:NMM}
\frac{1}{Z_{n}}\,e^{-n\, \Tr \mathcal{V}(M)}\ud M.
\end{equation}
A standard expression for the potential $\mathcal{V}$ in \eqref{eq:NMM} is
\begin{equation}\label{potcalV}
\mathcal{V}(M)=\frac{1}{t_{0}}(M M^{*}-V(M)-\overline{V}(M^{*})),\qquad t_{0}>0,
\end{equation}
where $V$ is a polynomial and $\overline{V}$ is the polynomial obtained from $V$ by conjugating 
the coefficients.
This model has attracted considerable interest in recent years and important connections 
have been established with other problems in mathematics and physics, most notably the 
study of Laplacian growth \cite{MPT,LeeTeoWieg,TBAZW,WiegZab}, two-dimensional Coulomb gases \cite{ElbauFelder,HedMak} and 
Toda lattice systems \cite{KKMWZ}, see \cite{Zab} for an overview. Recently, eigenvalue statistics 
in the normal matrix model are considered in \cite{AmHeMa,AmHeMa2,Riser} and 
connections with orthogonal polynomials on the complex plane are also in \cite{BBLM,Elbau,ItsTak,SincYatt}.

According to the theory that has been mainly developed by Wiegmann and Zabrodin, 
the eigenvalues of a matrix $M$ with distribution \eqref{eq:NMM}--\eqref{potcalV} fill 
out, as $n\rightarrow\infty$, a two-dimensional domain $\Omega$ (the \textit{droplet}) 
with uniform density. Moreover, the boundary of $\Omega$ evolves, as the \textit{time} 
parameter $t_{0}>0$ increases, according to the model of Laplacian growth, see \cite{KKMWZ}. 
At a critical time for $t_{0}$, the boundary of $\Omega$ is expected to develop singularities 
and a breakdown takes place in the Laplacian growth evolution.

The model \eqref{eq:NMM}--\eqref{potcalV} requires a certain regularization 
in order to be well-defined since the integral
\begin{equation}\label{partfunc}
Z_{n}=\int e^{-n\, \Tr \mathcal{V}(M)}\ud M
\end{equation}
diverges for any polynomial $V$ of degree $\geq 3$. A natural approach is to use
a cut-off, as proposed by Elbau and Felder \cite{Elbau,ElbauFelder}. This approach consists 
of restricting the model to those normal matrices with spectrum confined in a fixed two-dimensional bounded domain 
containing the droplet $\Omega$.
Another possible approach is to modify $V$ outside of the droplet 
in such a way that \eqref{partfunc} converges, while still keeping the essential features of the model (which take
place on the droplet). This is the approach followed by Ameur-Hedenmalm-Makarov \cite{AmHeMa,AmHeMa2,HedMak}. 

In these models the eigenvalues of $M$ are distributed according to a determinantal point process with a 
correlation kernel that is constructed in terms of polynomials that have a two-dimensional orthogonality.
In the cut-off model of \cite{Elbau,ElbauFelder} with cut-off domain $D$ the orthogonality is 
associated with the scalar product (depending on $n$)
\begin{equation}\label{def:scalarprod}
\langle f,g \rangle_{D}=\iint_{D} f(z)\,\overline{g(z)}\, e^{-n \mathcal{V}(z)} \ud A(z),
\end{equation}
where
\begin{equation}\label{potcalV:2}
\mathcal{V}(z)=\frac{1}{t_{0}}\Big(|z|^{2}-V(z)-\overline{V(z)}\Big),\qquad t_{0}>0,
\end{equation}
and $\ud A$ denotes area measure on $D$. For each $n$, if $(Q_{k,n})_{k=0}^{\infty}$ is the sequence of 
monic polynomials (i.e., $Q_{k,n}(z)=z^{k}+\cdots$) satisfying
\begin{equation}\label{orthogpolyQ}
\langle Q_{k,n}, Q_{j,n}\rangle_{D}=h_{k,n}\,\delta_{j,k},
\end{equation}
then the correlation kernel for the determinantal point process is given by
\[
K_{n}(w,z)=e^{-\frac{n}{2}(\mathcal{V}(z)+\mathcal{V}(w))}\sum_{k=0}^{n-1}\frac{Q_{k,n}(z) \overline{Q_{k,n}(w)}}{h_{k,n}}.
\]
Elbau and Felder also showed in \cite{Elbau,ElbauFelder} that for any polynomial $V$ of the form
\begin{equation}\label{def:V}
V(z)=\sum_{k=1}^{d+1}\frac{t_{k}}{k} z^{k},\qquad t_{k}\in\mathbb{C},
\end{equation}
with $t_{1}=0$ and $|t_{2}|<1$, there exists a compact domain $D$ with $0$ in its interior such that for all 
$t_{0}>0$ small enough, the eigenvalues of $M$ in the model with cut-off $D$ indeed accumulate in a domain 
$\Omega\subset D$ as $n\rightarrow\infty$. The boundary $\partial\Omega$ of $\Omega=\Omega(t_{0};t_{1},\ldots,t_{d+1})$ 
is moreover characterized as the only positively oriented polynomial curve of degree at most $d$ satisfying
\begin{equation} \label{eq:harmmoments}
\frac{1}{2\pi\ir}\oint_{\partial \Omega} \frac{\overline{z}}{z^{k}}\,\ud z=\left\{
\begin{array}{@{\hspace{0cm}}ll}
t_{k}, & k\in\{1,\ldots,d+1\},\\[0.5em]
0, & k\in\mathbb{N}\setminus\{1,\ldots,d+1\},
\end{array}
\right.
\end{equation}
and enclosing a domain with area $\pi t_{0}$. The equations \eqref{eq:harmmoments} are
characteristic for the model of Laplacian growth \cite{WiegZab}.

\subsection{Approach based on sesquilinear forms} 

Motivated by certain boundary integral estimates in the cut-off model and an analysis of the algebraic properties of 
the scalar product \eqref{def:scalarprod}, Bleher and Kuijlaars introduced in \cite{BleherKuij} a new construction which 
replaces the cut-off domain $D$ by a system of unbounded contours in the complex plane. This construction leads to a sequence 
of sesquilinear forms $\langle \cdot,\cdot\rangle_{n}:\mathcal{P}\times \mathcal{P}\longrightarrow\mathbb{C}$ defined on 
the vector space $\mathcal{P}$ of all polynomials in one variable. Such sesquilinear forms are defined as follows. Given 
an integer $d\geq 2$, let us consider the directions at infinity
\begin{equation}\label{pointsinf}
\infty_{\ell}=e^{\frac{(2\ell+1)\pi\ir}{d+1}}\infty,  \qquad \ell = 0, \ldots, d,
\end{equation}
and let $\Gamma_{\ell}$ denote any unbounded oriented contour from $\infty_{\ell-1}$ to $\infty_{\ell}$,
see Figure~\ref{contoursGamma} for the case $d = 3$. 
By $\overline{\Gamma}_{\ell}$ we indicate the image of $\Gamma_{\ell}$ under the map $z\mapsto \overline{z}$, with the induced orientation. Then, given $t_{0}>0$ and a polynomial $V$ as in \eqref{def:V} with $t_{d+1}>0$, the sesquilinear forms are defined by
\begin{equation}\label{def:sesqform}
\langle f,g \rangle_{n}=\sum_{j=0}^{d}\sum_{k=0}^{d} C_{j,k} \int_{\Gamma_{j}} \ud z \int_{\overline{\Gamma}_{k}}\ud w\, f(z)\,\overline{g}(w)\,e^{-\frac{n}{t_{0}}\,(wz-V(z)-\overline{V}(w))},\qquad f,g\in\mathcal{P},
\end{equation}
where $\overline{g}$ is the polynomial obtained from $g$ by conjugating the coefficients
and $C =(C_{j,k})_{j,k=0}^{d}$ is a complex matrix of coefficients.  

\begin{figure}
\begin{center}
\begin{tikzpicture}[scale=1]
\draw [very thin,>=stealth,->] (-3,0) -- (3,0);
\draw [very thin,>=stealth,->] (0,-3) -- (0,3);
\draw [line width =1, smooth, postaction = decorate, decoration = {markings, mark = at position .6 with {\arrow[black]{>};}}] (3.3,-3) .. controls (0.4,0) .. node [near start,right=3pt]{$\Gamma_{0}$} (3.3,3);
\draw [rotate=90, line width =1, smooth, postaction = decorate, decoration = {markings, mark = at position .6 with {\arrow[black]{>};}}] (3.3,-3) .. controls (0.4,0) .. node [near start,above=3pt]{$\Gamma_{1}$} (3.3,3);
\draw [rotate=180, line width =1, smooth, postaction = decorate, decoration = {markings, mark = at position .6 with {\arrow[black]{>};}}] (3.3,-3) .. controls (0.4,0) .. node [near start,left=3pt]{$\Gamma_{2}$} (3.3,3);
\draw [rotate=270, line width =1, smooth, postaction = decorate, decoration = {markings, mark = at position .6 with {\arrow[black]{>};}}] (3.3,-3) .. controls (0.4,0) .. node [near start,below=3pt]{$\Gamma_{3}$} (3.3,3);
\draw (3.3,-3.2) node[scale=1.0, right]{$\infty_3 = \infty_{-1}$};
\draw (3.3,3.2) node[scale=1.0, right]{$\infty_0$};
\draw (-3.3,3.2) node[scale=1.0, left]{$\infty_1$};
\draw (-3.3,-3.2) node[scale=1.0, left]{$\infty_2$};
\end{tikzpicture}
\end{center}
\caption{The contours $\Gamma_{\ell}$ and the points at infinity $\infty_{\ell}$ for $\ell =0,1, \ldots,d$ 
in the case $d=3$.}
\label{contoursGamma}
\end{figure}
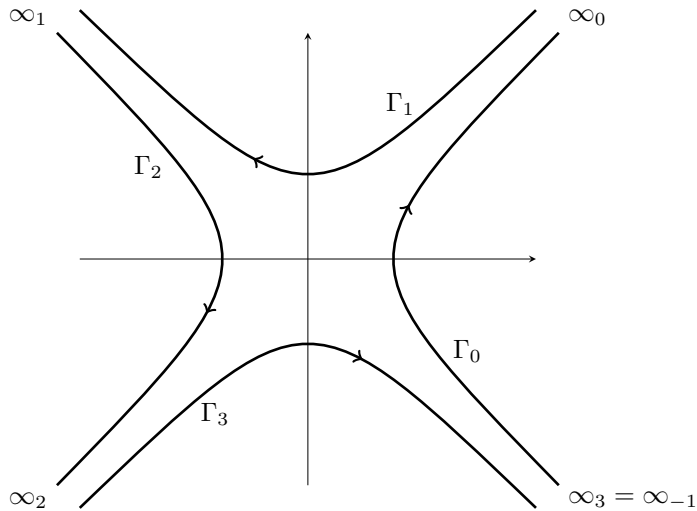

The matrix $C$ is naturally 
assumed in \cite{BleherKuij} to be Hermitian and circulant, in order to guarantee the hermiticity and rotational invariance 
(with respect to the angle $2\pi/(d+1)$) of the sesquilinear forms. In the cubic monomial case
\begin{equation}\label{cubicpot}
V(z)=\frac{t_{3}}{3}\,z^{3},\qquad t_{3}>0,
\end{equation}
an asymptotic analysis of the orthogonal polynomials associated with the corresponding sesquilinear forms was performed in 
\cite{BleherKuij}, for a suitable choice of the matrix $C$. The choice of $C$ is governed by the desire to recover the main 
features of the normal matrix model. In the cubic case \eqref{cubicpot}, it was shown in \cite{BleherKuij} that it is indeed 
possible to find a matrix $C=(C_{j,k})_{j,k=0}^{2}$ such that, in a subcritical regime, the orthogonal polynomials associated 
with the corresponding sesquilinear forms have the same asymptotic behavior as the orthogonal polynomials associated with 
the corresponding scalar product \eqref{def:scalarprod}--\eqref{potcalV:2}. 

The study in \cite{BleherKuij} involved a characterization of the orthogonal polynomials as multiple orthogonal
polynomials with respect to two orthogonality weights defined on the sets $\Gamma_0$, $\Gamma_1$ and $\Gamma_2$, and
built out of Airy functions.  The multiple orthogonality implies a characterization of the orthogonal polynomials
in terms of a $3 \times 3$ matrix valued Riemann-Hilbert problem, which is analyzed  in 
the large $n$ limit. A major role in the asymptotic analysis of \cite{BleherKuij} is played by
a vector equilibrium problem for two measures and an associated three sheeted
Riemann surface. The first component $\mu_1^*$ of the minimizer of the vector equilibrium problem gives
the limiting distribution of the zeros of the orthogonal polynomials. The support of this measure is
a three-star
\[ \supp(\mu_1^*) = [0, x^*] \cup [0, e^{\frac{2\pi \ir}{3}} x^*] \cup [0, e^{\frac{4\pi \ir}{3}} x^*] \]
for a certain $x^*$ that depends on $t_0$ and $t_3$. It is also shown in \cite{BleherKuij}
how the domain $\Omega$  that evolves according to the Laplacian growth
can be recovered from $\mu_1^*$.
This only applies to the subcritical regime, that
is for $t_0$ less than a critical value $t_{0,\crit}$,  depending on $t_3 >0$.

This approach based on sesquilinear forms and multiple orthogonality is extended to the  
supercritical regime in the recent work \cite{KuiTov}, where it is found that 
an analogue of the domain $\Omega$ can still be defined but this domain is no longer growing
as $t_0$ increases, but instead it shrinks to a point at a second critical value.
Lee et al.\ \cite{LeeTeoWieg2} argue that the
Painlev\'e I equation provides a model for the shock phenomomenon at the critical value
$t_{0,\crit}$. This connection with Painlev\'e I is also likely to appear, if one would try to
extend the Riemann-Hilbert analysis from \cite{BleherKuij} to the critical case. Indeed,
it would then be necessary to construct a local parametrix out of the Lax pair solutions 
associated with Painlev\'e I, as it is done for example in \cite{DuiKui}. However, this has
not been worked out yet.

\subsection{Aim of the paper}

In this paper we study monomial $V$ of higher degree, that is,
\begin{equation}\label{monomialpot}
V(z)=\frac{t_{d+1}}{d+1}\,z^{d+1},\qquad t_{d+1}>0,
\end{equation}
where $d \geq 3$.  
Our initial expectation was that the approach of \cite{BleherKuij} for the case $d=2$ would
carry through, with proper but non-essential modifications. That is, we expected
a formulation of the orthogonal polynomials as multiple orthogonal polynomials
with $d$ weights built out of solutions of the differential equation
\begin{equation} \label{eq:pODE} 
	p^{(d)}(z) = (-1)^d z p(z) 
	\end{equation}
of order $d$, which for $d=2$ is the Airy differential equation. This gives rise to a
RH problem of size $(d+1) \times (d+1)$. An ingredient for the asymptotic analysis
of this RH problem would be  an appropriate vector equilibrium problem with $d$ measures 
which is related to a $(d+1)$-sheeted Riemann surface. The first component $\mu_1^*$ 
of the minimizer of the vector equilibrium problem would be the limiting zero distribution of the
orthogonal polynomials. The support of $\mu_1^*$ is a $(d+1)$-star
\[ \supp(\mu_1^*) = \bigcup_{j=0}^d  [0, \omega^j x^*], \qquad \omega = \exp(\tfrac{2\pi \ir}{d+1}) \]
for some $x^* > 0$, and from $\mu_1^*$ one would recover the domain 
$\Omega = \Omega(t_0; 0, \ldots, 0, t_{d+1}, 0, \ldots)$
satisfying \eqref{eq:harmmoments} and evolving according to the Laplacian
growth model.

It came somewhat as a surprise to us that for the case $d=3$ we could no longer 
select a circulant Hermitian matrix $C = (C_{j,k})_{j,k=0}^3$ 
that gives a sesquilinear form \eqref{def:sesqform} for which we can do
an asymptotic analysis of the corresponding orthogonal polynomials in the way
described above.  In addition to being circulant and Hermitian, the matrix $C$ 
should give rise to multiple orthogonality involving
on the interval $[0, x^*]$ the recessive solution of \eqref{eq:pODE} as $z \to \infty$, 
$\arg z = 0$. This yields a number of conditions  on the coefficients $C_{j,k}$
which turned out to be incompatible with the conditions to be circulant and Hermitian.
We do not have a conceptual proof why this happens.

\subsection{From 2D orthogonality to orthogonality on contours}

For this reason, we had to devise another approach, which is inspired by the work 
of Balogh et al.\ \cite{BBLM}, and consists of first
transforming the two-dimensional orthogonality on $D$ given by \eqref{def:scalarprod}--\eqref{potcalV:2} 
into orthogonality over contours by means of Green's theorem. The contours consist of $\partial D$
together with a certain number of contours within $D$ that for the monomial potential \eqref{monomialpot}
we take as a $(d+1)$-star
\begin{equation}\label{def:starSigma}
  \Sigma := \{ z \in D \mid z^{d+1} \in \mathbb R^+ \},
\end{equation}
see Figure \ref{fig:OmegaD}. We let
\begin{equation} \label{def:omega} 
	\omega := \omega_{d+1} = \exp\left( \frac{2\pi \ir}{d+1}\right) 
	\end{equation}
be the primitive $(d+1)$-st root of unity, and this notation will be used throughout the paper.
We also continue to use the notation $\infty_{\ell}$ as in  \eqref{pointsinf} 
and the contours $\Gamma_{\ell}$ that appear in \eqref{def:sesqform},
where  $\ell$ is considered modulo $d+1$, so that
\[ \infty_{-\ell} = \infty_{d+1-\ell}, \qquad \Gamma_{-\ell} = \Gamma_{d+1-\ell}. \]

\begin{proposition}\label{prop:equivorthog}
Let $V(z)$ be the monomial \eqref{monomialpot} with $d \geq 2$.
Let $D$ be a simply connected two-dimensional Jordan domain with $0$ in its interior, that is
invariant under rotation $z \mapsto \omega z$, as well as under reflection $z \mapsto \bar{z}$ in
the real axis and let $\Sigma$ be as in \eqref{def:starSigma}.
Then, for a polynomial $Q$ and an integer $j \geq 0$, we have
\begin{equation}\label{contourort}
	2\ir\iint_{D} Q(z)\,\overline{z}^{j}\, e^{-\frac{n}{t_{0}}(|z|^2-V(z)-\overline{V(z)})}\, \ud A(z)
	=
	\int_{\Sigma} Q(z)\,w_{j,n}(z)\,\ud z+\oint_{\partial D} Q(z)\, \widetilde{w}_{j,n}(z)\,\ud z,
\end{equation}
where the functions $w_{j,n}(z)$ and $\widetilde{w}_{j,n}(z)$ have the following expressions:
\begin{align}
	w_{j,n}(z) & = \int_{\Gamma_{-\ell}} s^{j}\,e^{-\frac{n}{t_{0}}(sz-V(s)-V(z))} \ud s,
		\qquad \arg z = \frac{2\pi}{d+1} \ell, \label{def:wjn}\\
	\widetilde{w}_{j,n}(z) & = \int_{\infty_{-\ell-1}}^{\overline{z}} s^{j}\,e^{-\frac{n}{t_{0}}(sz-V(s)-V(z))} \ud s,
		\qquad  \frac{2\pi}{d+1} \ell < \arg z < \frac{2\pi}{d+1} (\ell+1),
		\label{def:tildewjn}
\end{align}
where $\ell = 0, \ldots, d+1$. In \eqref{contourort}, each segment of 
$\Sigma$ is given the outward orientation (i.e., away from $0$) and 
$\partial D$ is given the positive counterclockwise orientation.
\end{proposition}
\begin{proof}
Let us define the functions
\begin{equation} \label{def:Fjell}
F_{j,\ell}(z)=\int_{\infty_{-\ell-1}}^{\overline{z}} s^{j}\, e^{-\frac{n}{t_{0}}(sz-V(s))} \ud s,
\qquad j\geq 0,\qquad 0\leq \ell\leq d.
\end{equation}
Let $D_{\ell}$ be the part of $D$ given by
\begin{equation} \label{sectorD1}
	D_{\ell} = \{ z \in D \mid   \frac{2\pi}{d+1} \ell < \arg z < \frac{2\pi}{d+1} (\ell+1) \}.
	\end{equation}
Noting that 
\[ \frac{\partial F_{j,\ell}}{\partial \overline{z}} = \overline{z}^j e^{-\frac{n}{t_0}(|z|^2 - \overline{V(z)})} \]
we find by Green's theorem in the complex plane
\begin{align} \nonumber 
\iint_{D_\ell} Q(z)\,\overline{z}^{j}\, e^{-\frac{n}{t_{0}}(|z|^2-V(z)-\overline{V(z)})} \ud A(z) 
	& = \iint_{D_{\ell}} Q(z)\, e^{\frac{n}{t_0} V(z)}\,\frac{\partial F_{j,\ell}(z)}{\partial \overline{z}} \ud A(z) \\
	\label{Stokes}
  &=\frac{1}{2\ir} \oint_{\partial D_{\ell}} Q(z)\,e^{\frac{n V(z)}{t_{0}}} F_{j,\ell}(z)\,\ud z.
\end{align}
The boundary $\partial D_{\ell}$ consists of the part of $\partial D$ with 
$\frac{2\pi}{d+1} \ell < \arg z < \frac{2\pi}{d+1}(\ell+1)$ and 
the parts in $\Sigma$ with $\arg z = \frac{2\pi}{d+1} \ell$
and $\arg z = \frac{2 \pi}{d+1}(\ell+1)$, the latter one traversed towards the origin.

We take the sum of \eqref{Stokes} for $\ell = 0,\ldots, d$ and we find that
the contribution of $\partial D$  is
\begin{equation} \label{int:dD} 
	\frac{1}{2\ir}\oint_{\partial D} Q(z) \widetilde{w}_{j,n}(z) \ud z 
	\end{equation}
since $e^{\frac{n}{t_0} V(z)} F_{j,\ell}(z) = \widetilde{w}_{j,n}(z)$ for $z \in \partial D_{\ell} \setminus \Sigma$, 
see \eqref{def:tildewjn} and \eqref{def:Fjell}.

Each segment in $\Sigma$ is part of the boundary of two subdomains, say  $D_{\ell}$ and $D_{\ell-1}$
for some $\ell$. We observe by \eqref{def:wjn} and \eqref{def:Fjell} that
\[ e^{\frac{n}{t_0} V(z)} ( F_{j, \ell}(z) - F_{j, \ell-1}(z)) = w_{j,n}(z) \qquad \text{for }  \arg z = \frac{2\pi}{d+1} \ell,
\]
since by definition $\Gamma_{-\ell}$ is a contour from $\infty_{-\ell-1}$ to $\infty_{-\ell}$.
Thus the total contribution of $\Sigma$ in the sum of \eqref{Stokes}
\begin{equation} \label{int:Sigma} 
	\frac{1}{2\ir}\int_{\Sigma} Q(z) w_{j,n}(z) \ud z. 
	\end{equation}
In total we get the sum of \eqref{int:dD} and \eqref{int:Sigma},
which proves \eqref{contourort}.
\end{proof}

\begin{figure}[t]
\begin{center}
\begin{tikzpicture}[scale=1.2]
\draw [postaction=decorate, decoration = {markings, mark = at position .63 with {\arrow[black]{>};}}, 
scale=8, line width=1,domain=0:360, smooth, variable=\t, black] 
plot ({0.2*cos(\t)+0.05*cos(3*\t)},{0.2*sin(\t)-0.05*sin(3*\t)});
\draw [postaction=decorate, decoration = {markings, mark = at position .39 with {\arrow[black]{>};}}, 
scale=8, line width=1,domain=0:360, smooth, variable=\t, black] 
plot ({0.2*cos(\t)+0.05*cos(3*\t)},{0.2*sin(\t)-0.05*sin(3*\t)});
\draw [postaction=decorate, decoration = {markings, mark = at position .12 with {\arrow[black]{>};}}, 
scale=8, line width=1,domain=0:360, smooth, variable=\t, black] 
plot ({0.2*cos(\t)+0.05*cos(3*\t)},{0.2*sin(\t)-0.05*sin(3*\t)});
\draw [postaction=decorate, decoration = {markings, mark = at position .88 with {\arrow[black]{>};}}, 
scale=8, line width=1,domain=0:360, smooth, variable=\t, black] 
plot ({0.2*cos(\t)+0.05*cos(3*\t)},{0.2*sin(\t)-0.05*sin(3*\t)});
\draw [line width=1, postaction=decorate, decoration = {markings, mark = at position .5 with {\arrow[black]{>};}}] (0,0) -- (2,0);
\draw [line width=1, postaction=decorate, decoration = {markings, mark = at position .5 with {\arrow[black]{>};}}] (0,0) -- (0,2);
\draw [line width=1, postaction=decorate, decoration = {markings, mark = at position .5 with {\arrow[black]{>};}}] (0,0) -- (-2,0);
\draw [line width=1, postaction=decorate, decoration = {markings, mark = at position .5 with {\arrow[black]{>};}}] (0,0) -- (0,-2);
\draw (1,0) node[scale=0.8, above=0.3pt]{$\Sigma$};
\draw (0,1) node[scale=0.8, left=0.3pt]{$\Sigma$};
\draw (-1,0) node[scale=0.8, below=0.3pt]{$\Sigma$};
\draw (0,-1) node[scale=0.8, right=0.3pt]{$\Sigma$};
\draw (2.15,0) node[scale=0.8]{$\widehat{x}$};
\draw (0,2.2) node[scale=0.8]{$\ir \widehat{x}$};
\draw (-2.25,0) node[scale=0.8]{$-\widehat{x}$};
\draw (0,-2.2) node[scale=0.8]{$-\ir \widehat{x}$};
\draw (1.2,1.0) node[scale=0.8]{$\partial D$};
\filldraw [black] (2,0) circle (1pt);
\filldraw [black,rotate=90] (2,0) circle (1pt);
\filldraw [black,rotate=180] (2,0) circle (1pt);
\filldraw [black,rotate=270] (2,0) circle (1pt);
\end{tikzpicture}
\end{center}
\caption{A possible domain $D$ and the star $\Sigma$ in case $d=3$. 
The star has endpoints $\omega^\ell \widehat{x}$ for $\ell=0, \ldots, d$.}
\label{fig:OmegaD}
\end{figure}
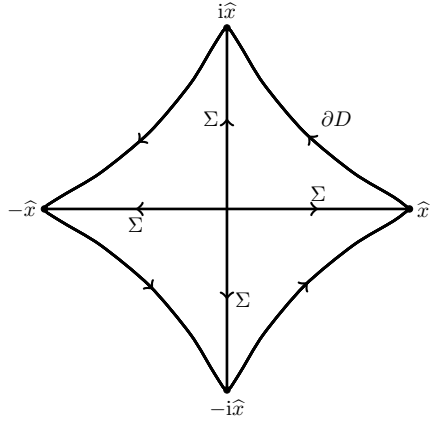

It is a consequence of Proposition \ref{prop:equivorthog} that the 
orthogonal polynomials $Q_{k,n}$ also satisfy
\begin{equation} \label{altdefQkn} 
	\int_{\Sigma} Q_{k,n}(z) w_{j,n}(z) \ud z + \oint_{\partial D} Q_{k,n}(z) \widetilde{w}_{j,n}(z) \ud z = 0,
	\qquad j = 0, \ldots, k-1.
	\end{equation}
	
\subsection{Orthogonality on $(d+1)$-star and multiple orthogonality}
Guided by the idea that in the subcritical case, the domain $D$ can be chosen in such a
way that boundary contributions are negligible in the large $n$ limit, we now 
take the approach to drop the integral over $\partial D$, and consider instead of the
polynomials $Q_{k,n}$ a new set of polynomials $P_{k,n}$ that are orthogonal on the
$(d+1)$-star $\Sigma$ with respect to the weights \eqref{def:wjn}. 
Since these polynomials will be the focus of the present paper we introduce them
with the following formal definition.

\begin{definition} \label{def:orthogpol} Let $d \geq 2$, and $\widehat{x}, t_0, t_{d+1} > 0$.
Then for every $k, n \in \mathbb N$, we let $P_{k,n}(z)=z^{k}+\cdots$ be the 
monic polynomial that satisfies
\begin{equation}\label{ortcond:Pkn}
	\int_{\Sigma} P_{k,n}(z)\,w_{j,n}(z)\,\ud z=0,\qquad j=0,\ldots,k-1,
\end{equation}
where 
\[ \Sigma = \bigcup_{j=0}^d [0, \omega^j \widehat{x}] \]
and with $w_{j,n}(z)$ given in \eqref{def:wjn}.
\end{definition}
The polynomials  depend on the endpoint $\widehat{x}$ of the star. In what follows we are going
to take it in a suitable way, depending on $t_0$ and $t_{d+1}$.

Note that $w_{j,n}$ is analytic on each segment of $\Sigma$.  If we put
\begin{equation} \label{def:pell} 
	p_{\ell}(z) = \frac{1}{2\pi \ir} \int_{\Gamma_{\ell}} e^{-sz + \frac{1}{d+1} s^{d+1}} \ud s 
	\end{equation}
then it is easy to see from \eqref{def:wjn} that
\begin{equation} \label{w0nandpell} 
	w_{0,n}(z) = 2 \pi \ir d_n e^{\frac{n}{t_0} V(z)} p_{-\ell}(c_n z),  \quad \arg z = \frac{2\pi }{d+1} \ell. 
	\end{equation}
with constants 
\begin{equation}\label{def:dn:cn}
 c_{n}=\Big(\frac{n^{d}}{t_{0}^{d} t_{d+1}}\Big)^{\frac{1}{d+1}}, \qquad
 d_{n}=  \Big(\frac{t_{0}}{n t_{d+1}}\Big)^{\frac{1}{d+1}},
\end{equation}
and for every $j \geq 1$,
\begin{equation} \label{wjnandpell} 
	w_{j,n}(z) = 2 \pi \ir (-1)^j d_n^{j+1}  e^{\frac{n}{t_0} V(z)} p_{-\ell}^{(j)}(c_nz), \quad \arg z = \frac{2\pi }{d+1} \ell.
\end{equation}
The functions \eqref{def:pell} satisfy \eqref{eq:pODE} and $p_{-\ell}$ is the recessive solution of \eqref{eq:pODE}
as $z \to \infty$ in the sector $ \frac{2\pi }{d+1}(\ell - \frac{1}{2}) < \arg z < \frac{2\pi }{d+1}(\ell + \frac{1}{2})$.

Inserting the definition  \eqref{def:wjn} of $w_{j,n}$ we can rewrite the orthogonality
\eqref{ortcond:Pkn} in terms of a double integral
\[ 
\sum_{\ell = 0}^d \int_0^{\omega^{\ell} \widehat{x}} \ud z \int_{\Gamma_{-\ell}} \ud s P_{k,n}(z) s^j e^{-\frac{n}{t_0} (sz - V(z) - V(s))} = 0,
	\qquad j=0, 1, \ldots, k-1.
\]
This can be viewed as orthogonality with respect to a sesquilinear form
\begin{equation} \label{def:newsesq} 
	\langle f, g \rangle = \sum_{\ell=0}^d \int_0^{\omega^{\ell} \widehat{x}} \ud z \int_{\Gamma_{-\ell}} \ud s 
	f(z) \overline{g}(s) e^{-\frac{n}{t_0} (sz - V(z) - V(s))},
	 \end{equation}
comparable to \eqref{def:sesqform}.  The new sesquilinear form \eqref{def:newsesq}
is not Hermitian, but it does satisfy the structure relation
\begin{equation} \label{structure}
	t_0 \langle f, g'\rangle - n \langle z f, g \rangle + n \langle f, V'g \rangle=0, 
\end{equation}
as can be seen from an integration by parts on the $s$-integral in \eqref{def:newsesq},
cf.\ also  \cite[section 2.1]{BleherKuij}. 

As in \cite{BleherKuij} it follows from \eqref{structure} that the polynomials $P_{n,n}$
are multiple orthogonal with respect to the system 
of $d$ weights $\{w_{j,n}\}_{j=0}^{d-1}$.

\begin{lemma}\label{lemma:multorthog}
The polynomials $P_{n,n}$ are characterized by the multiple orthogonality conditions
\begin{equation}\label{eq:multorthog}
	\int_{\Sigma}P_{n,n}(z)\,z^{k}\,w_{j,n}(z)\,\ud z=0,\qquad k=0,\ldots,\Big\lceil\frac{n-j}{d}\Big\rceil-1,\qquad j=0,\ldots,d-1.
\end{equation}
\end{lemma}
\begin{proof}
This follows from the structure relation \eqref{structure} in the same manner as in  
\cite[Lemma 5.1]{BleherKuij}.
\end{proof}

Multiple orthogonal polynomials on star-like sets have been studied in recent years under different 
frameworks. Some references on this subject are \cite{AptKalVanIseg,AptKalSaff,DelLop,Lop,Romdhane}. 
See also \cite{EiVar,HeSaff} for the related study of Faber polynomials associated with hypocycloidal domains and stars.

Lemma \ref{lemma:multorthog} allows us to obtain the asymptotic properties of the polynomials 
$P_{n,n}$ through the analysis of a matrix-valued Riemann-Hilbert problem of size $(d+1)\times (d+1)$ 
that encodes the multiple orthogonality conditions \eqref{eq:multorthog}. The first characterization 
of multiple orthogonality in terms of a RH problem appeared in \cite{VAGK}. The RH problem we will 
analyze is described in Section \ref{section:firstRHP}.

The analysis is basically along the same lines as in \cite{BleherKuij}. As in \cite{BleherKuij}
a major role is played by a vector equilibrium problem that we will describe first. 

\section{Statement of results}

In this section we state our main results.

\subsection{The vector equilibrium problem}

In this paper we will obtain the asymptotic behavior of the polynomials $P_{n,n}$ as $n \to \infty$
in a \textit{subcritical} regime. It is characterized  through the solution of a vector equilibrium problem. 
In order to describe this vector equilibrium problem we need to introduce the star-like sets,
where we recall that $\omega = \omega_{d+1} = \exp(2\pi \ir/(d+1))$,
\begin{equation} \label{def:Sigma1}
	\Sigma_{1} := \bigcup_{\ell=0}^{d} [0, \omega^{\ell} \widehat{x}], \qquad \widehat{x} >0,
\end{equation}
and for $2 \leq k \leq d$,
\begin{equation}
	\Sigma_k := \begin{cases}
		\{ z \in \mathbb C \mid z^{d+1} \in \mathbb R^- \} & \text{if $k$ is even}, \\
		\{ z \in \mathbb C \mid z^{d+1} \in \mathbb R^+ \} & \text{if $k$ is odd}.
		\end{cases} \label{def:Sigmak}
		\end{equation}
Thus $\Sigma_k$ is an unbounded set consisting of $d+1$ halfrays for each $k \geq 2$, and they
alternate between $\Sigma_2$ and $\Sigma_3$. Note also that $\Sigma_1 \subset \Sigma_3$.
		
Given two positive measures $\mu$, $\nu$, we use the standard notations
from logarithmic potential theory
\[
I(\mu)=\iint\log\frac{1}{|x-y|}\ud\mu(x)\,\ud\mu(y),\qquad I(\mu,\nu)=\iint\log\frac{1}{|x-y|}\ud\mu(x)\,\ud\nu(y),
\]
for the logarithmic energy of $\mu$ and the mutual logarithmic energy of $\mu$ and $\nu$, respectively. 
Main references in logarithmic potential theory are \cite{Deift,NikSor,Ransford,SaffTotik}.

\begin{definition}\label{def:VEP}
Fix $\widehat{x}, t_{0}, t_{d+1}>0$. The vector equilibrium problem asks for the minimization of the energy functional
\begin{equation}\label{energyfunc}
\sum_{k=1}^{d} I(\mu_{k})-\sum_{k=1}^{d-1}I(\mu_{k},\mu_{k+1})
+\frac{1}{t_{0}}\int\Big(\frac{d}{d+1}\,\frac{1}{t_{d+1}^{1/d}}\,
|z|^\frac{d+1}{d}-\frac{t_{d+1}}{d+1}\,z^{d+1}\Big)\,\ud \mu_{1}(z),
\end{equation}
among all positive Borel measures $\mu_{1},\ldots,\mu_{d}$ satisfying the following conditions:
\begin{itemize}
\item[(1)] For each $k=1,\ldots,d,$ the measure $\mu_{k}$ has total mass
\begin{equation}\label{masses}
\|\mu_{k}\|=1-\frac{k-1}{d}.
\end{equation}
\item[(2)] For each $k=1,\ldots,d,$
\begin{equation}\label{supp:muk}
\supp (\mu_{k})\subset \Sigma_{k},
\end{equation}
where the sets $\Sigma_{k}$ are given in \eqref{def:Sigma1}--\eqref{def:Sigmak}.
\end{itemize}
\end{definition}

Observe that this vector equilibrium problem depends exclusively on $\widehat{x}, t_{0}, t_{d+1}>0$. 
For any choice of these parameters, this vector equilibrium problem is weakly admissible in the 
sense of \cite{HardyKuij}, and therefore it admits a unique minimizer $(\mu_{1}^{*},\ldots,\mu_{d}^{*})$. 
See also \cite{BKMW} for a similar class of vector equilibrium problems. As we indicated before, in 
this paper we are interested in the solution of the stated problem under special assumptions on the 
parameters involved, that is, in a subcritical situation. The vector equilibrium problem was 
analyzed in \cite{BleherKuij} in the case $d=2$.

\subsection{Subcritical case}

\begin{theorem}\label{theo:subcrit}
Let $d\geq 2$ be an arbitrary integer. Fix $t_{d+1}>0$ and set
\begin{equation}\label{def:t0crit}
t_{0,\crit}=(d^{-\frac{2}{d-1}}-d^{-\frac{d+1}{d-1}}) \, t_{d+1}^{-\frac{2}{d-1}} >0.
\end{equation}
Let $0<t_{0} <  t_{0,\crit}$ and define
\begin{equation}\label{def:xstar}
x^{*}=(d+1)\,d^{-\frac{d}{d+1}}\,t_{d+1}^{\frac{1}{d+1}}\,r^{\frac{2d}{d+1}},
\end{equation}
where $r$ denotes the smallest positive solution of the equation
\begin{equation}\label{def:r}
t_{0}=r^2-d\,t_{d+1}^{2}\,r^{2d}
\end{equation}
(which exists  because of $t_0 <  t_{0,\crit}$).
Then there is an $x^{**} >  x^*$ such that for every $\widehat{x} \in [x^*, x^{**}]$,
the unique minimizer $(\mu_1^*, \ldots, \mu_d^*)$ of the vector equilibrium problem
of Definition \ref{def:VEP} satisfies the following.

\begin{enumerate}
\item[\rm (a)] All measures are invariant with respect to rotation $z \mapsto \omega z$.
\item[\rm (b)] $\mu_1^*$ is supported on 
\begin{equation} \label{def:Sigma1star}
	\supp(\mu_1^*) = \Sigma_1^* = \bigcup_{j=0}^d [0, \omega^j x^*]
	\end{equation}
	and the density
of $\mu_1^*$ is positive on $[0, x^*)$ and vanishes like a square root at $x^*$.
\item[\rm (c)] For $k = 2, \ldots, d$, the measure $\mu_k^*$ is supported on 
\begin{equation} \label{def:Sigmakstar}
	\supp(\mu_k^*) = \Sigma_k^* = \Sigma_k,
	\end{equation} 
	see \eqref{def:Sigmak}, with a positive density.
\end{enumerate}
\end{theorem}

The condition $t_0 < t_{0,\crit}$ represents the subcritical case. 
In the critical case $t_0 = t_{0,\crit}$, we can take $\hat{x} = x^*$ and then all of the
above still holds, except that the density of $\mu_1^*$ vanishes to higher order at the
endpoints.

\subsection{Asymptotics of $P_{n,n}$ in subcritical case}

We assume $t_{d+1} > 0$ is fixed and we let $0 < t_0 < t_{0,\crit}$, where 
$t_{0,\crit}$ is given by \eqref{def:t0crit}. Let $x^* > 0$ and $r > 0$ be as in
\eqref{def:xstar}--\eqref{def:r}. Then, for a choice of $\widehat{x} > x^*$ sufficiently
close to $x^*$ we can perform an asymptotic analysis of the RH problem for
the polynomials $P_{n,n}$. This analysis occupies a large  part of the paper.

As an outcome of the analysis we have strong and uniform asymptotic expansions for
the polynomials $P_{n,n}$ as $n \to \infty$ in all domains in the complex plane.
These include the exponential asymptotics in the domain $\mathbb C \setminus \Sigma_1^*$,
the oscillatory asymptotics on $\Sigma_1^*$ and turning point asymptotics involving
Airy functions at the endpoints $\omega^j x^*$.  We do not spell out all results in detail
but we only give the exponential asymptotics (which is the simplest), which takes
the following form.

\begin{theorem} \label{theo:strongasymp}
Let $d\geq 2$ be an integer. Assume that $t_{d+1}>0$, $0<t_{0}<t_{0,\crit}$, 
and let $x^{*}>0$ be given by \eqref{def:xstar}--\eqref{def:r}. Then, 
for all $\widehat{x}>x^{*}$ sufficiently close to $x^{*}$, the polynomials $P_{n,n}$ exist 
and are unique for all $n$ large enough that are a multiple of $d$. The polynomials $P_{n,n}$ satisfy
\begin{equation}\label{strongasympform}
P_{n,n}(z)=\mathbf{M}_{1,1}(z)\,e^{n g_{1}(z)}\,(1+O(1/n)),\qquad \text{as } n\rightarrow\infty,
\end{equation}
uniformly on compact subsets of $\mathbb{C}\setminus\Sigma_{1}^*$. Here $g_{1}$ is given by
\[
g_{1}(z)=\int\log(z-t)\,\ud \mu_{1}^{*}(t),\qquad z\in\mathbb{C}\setminus\Sigma_{1}^*,
\]
and $\mathbf{M}_{1,1}$ is an analytic function with no zeros in $\mathbb{C}\setminus\Sigma_{1}^*$.
\end{theorem}
The Szeg\H{o} type prefactor $\mathbf{M}_{1,1}$ is denoted this way since it is the $(1,1)$ entry 
of a certain matrix-valued function $\mathbf{M}$. This function $\mathbf{M}$ is the solution to a certain Riemann-Hilbert problem 
described in Section \ref{section:globalparam}. An explicit formula for $\mathbf{M}_{1,1}$ is given in \eqref{formulaM11} 
in terms of a certain meromorphic differential $\bm{\eta}$ defined on a $(d+1)$-sheeted Riemann surface. 

Theorem \ref{theo:strongasymp} is similar to \cite[Lemma 6.1]{BleherKuij} which is for polynomials
that are orthogonal with respect to a sesquilinear form \eqref{def:sesqform} with $d=2$. 
The proof is also along the same lines as the proof in \cite{BleherKuij}.
However, since we are dealing with arbitrary $d$, some issues arise at a number of
places. These issues could be dealt with in an ad hoc way for $d=2$, but now we have to resolve
them more systematically.  

The restriction to $n$ being a multiple of $d$ is made for purposes of presentation only.
There will be no essential difficulty to extend the proof to general values of $n$.
However, doing so  would lengthen the exposition even more (it is already quite long) 
and also complicate the notation, since one would have to distinguish between
$\lfloor \frac{n}{d} \rfloor$ and $\lceil \frac{n}{d} \rceil$ at many places in the paper.
Therefore we decided to restrict ourselves to values of $n$ that are multiple of $d$.

As a corollary of Theorem \ref{theo:strongasymp} we have the 
following asymptotic zero distribution for the sequence of polynomials $P_{n,n}$. 
If $P$ is a polynomial of degree $n$ with zeros $\{z_{j}\}_{j=1}^{n}\subset \mathbb{C}$, we let
\[
\nu(P)=\frac{1}{n}\sum_{j=1}^{n}\delta_{z_{j}}
\]
denote the associated normalized zero counting measure.

\begin{corollary}\label{cor:main:1}
Under the same conditions as in Theorem \ref{theo:strongasymp}, we have that 
all zeros of $P_{n,n}$ accumulate on the star 
$\Sigma_{1}^*$ given by \eqref{def:Sigma1} as $n \to \infty$,  and the sequence 
$(\nu(P_{n,n}))$ converges weakly to $\mu_{1}^{*}$, the first component of 
the solution to the vector equilibrium problem given in Definition \ref{def:VEP}.
\end{corollary}

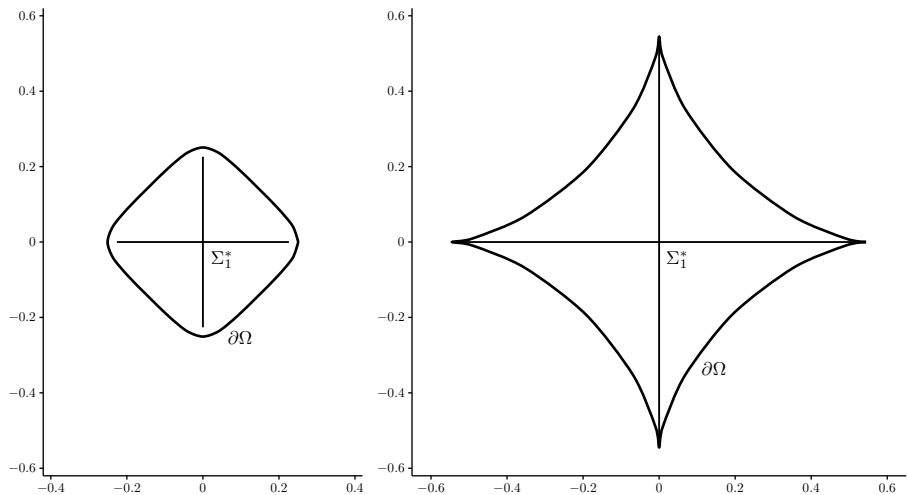
\begin{figure}[t]
\begin{center}
\begin{tikzpicture}[scale=5]
\draw [scale=1,line width=1,domain=0:360, smooth, variable=\t] plot ({0.2272747707857996649689087*cos(\t)+0.0234792208540358571900382*cos(3*\t)},
{0.2272747707857996649689087*sin(\t)-0.0234792208540358571900382*sin(3*\t)});
\draw [line width=0.7] (0,0) -- (0.2261014730906884413472390,0);
\draw [line width=0.7,rotate=90] (0,0) -- (0.2261014730906884413472390,0);
\draw [line width=0.7,rotate=180] (0,0) -- (0.2261014730906884413472390,0);
\draw [line width=0.7,rotate=270] (0,0) -- (0.2261014730906884413472390,0);
\draw (0.05,0) node[below=1pt, scale=0.7]{$\Sigma_{1}^{*}$};
\draw (0,-0.22) node[below,scale=0.7,xshift=0.7cm]{$\partial \Omega$};
\draw [scale=1,line width=1,xshift=1.2cm,domain=0:360, smooth, variable=\t] plot ({0.40824829046*cos(\t)+0.1360827635*cos(3*\t)},
{0.40824829046*sin(\t)-0.1360827635*sin(3*\t)});
\draw [line width=0.7,xshift=1.2cm] (0,0) -- (0.5443310542,0);
\draw [line width=0.7,xshift=1.2cm,rotate=90] (0,0) -- (0.5443310542,0);
\draw [line width=0.7,xshift=1.2cm,rotate=180] (0,0) -- (0.5443310542,0);
\draw [line width=0.7,xshift=1.2cm,rotate=270] (0,0) -- (0.5443310542,0);
\draw (0.1,0) node[below=1pt, scale=0.7,xshift=8.2cm]{$\Sigma_{1}^{*}$};
\draw (0,-0.3) node[below,scale=0.7,xshift=9.6cm]{$\partial \Omega$};
\draw [scale=1,line width=0.5] (-0.42,-0.62) -- (0.42,-0.62);
\draw [line width=0.5] (0,-0.62) -- (0,-0.63);
\draw [line width=0.5] (0.2,-0.62) -- (0.2,-0.63);
\draw [line width=0.5] (0.4,-0.62) -- (0.4,-0.63);
\draw [line width=0.5] (-0.2,-0.62) -- (-0.2,-0.63);
\draw [line width=0.5] (-0.4,-0.62) -- (-0.4,-0.63);
\draw [line width=0.5] (-0.42,-0.6) -- (-0.43,-0.6);
\draw (-0.42,-0.6) node[left=1pt, scale=0.5]{$-0.6$};
\draw (0.2,-0.62) node[below=1pt, scale=0.5]{$0.2$};
\draw (0,-0.62) node[below=1pt, scale=0.5]{$0$};
\draw (0.4,-0.62) node[below=1pt, scale=0.5]{$0.4$};
\draw (-0.2,-0.62) node[below=1pt, scale=0.5]{$-0.2$};
\draw (-0.4,-0.62) node[below=1pt, scale=0.5]{$-0.4$};
\draw (-0.42,-0.4) node[left=1pt, scale=0.5]{$-0.4$};
\draw (-0.42,-0.2) node[left=1pt, scale=0.5]{$-0.2$};
\draw (-0.42,0) node[left=1pt, scale=0.5]{$0$};
\draw (-0.42,0.2) node[left=1pt, scale=0.5]{$0.2$};
\draw (-0.42,0.4) node[left=1pt, scale=0.5]{$0.4$};
\draw (-0.42,0.6) node[left=1pt, scale=0.5]{$0.6$};
\draw [line width=0.5] (-0.42,-0.4) -- (-0.43,-0.4);
\draw [line width=0.5] (-0.42,-0.2) -- (-0.43,-0.2);
\draw [line width=0.5] (-0.42,0) -- (-0.43,0);
\draw [line width=0.5] (-0.42,0.2) -- (-0.43,0.2);
\draw [line width=0.5] (-0.42,0.4) -- (-0.43,0.4);
\draw [line width=0.5] (-0.42,0.6) -- (-0.43,0.6);
\draw [scale=1, line width=0.5] (-0.42,-0.62) -- (-0.42,0.62);
\draw [scale=1, line width=0.5, xshift=1.2cm] (-0.65,-0.62) -- (0.65,-0.62);
\draw [scale=1, line width=0.5, xshift=1.2cm] (-0.65,-0.62) -- (-0.65,0.62);
\draw [line width=0.5, xshift=1.2cm] (-0.65,-0.6) -- (-0.66,-0.6);
\draw [line width=0.5, xshift=1.2cm] (-0.65,-0.4) -- (-0.66,-0.4);
\draw [line width=0.5, xshift=1.2cm] (-0.65,-0.2) -- (-0.66,-0.2);
\draw [line width=0.5, xshift=1.2cm] (-0.65,0) -- (-0.66,0);
\draw [line width=0.5, xshift=1.2cm] (-0.65,0.2) -- (-0.66,0.2);
\draw [line width=0.5, xshift=1.2cm] (-0.65,0.4) -- (-0.66,0.4);
\draw [line width=0.5, xshift=1.2cm] (-0.65,0.6) -- (-0.66,0.6);
\draw [line width=0.5, xshift=1.2cm] (0,-0.62) -- (0,-0.63);
\draw [line width=0.5, xshift=1.2cm] (0.2,-0.62) -- (0.2,-0.63);
\draw [line width=0.5, xshift=1.2cm] (0.4,-0.62) -- (0.4,-0.63);
\draw [line width=0.5, xshift=1.2cm] (-0.2,-0.62) -- (-0.2,-0.63);
\draw [line width=0.5, xshift=1.2cm] (-0.4,-0.62) -- (-0.4,-0.63);
\draw [line width=0.5, xshift=1.2cm] (-0.6,-0.62) -- (-0.6,-0.63);
\draw [line width=0.5, xshift=1.2cm] (0.6,-0.62) -- (0.6,-0.63);
\draw [xshift=1.2cm] (0.2,-0.62) node[below=1pt, scale=0.5]{$0.2$};
\draw [xshift=1.2cm] (0,-0.62) node[below=1pt, scale=0.5]{$0$};
\draw [xshift=1.2cm] (0.4,-0.62) node[below=1pt, scale=0.5]{$0.4$};
\draw [xshift=1.2cm] (-0.2,-0.62) node[below=1pt, scale=0.5]{$-0.2$};
\draw [xshift=1.2cm] (-0.4,-0.62) node[below=1pt, scale=0.5]{$-0.4$};
\draw [xshift=1.2cm] (-0.6,-0.62) node[below=1pt, scale=0.5]{$-0.6$};
\draw [xshift=1.2cm] (0.6,-0.62) node[below=1pt, scale=0.5]{$0.6$};
\draw [xshift=1.2cm] (-0.65,0.6) node[left=1pt, scale=0.5]{$0.6$};
\draw [xshift=1.2cm] (-0.65,0.4) node[left=1pt, scale=0.5]{$0.4$};
\draw [xshift=1.2cm] (-0.65,0.2) node[left=1pt, scale=0.5]{$0.2$};
\draw [xshift=1.2cm] (-0.65,0) node[left=1pt, scale=0.5]{$0$};
\draw [xshift=1.2cm] (-0.65,-0.2) node[left=1pt, scale=0.5]{$-0.2$};
\draw [xshift=1.2cm] (-0.65,-0.4) node[left=1pt, scale=0.5]{$-0.4$};
\draw [xshift=1.2cm] (-0.65,-0.6) node[left=1pt, scale=0.5]{$-0.6$};
\end{tikzpicture}
\end{center}
\caption{
The boundary $\partial \Omega$ of the domain $\Omega$ and the star $\Sigma_{1}^*$ that
supports the measure $\mu_1^*$ for the case $d=3$ and $t_4=2$ at a 
subcritical time $t_{0}=\frac{1}{20}$ (on the left) and at the critical time 
$t_{0}=t_{0,\crit}=\frac{1}{9}$ (on the right).}
\label{fig:OmegaStar}
\end{figure}

\subsection{Laplacian growth and the spectral curve}

Given the data $t_0$, $t_{d+1}$ and $x^*$ as in Theorem \ref{theo:strongasymp}, we let
$\mu_{1}^{*}$ be the first component of the minimizer of the corresponding vector equilibrium problem. 

\begin{definition} \label{def:Schwarz}
We define $\xi_1 : \mathbb C \setminus \Sigma_1^* \to \mathbb C$ by
\begin{equation}\label{rel:xi1mu1}
\xi_{1}(z)=t_{d+1}\,z^{d}+t_{0}\int\frac{\ud \mu_{1}^{*}(\zeta)}{z-\zeta}.
\end{equation}
\end{definition}

\begin{theorem} \label{theo:Omega} Let $t_{d+1} > 0$ and $0 < t_0 < t_{0,\crit}$.
Then there is a simply connected bounded domain $\Omega$, whose boundary
satisfies the equation
\begin{equation}\label{def:boundOmega}
	\xi_1(z) = \overline{z}, \qquad z \in \partial \Omega.
\end{equation}
The star $\Sigma_1^*$ is contained in $\Omega$. 
For an integer $k\geq 0$ we have
\begin{equation}\label{ehm}
	\frac{1}{2\pi\ir}\oint_{\partial\Omega}\frac{\overline{z}}{z^{k}}\,\ud z=\left\{\begin{array}{@{\hspace{0cm}}ll}
t_{0}, & \text{if } k=0,\\[0.3em]
t_{d+1}, & \text{if } k=d+1,\\[0.3em]
0, & \text{otherwise},
\end{array}
\right.
\end{equation}
where $\partial \Omega$ is given the positive orientation. Moreover, 
$\mu_{1}^{*}$ and the normalized area measure on $\Omega$ are related by
\begin{equation}\label{rel:mu1:area}
\int\frac{\ud\mu_{1}^{*}(\zeta)}{z-\zeta}=\frac{1}{\pi t_{0}}\iint\frac{\ud A(\zeta)}{z-\zeta},\qquad z\in\mathbb{C}\setminus\Omega.
\end{equation}
\end{theorem}

Theorem \ref{theo:Omega} generalizes the result obtained in \cite[Theorem 2.6]{BleherKuij} in the case $d=2$. 
Observe that \eqref{def:boundOmega} indicates that $\xi_1$ is 
the Schwarz function associated with the curve $\partial\Omega$. The property \eqref{ehm} with $k=0$
means that
\[ \area (\Omega) = \pi t_0, \]
and so in terms of the parameter $t_0$, the domain $\Omega$ evolves according to the equations of Laplacian growth,
see Figure \ref{fig:OmegaStar} for an illustration in the case $d=3$.

Our final result is that $\xi_1$ satisfies an algebraic equation, the so-called spectral curve.

\begin{theorem}\label{theo:spectralcurve}
The Schwarz function \eqref{rel:xi1mu1} satisfies an algebraic equation of degree $d+1$ with 
real coefficients of the form
\begin{equation}\label{algequat}
P(z,\xi)=\xi^{d+1}+z^{d+1}-\sum_{k=1}^{d} c_{k} z^{k} \xi^{k}+\beta=0,
\end{equation}
where $c_{k}>0$ for all $k=1,\ldots,d,$ and $\beta>0$.

The coefficients $c_{d}$ and $c_{d-1}$ in  \eqref{algequat} have the form
\begin{align}
c_{d} & =t_{d+1},\label{def:cd}\\
c_{d-1} & = \begin{cases} t_{0}\,t_{d+1}, & \text{for } d\geq 3,\\[0.3em] 
	t_{0}\,t_{3}+\frac{1}{t_{3}}, & \text{for } d=2.
\end{cases} \label{def:cdm1}
\end{align}

\end{theorem}

The function $\xi_{1}(z)$ is the only solution of \eqref{algequat} that can be defined as an analytic 
function in the exterior of $\Sigma_{1}^*$. Because of \eqref{algequat}, it has an analytic continuation
to a $(d+1)$-sheeted Riemann surface. The analytic continuations  are also important for the asymptotic analysis
that follows.

\subsection{Outline of the paper}

The rest of the paper is organized as follows. In Section \ref{section:vep} we analyze the properties of the solution 
to the vector equilibrium problem and prove Theorems \ref{theo:subcrit}, \ref{theo:Omega} and \ref{theo:spectralcurve}. In Section 
\ref{section:firstRHP} we describe the RH 
problem associated with the multiple orthogonal polynomials $P_{n,n}$. As a result of the steepest descent 
analysis of this RH problem, we will obtain the strong asymptotic behavior of the polynomials $P_{n,n}$ outside 
the star $\Sigma_{1}^*$, see Theorem \ref{theo:strongasymp}. The steepest descent analysis is developed in 
Sections \ref{section:firsttransf}--\ref{section:conclusion}.

\section{Proofs of Theorems \ref{theo:subcrit}, \ref{theo:Omega} and \ref{theo:spectralcurve}} \label{section:vep}

In this section we analyze the vector equilibrium problem introduced in Definition \ref{def:VEP}, in the subcritical regime
and thereby prove the Theorems  \ref{theo:subcrit}, \ref{theo:Omega} and \ref{theo:spectralcurve}.

\subsection{The Riemann surface in the subcritical case} \label{subsection:xik}

Let $t_{d+1}>0$ be fixed. Then it is easy to see that the function
\[ r \mapsto r^2-d\,t_{d+1}^{2}\,r^{2d} \]
attains its maximum value on the positive real axis at the point $r_{\crit}=(d\,t_{d+1})^{-\frac{1}{d-1}}$, and
the maximum value is exactly given by $t_{0,\crit}$ in \eqref{def:t0crit}. The function is increasing
on $[0,r_{\crit}]$. See also Figure \ref{graph:f}.
Thus for $t_{0}\in(0,t_{0,\crit}]$, there is a unique $r \in [0,r_{\crit}]$ such that \eqref{def:r} holds. 
Then $x^*$ is defined by \eqref{def:xstar} as in Theorem \ref{theo:subcrit}.
From now on, we assume that $t_0 < t_{0,\crit}$, $r$ and $x^*$ are fixed.

\begin{figure}[t]
\begin{center}
\begin{tikzpicture}[scale=1.2]
\draw [line width=0.5,>=stealth,->] (-1,0) -- (3,0);
\draw [line width=0.5,>=stealth,->] (0,-1) -- (0,3);
\draw [line width=1.2, smooth, rounded corners=8pt] (0,0) .. controls (0.7,0) and (1,2.3) .. (3,-1);
\draw [densely dotted, line width=0.4] (0,0.79) -- (1.25,0.79);
\draw [densely dotted, line width=0.4] (1.25,0.79) -- (1.25,0);
\draw (-0.35,0.79) node[scale=0.8]{$t_{0,\crit}$};
\draw [densely dotted, line width=0.4] (0,0.4) -- (0.55,0.4);
\draw (-0.19,0.4) node[scale=0.8]{$t_{0}$};
\draw [densely dotted, line width=0.4] (0.56,0.4) -- (0.56,0);
\draw (0.56,-0.15) node[scale=0.8]{$r$};
\draw (1.25,-0.18) node[scale=0.8]{$r_{\crit}$};
\end{tikzpicture}
\end{center}
\caption{The graph of the function $r \mapsto r^{2}-d\,t_{d+1}^{2}\,r^{2d}$ for $r \geq 0$.}
\label{graph:f}
\end{figure}
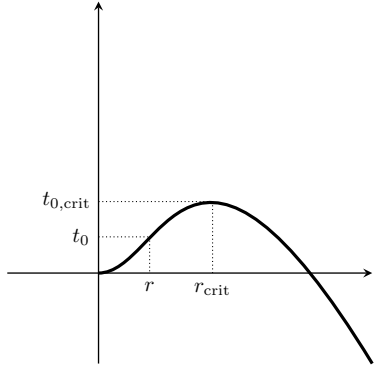

Our starting point is the compact Riemann surface $\mathcal{R}$ of genus zero 
given by the collection of all pairs $(z,\xi)$ of the form
\begin{equation}\label{param:RS}
\begin{array}{l}
z=\psi(w)=rw+\frac{t_{d+1}\,r^{d}}{w^{d}},\\[0.3em]
\xi=\psi(\frac{1}{w})=\frac{r}{w}+t_{d+1}\,r^{d}\,w^{d}, 
\end{array} \qquad w \in \overline{\mathbb C}.
\end{equation}
Let us first examine the sheet structure of this Riemann surface.

The finite branch points of $\mathcal{R}$ are precisely the endpoints $\omega^{\ell} x^{*}$, $\ell =0, \ldots, d$ 
of the star $\Sigma_{1}^*$. Indeed, the derivative $\psi'(w) = r - \frac{d t_{d+1}}{w^{d+1}} r^d$ has zeros 
$\omega^{\ell}\,w^{*}$, $\ell=0, \ldots, d$, where $w^{*}=(d\,t_{d+1}\,r^{d-1})^{\frac{1}{d+1}}$ and
$\psi(\omega^{\ell} w^{*})= \omega^{\ell}\,x^{*}$.

For any $z$, the equation 
\begin{equation} \label{eq:zw} 
z = r w + \frac{t_{d+1} r^d}{w^d} 
\end{equation}
has $d+1$ solutions $w_k(z)$, $k=1, \ldots, d+1$, which we label such that 
\begin{equation}\label{ordering}
	|w_{1}(z)|\geq |w_{2}(z)|\geq \cdots \geq |w_{d+1}(z)|>0.
\end{equation}
The values $w_{k}(z)$ are unambiguously defined at points $z$ where all inequalities in 
\eqref{ordering} are strict. If an equality occurs, we assign an arbitrary 
labeling such that \eqref{ordering} holds. It follows from \cite[Thm. 2.2]{DelLop} 
(see also \cite[Prop. 1]{AptKalSaff}) that the star-like sets \eqref{def:Sigma1star}--\eqref{def:Sigmakstar} 
are such that
\begin{equation} \label{eq:Sigmak}
\Sigma_{k}^*=\{z\in\mathbb{C}: |w_{k}(z)|=|w_{k+1}(z)|\},\qquad k=1,\ldots,d.
\end{equation}
This implies that for each $k$, the solution $w_{k}(z)$ of \eqref{eq:zw} defines an analytic 
function on $\mathbb{C}\setminus(\Sigma_{k-1}^*\cup\Sigma_{k}^*)$ (
with the understanding that $\Sigma_{0}^*=\Sigma_{d+1}^*=\emptyset$) 
and it is analytically continued by $w_{k+1}(z)$ through $\Sigma_{k}^*$.

We use the branches $w_k$ to define the sheet structure on $\mathcal R$. Thus $\mathcal R$
has sheets $\mathcal R_1, \ldots, \mathcal R_{d+1}$ given by
\begin{equation}\label{def:RiemannsurfR}
\begin{aligned}
\mathcal{R}_{1} & =\overline{\mathbb{C}}\setminus\Sigma_{1}^*, \\
\mathcal{R}_{k} & =\overline{\mathbb{C}}\setminus(\Sigma_{k-1}^*\cup\Sigma_{k}^*),\quad 2\leq k\leq d, \\
\mathcal{R}_{d+1}& =\overline{\mathbb{C}}\setminus\Sigma_{d}^*.
\end{aligned}
\end{equation}
For each $k=1,\ldots,d,$ the sheets $\mathcal{R}_{k}$ and $\mathcal{R}_{k+1}$ are glued together 
through the cut $\Sigma_{k}$ in the usual crosswise manner.
See Figure \ref{riemannsurface} for a visualization of the Riemann surface in the case $d=3$. 

\begin{definition}
We define, in accordance with \eqref{param:RS}, the functions
\begin{equation}\label{def:xifunc}
	\xi_{k}(z)=\psi\Big(\frac{1}{w_{k}(z)}\Big)=\frac{r}{w_{k}(z)}+t_{d+1}\,r^{d}\,w_{k}(z)^{d}, 
\end{equation}
for $k = 1, \ldots, d+1$.
\end{definition}
We consider $\xi_k(z)$ to be defined on the $k$th sheet of the Riemann surface,
and then these functions are branches of a meromorphic function on $\mathcal R$.
 
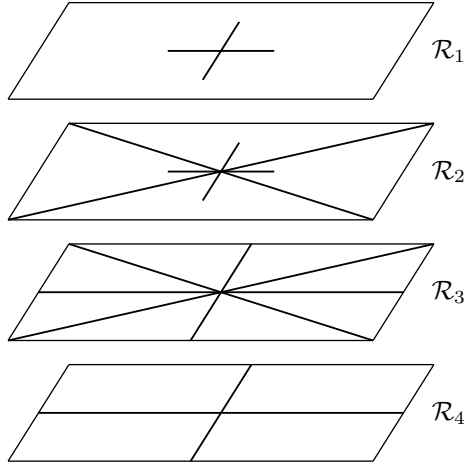
\begin{figure}[t]
\begin{center}
\begin{tikzpicture}[scale = 2,line width = .5]
\draw [line width = .7](-0.35,0) -- (0.35,0);
\draw [line width = .7](-0.12,-0.192) -- (0.12,0.192);
\draw (1,-0.32) -- (1.4,0.32);
\draw (-1.4,-0.32) -- (1,-0.32);
\draw (-1.4,-0.32) -- (-1,0.32);
\draw (-1,0.32) -- (1.4,0.32);
\draw (1.5,0) node[scale=1]{$\mathcal{R}_{1}$};
\draw [line width = .7,yshift=-0.8cm] (-0.35,0) -- (0.35,0);
\draw [line width = .7,yshift=-0.8cm] (-0.12,-0.192) -- (0.12,0.192);
\draw [yshift=-0.8cm] (1,-0.32) -- (1.4,0.32);
\draw [yshift=-0.8cm] (-1.4,-0.32) -- (1,-0.32);
\draw [yshift=-0.8cm] (-1.4,-0.32) -- (-1,0.32);
\draw [yshift=-0.8cm] (-1,0.32) -- (1.4,0.32);
\draw (1.5,0) node[scale=1,yshift=-1.6cm]{$\mathcal{R}_{2}$};
\draw [line width = .7,yshift=-0.8cm] (-1.4,-0.32) -- (1.4,0.32);
\draw [line width = .7,yshift=-0.8cm] (-1,0.32) -- (1,-0.32);
\draw [yshift=-1.6cm] (1,-0.32) -- (1.4,0.32);
\draw [yshift=-1.6cm] (-1.4,-0.32) -- (1,-0.32);
\draw [yshift=-1.6cm] (-1.4,-0.32) -- (-1,0.32);
\draw [yshift=-1.6cm] (-1,0.32) -- (1.4,0.32);
\draw [line width = .7,yshift=-1.6cm] (-1.4,-0.32) -- (1.4,0.32);
\draw [line width = .7,yshift=-1.6cm] (-1,0.32) -- (1,-0.32);
\draw [line width = .7,yshift=-1.6cm] (-1.2,0) -- (1.2,0);
\draw [line width = .7,yshift=-1.6cm] (-0.2,-0.32) -- (0.2,0.32);
\draw (1.5,0) node[scale=1,yshift=-3.2cm]{$\mathcal{R}_{3}$};
\draw [yshift=-2.4cm] (1,-0.32) -- (1.4,0.32);
\draw [yshift=-2.4cm] (-1.4,-0.32) -- (1,-0.32);
\draw [yshift=-2.4cm] (-1.4,-0.32) -- (-1,0.32);
\draw [yshift=-2.4cm] (-1,0.32) -- (1.4,0.32);
\draw [line width = .7,yshift=-2.4cm] (-1.2,0) -- (1.2,0);
\draw [line width = .7,yshift=-2.4cm] (-0.2,-0.32) -- (0.2,0.32);
\draw (1.5,0) node[scale=1,yshift=-4.8cm]{$\mathcal{R}_{4}$};
\end{tikzpicture}
\caption{The Riemann surface $\mathcal{R}$ in the case $d=3$. The Riemann surface has $d+1$ sheets
and $z\mapsto \xi_k(z)$ is considered as an analytic function on the $k$th sheet.}
\label{riemannsurface}
\end{center}
\end{figure}

It easily follows from \eqref{eq:zw} that for each $k=1,\ldots,d+1,$
\begin{align}
w_{k}(\overline{z}) & =\overline{w_{k}(z)},\label{symm:w:1}\\
w_{k}(\omega  z) & =\omega \, w_{k}(z).\label{symm:w:2}
\end{align}
We also deduce easily that as $z\rightarrow\infty$,
\begin{align}
w_{1}(z) & =\frac{z}{r}-\frac{t_{d+1}\,r^{2d-1}}{z^{d}}+O\Big(\frac{1}{z^{2d+1}}\Big),\label{asymp:w1}\\
w_{k}(z) & =O\Big(\frac{1}{z^{1/d}}\Big), & \text{for } k=2,\ldots,d+1.\label{asymp:wk}
\end{align}
The properties \eqref{symm:w:1}--\eqref{asymp:wk} and \eqref{def:xifunc} imply that for $k=1,\ldots,d+1,$
\begin{align}
\xi_{k}(\overline{z}) & = \overline{\xi_{k}(z)},   \label{symm:xik:1}\\
\xi_{k}(\omega z) & = \omega^{-1}\,\xi_{k}(z),   \label{symm:xik:2}
\end{align}
and as $z\rightarrow\infty$,
\begin{align}
\xi_{1}(z) & =t_{d+1}\,z^{d}+\frac{t_{0}}{z}+O\Big(\frac{1}{z^{d+2}}\Big), \label{asymp:xi1}\\
\xi_{k}(z) & =O(z^{1/d}), & \text{for } k=2,\ldots,d+1,\label{asymp:xik}
\end{align}
where in \eqref{asymp:xi1} we used the relation $t_{0}=r^{2}-d\,t_{d+1}^{2}\,r^{2d}$ in identifying the coefficient of $1/z$. We will now give a more detailed description of the asymptotic behavior \eqref{asymp:xik}.

\subsection{Asymptotics of functions $\xi_k$}

In the complex plane we define the infinite sectors
\begin{equation}\label{sectorSl}
S_{\ell}=\{z\in \mathbb{C}: \frac{(2\ell-1)\pi}{d+1}<\arg z<\frac{(2\ell+1)\pi}{d+1}\},\qquad \ell=0,\ldots,d,
\end{equation}
and we subdivide each sector $S_{\ell}$ into two parts
\begin{equation}\label{sectorSlpSlm}
\begin{aligned}
S_{\ell}^{+} & =\{z\in \mathbb{C}: \frac{2\ell\pi}{d+1}<\arg z<\frac{(2\ell+1)\pi}{d+1}\},\\[0.5em]
S_{\ell}^{-} & =\{z\in \mathbb{C}: \frac{(2\ell-1)\pi}{d+1}<\arg z<\frac{2\ell\pi}{d+1}\}.
\end{aligned}
\end{equation}
We are going to use the index $\ell$ with period $d+1$, 
so that $S_{\ell} = S_{\ell+d+1}$. In total we have $2(d+1)$ sectors $S_{\ell}^{\pm}$. 
Frequently, and especially in Lemma \ref{lem:xikasymp} below, we will not number
the sectors from $S_0^-$ to $S_d^+$ but from 
$S_{- \frac{d}{2}}^-$ to $S_{\frac{d}{2}}^+$ in case $d$ is even and from
$S_{- \frac{d+1}{2}}^+$ to $S_{\frac{d+1}{2}}^-$ in case $d$ is odd.
See Figures~\ref{fig:sectors} and \ref{fig:sectors4} for the cases $d=3$ and $d=4$.

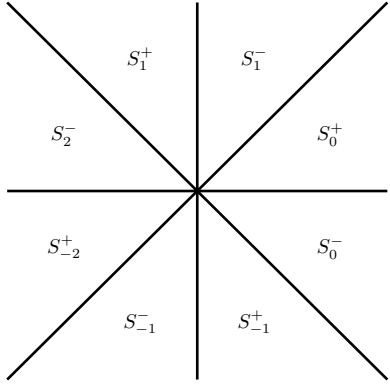
\begin{figure}[t]
\begin{center}
\begin{tikzpicture}
\draw [line width=1] (-2.5,0) -- (2.5,0);
\draw [line width=1] (0,-2.5) -- (0,2.5);
\draw [line width=1] (-2.5,-2.5) -- (2.5,2.5);
\draw [line width=1] (-2.5,2.5) -- (2.5,-2.5);
\draw (1.75,0.75) node[scale=0.75]{$S_{0}^{+}$};
\draw (1.75,-0.75) node[scale=0.75]{$S_{0}^{-}$};
\draw [rotate=90] (1.75,-0.75) node[scale=0.75]{$S_{1}^{-}$};
\draw [rotate=90] (1.75,0.75) node[scale=0.75]{$S_{1}^{+}$};
\draw [rotate=-90] (1.75,-0.75) node[scale=0.75]{$S_{-1}^{-}$};
\draw [rotate=-90] (1.75,0.75) node[scale=0.75]{$S_{-1}^{+}$};
\draw (-1.75,0.75) node[scale=0.75]{$S_{2}^{-}$};
\draw (-1.75,-0.75) node[scale=0.75]{$S_{-2}^{+}$};
\end{tikzpicture}
\caption{The eight sectors $S_{\ell}^{\pm}$ in the case $d=3$.}
\label{fig:sectors}
\end{center}
\end{figure}

\begin{figure}[t]
\begin{center}
\begin{tikzpicture}
\draw [line width=1] (-2.5,0) -- (2.5,0);
\draw [line width=1] (-2.5,-1.8) -- (2.5,1.8);
\draw [line width=1] (-2.5,1.8) -- (2.5,-1.8);
\draw [line width=1] (-0.8,-2.5) -- (0.8,2.5);
\draw [line width=1] (-0.8,2.5) -- (0.8,-2.5);
\draw (1.9,0.6) node[scale=0.75]{$S_{0}^{+}$};
\draw (1.9,-0.6) node[scale=0.75]{$S_{0}^{-}$};
\draw (1.2,1.6) node[scale=0.75]{$S_{1}^{-}$};
\draw (0, 2.0) node[scale=0.75]{$S_{1}^{+}$};
\draw (1.2,-1.6) node[scale=0.75]{$S_{-1}^{+}$};
\draw (0,-2.0) node[scale=0.75]{$S_{-1}^{-}$};
\draw (-1.2,1.6) node[scale=0.75]{$S_2^-$};
\draw (-1.9,0.6) node[scale=0.75]{$S_2^+$};
\draw (-1.2,-1.6) node[scale=0.75]{$S_{-2}^{+}$};
\draw (-1.9,-0.6) node[scale=0.75]{$S_{-2}^{-}$};
\end{tikzpicture}
\caption{The ten sectors $S_{\ell}^{\pm}$ in the case $d=4$.}
\label{fig:sectors4}
\end{center}
\end{figure}
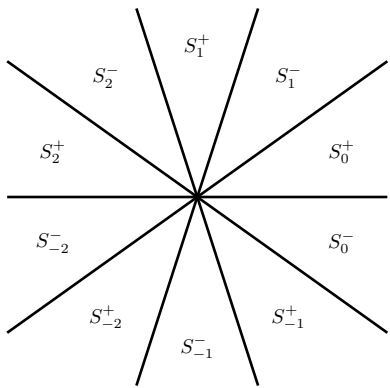

Thus, we have 
\[ \mathbb{C}\setminus(\Sigma_{2}\cup\Sigma_{3}) = \bigcup_{\ell=0}^{d} (S_{\ell}^{+}\cup S_{\ell}^{-}), \]
cf.~\eqref{def:Sigmak} and \eqref{sectorSlpSlm}. 

In Lemma \ref{lem:xikasymp} and throughout the paper we use the principal  branch of fractional exponents,
that is, with a branch cut along the negative real axis.
\begin{lemma}   \label{lem:xikasymp}
For $k =2,\ldots, d+1$, we have
\begin{equation}\label{asymp:xik:2}
	\xi_{k}(z) = \kappa_{k, \ell}^{\pm} \,z^{1/d}-\frac{t_{0}}{d} \, z^{-1} 
	+O\Big(z^{-2 - 1/d} \Big), \qquad \text{as } z \to \infty \text{ in } S_{\ell}^{\pm} 
\end{equation}
where
\begin{equation} \label{def:kappa}
\begin{aligned}
	\kappa_{k,\ell}^+ & = \omega_d^{-\ell + (-1)^k \lfloor \frac{k-1}{2} \rfloor}  t_{d+1}^{-1/d}, && -\lceil d/2 \rceil \leq \ell\leq \lfloor d/2 \rfloor,\\
	\kappa_{k,\ell}^- & = \omega_d^{-\ell - (-1)^k \lfloor \frac{k-1}{2} \rfloor}  t_{d+1}^{-1/d}, && -\lfloor d/2 \rfloor \leq \ell \leq \lceil d/2 \rceil,
\end{aligned}	
\end{equation}
and $\omega_d = \exp( \frac{2\pi \ir}{d})$.
\end{lemma}
\begin{proof}
Let $2\leq k\leq d+1$. Since the function $w_{k}(z)$ vanishes at infinity and satisfies
\[
w_{k}(z)^{d+1}-\frac{z}{r}\,w_{k}(z)^{d}+t_{d+1}\,r^{d-1}=0,
\]
we deduce that as $z\rightarrow\infty$ in one of the sectors $S_{\ell}^{\pm}$,
\begin{equation}\label{asymp:wk:2}
w_{k}(z)= c\, t_{d+1}^{1/d}\, r  z^{-1/d} + c^{2}\, t_{d+1}^{2/d}\, r^{3} d^{-1} z^{-1 - 2/d} 
	+O\Big(z^{-2 - 3/d}\Big),
\end{equation}
where $c$ is a $d$-th root of unity (i.e., $c^d=1$) that depends on $k$ and on the sector. 
In order to respect the ordering \eqref{ordering} we have
\begin{equation} \label{asymp:wk:3} 
\begin{aligned} w_k(z) & =  \omega_d^{\ell - (-1)^k \lfloor \frac{k-1}{2} \rfloor}   t_{d+1}^{1/d}\, r  z^{-1/d} + O(z^{-1 - 2/d})
	&& \text{ as } z \to \infty \text{ in } S_{\ell}^+, \\
	w_k(z) & = \omega_d^{\ell + (-1)^k \lfloor \frac{k-1}{2} \rfloor}  t_{d+1}^{1/d}\, r  z^{-1/d} + O(z^{-1 - 2/d})
	&& \text{ as } z \to \infty \text{ in } S_{\ell}^-.
	\end{aligned}
	\end{equation}
Then applying \eqref{def:xifunc} and \eqref{asymp:wk:2}-\eqref{asymp:wk:3} we arrive at \eqref{asymp:xik:2}.
\end{proof}
Note that the coefficient of $z^{-1}$ in \eqref{asymp:xik:2} does not vary with $k$ and the sector.

\subsection{Construction of the vector equilibrium measure in terms of the functions $\xi_{k}$}
\label{subsection:construction}

We use the functions $\xi_k$ defined in \eqref{def:xifunc} to explicitly solve the
vector equilibrium problem. As already mentioned, the vector equilibrium problem
belongs to the class of weakly admissible problems considered in \cite{HardyKuij}. 
Consequently, it has a unique minimizer $(\mu_{1}^{*},\ldots,\mu_{d}^{*})$. 
The uniqueness readily implies that all the measures $\mu_{k}^{*}$ are  invariant under the
rotation $z \mapsto \omega z$.

Given a positive measure $\mu$, let us denote by $U^{\mu}$ its logarithmic potential, that is,
\begin{equation} \label{def:logpotential}
U^{\mu}(z)=\int\log\frac{1}{|z-t|}\,\ud\mu(t),\qquad z\in\mathbb{C}.
\end{equation}
In what follows we describe some general properties of the vector equilibrium measure $(\mu_{1}^{*},\ldots,\mu_{d}^{*})$ that are valid regardless of the values of $t_{0}$, $t_{d+1}$ and $x^{*}$. Later we will focus on the subcritical regime for these parameters.

Observe from \eqref{energyfunc} that the first measure $\mu_{1}^{*}$ minimizes the energy functional
\[
\mu_{1}\mapsto I(\mu_{1})-I(\mu_{1},\mu_{2}^{*})+\frac{1}{t_{0}}
\int_{\Sigma_{1}}\left(\frac{d}{(d+1)\,t_{d+1}^{1/d}}\,
|z|^\frac{d+1}{d}-\frac{t_{d+1}}{d+1}\,z^{d+1}\right)\,\ud \mu_{1}(z)
\]
among all probability measures on $\Sigma_{1}$. This means that $\mu_{1}^{*}$ is the equilibrium measure 
in the presence of an external field, cf.~\cite{SaffTotik}, therefore $\mu_{1}^{*}$ is characterized by 
the variational conditions
\begin{equation}\label{varcondmu1}
-2 U^{\mu_{1}^{*}}(z)+U^{\mu_{2}^{*}}(z)-
	\frac{1}{t_{0}}\left(\frac{d}{(d+1)\, t_{d+1}^{1/d}}|z|^{\frac{d+1}{d}}-\frac{t_{d+1}}{d+1}z^{d+1}\right)
	\begin{cases} = \ell_{1}, & z\in\supp(\mu_{1}^{*}),\\[0.3em]
	\leq \ell_{1}, & z\in\Sigma_{1}\setminus\supp(\mu_{1}^{*}),
	\end{cases}
\end{equation}
for a certain constant $\ell_{1}$.

Similarly, for a fixed $k$ with $2\leq k\leq d-1$, the measure $\mu_{k}^{*}$ minimizes the energy functional
\[
\mu_{k}\mapsto I(\mu_{k})-I(\mu_{k-1}^{*},\mu_{k})-I(\mu_{k},\mu_{k+1}^{*}),
\]
among all positive measures $\mu_{k}$ satisfying \eqref{masses}--\eqref{supp:muk}. 
Consequently, $\mu_{k}^{*}$ is the balayage of $(\mu_{k-1}^{*}+\mu_{k+1}^{*})/2$ onto 
$\Sigma_{k}$. (See e.g.\ \cite{SaffTotik} for the notion of balayage of a measure in logarithmic
potential theory.) This in turn implies that $\supp(\mu_{k}^{*})=\Sigma_{k}$ and $\mu_{k}^{*}$ 
is characterized by the condition
\begin{equation}\label{varcondmuk}
2 U^{\mu_{k}^{*}}(z)=U^{\mu_{k-1}^{*}}(z)+U^{\mu_{k+1}^{*}}(z),\qquad z\in\Sigma_{k}.
\end{equation}
Analogously, we obtain that $\mu_{d}^{*}$ is the balayage of $\mu_{d-1}^{*}/2$ onto $\Sigma_{d}$, 
and so $\supp(\mu_{d}^{*})=\Sigma_{d}$ and $\mu_{d}^{*}$ is characterized by the condition
\begin{equation}\label{varcondmud}
2 U^{\mu_{d}^{*}}(z)=U^{\mu_{d-1}^{*}}(z),\qquad z\in\Sigma_{d}.
\end{equation}

In what follows we choose an orientation on the stars $\Sigma_1^*$ and $\Sigma_k$, $k=2, \ldots, d$ 
that on each segment is pointing away from the origin. It induces a $+$ and $-$ side on these contours, where the $+$
side lies on the left and the $-$ side on the right as one moves away from the origin. For an oriented
contour $\Sigma$ and a function $f$ defined on a neighborhood of $\Sigma$ in the complex plane, 
we use $f_{\pm}$ to denote the boundary values on $\Sigma$ as one approaches $\Sigma$ from
the $\pm$ side.

\begin{proposition} \label{prop:VEPsol}
Let $t_{d+1} >0$, $t_{0} \in (0, t_{0,\crit})$ and $x^{*}$ be as in Theorem \ref{theo:subcrit}
and let $(\mu_{1}^{*},\ldots,\mu_{d}^{*})$ be the minimizer  of the vector equilibrium problem of
Definition \ref{def:VEP} with $\widehat{x} = x^*$. Then we have the following:
\begin{enumerate}
\item[\rm  (a)] The measure $\mu_1^*$ is given by
\begin{equation}\label{def:mu1}
	\ud\mu_{1}^{*}(z)=\frac{1}{2\pi\ir\, t_{0}}\,(\xi_{1,-}(z)-\xi_{1,+}(z))\,\ud z, \qquad z \in \Sigma_1^*,
\end{equation}
where $\ud z$ is the complex line element on $\Sigma_1^*$,
\item[\rm (b)] For $k \geq 2$, we have that $\mu_k^*$ is given by
\begin{equation}\label{def:muk}
\ud \mu_{k}^{*}(z)
	= \frac{1}{2\pi\ir\, t_{0}}\,\Big(\eta_{k,-}(z)-\eta_{k,+}(z) \Big)		\,\ud z, \qquad z \in \Sigma_k,
\end{equation}
where
\begin{equation} \label{def:etak}
\eta_k(z) = \xi_k(z) - \kappa_{k,\ell}^{\pm} z^{1/d}, \qquad z \in S_{\ell}^{\pm}, 
\end{equation}
and the numbers $\kappa_{k,\ell}^{\pm}$ are as in Lemma \ref{lem:xikasymp}.
\end{enumerate}
\end{proposition}
Note that it follows from \eqref{def:mu1} and \eqref{asymp:xi1} that
\[ \xi_1(z) = t_{d+1} z^d + t_0 \int \frac{d\mu_1^*(s)}{z-s} \]
by Plemelj's formula  for the boundary values of a Cauchy transform. So the definition
of $\xi_1$ in \eqref{def:xifunc} is consistent with the earlier Definition \ref{def:Schwarz}.

To prove Proposition \ref{prop:VEPsol} we assume that $\mu_k^*$ for $k=1, \ldots, d$
are given by the right-hand sides of 
\eqref{def:mu1}--\eqref{def:muk} and we verify that they are real and positive measures with total masses
\begin{equation}\label{mass:muk}
\int_{\Sigma_{k}}\ud \mu_{k}^{*}(z)=1-\frac{k-1}{d},\qquad k=1,\ldots,d,
\end{equation}
satisfying the variational conditions \eqref{varcondmu1}, \eqref{varcondmuk} and \eqref{varcondmud}.
Note that due to \eqref{asymp:xik:2} and \eqref{def:etak} we have
\[ \eta_{k,-}(z) - \eta_{k,+}(z) = O\left( z^{-2 - 1/d}\right)  \text{ as } z \to \infty \]
so that the right-hand side of \eqref{def:muk} is a finite, but a priori complex, measure on $\Sigma_k$.

The proof of Proposition \ref{prop:VEPsol} is somewhat lengthy and it
is subdivided into a number of steps. 

\subsection{Proof of Proposition \ref{prop:VEPsol}}
\subsubsection{Total masses} 
We  start by establishing \eqref{mass:muk}.

\begin{lemma}\label{lemma:massesmuk}
The (a priori complex) measures $\mu_{k}^{*}$ defined in \eqref{def:mu1}--\eqref{def:muk} 
are real-valued, rotationally invariant, and have the total masses \eqref{mass:muk}.
\end{lemma}
\begin{proof}
The property \eqref{symm:xik:1} implies that $\xi_{1,-}(x) = \overline{\xi_{1,+}(x)}$ 
for every $x\in(0,x^{*})$, hence $\mu_{1}^{*}$ is real-valued on $[0,x^{*}]$. But \eqref{symm:xik:2} 
implies immediately that $\mu_{1}^{*}$ is rotationally invariant, so this measure is real-valued everywhere.

Applying \eqref{symm:xik:1}--\eqref{symm:xik:2} for $z\in S_{0}^{-}$ and using \eqref{asymp:xik:2}, 
we deduce that $\kappa_{k,0}^{+}=\overline{\kappa_{k,0}^{-}}$ and 
$\kappa_{k,1}^{-}=\omega_{d}^{-1}\,\kappa_{k,0}^{-}$, hence $\kappa_{k,1}^{-}=\omega_{d}^{-1}\,\overline{\kappa_{k,0}^{+}}$. 
We also deduce from \eqref{symm:xik:1}--\eqref{symm:xik:2} that for $z\in\Sigma_{2}$ with $\arg z=\frac{\pi}{d+1}$, 
we have $\xi_{k,+}(z)=\omega_{d+1}^{-1}\,\overline{\xi_{k,-}(z)}$. With this information we conclude that for $k$ odd, 
$k\geq 3$, the measure $\mu_{k}^{*}$ is real-valued on $[0,\infty)$, and for $k$ even, $k\geq 2$, the measure 
$\mu_{k}^{*}$ is real-valued on $e^{\frac{\pi\ir}{d+1}}\,[0,\infty)$. The rotational invariance is then a 
consequence of \eqref{symm:xik:2}.

We now prove \eqref{mass:muk}. Firstly, since $\xi_{1}$ is analytic in 
$\mathbb{C}\setminus\Sigma_{1}^{*}$, using \eqref{def:mu1}, \eqref{asymp:xi1} and Cauchy's theorem, we obtain
\begin{equation}\label{masscomput:mu1}
\int_{\Sigma_{1}^{*}}\ud\mu_{1}^{*}(z)=\frac{1}{2\pi\ir\,t_{0}}\int_{\Sigma_{1}^{*}}(\xi_{1,-}(z)-\xi_{1,+}(z))\,\ud z=\frac{1}{2\pi\ir\, t_{0}}\oint_{\gamma_{1}}\xi_{1}(z)\,\ud z=1,
\end{equation}
where $\gamma_{1}$ is any positively oriented closed curve surrounding $\Sigma_{1}^{*}$, as in Figure \ref{fig:curves}. 
This proves \eqref{mass:muk} for $k=1$.

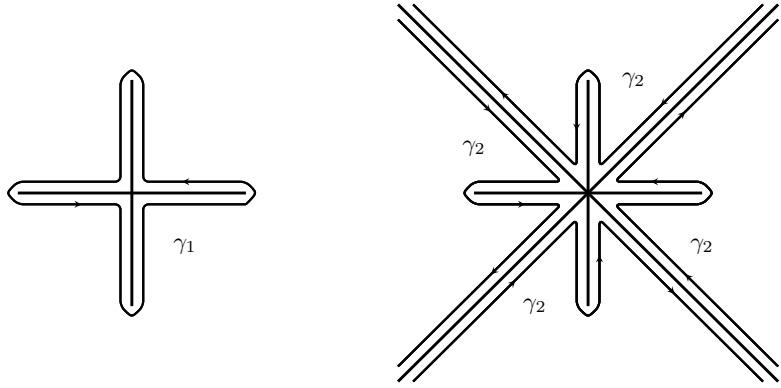
\begin{figure}[t]
\begin{center}
\begin{tikzpicture}
\draw [line width=1.1] (-1.5,0) -- (1.5,0);
\draw [line width=1.1] (0,-1.5) -- (0,1.5);
\draw [line width=0.8,smooth,rounded corners=2pt,postaction = decorate, decoration = {markings, mark = at position .6 with {\arrow[black,line width=0.2]{stealth};}}] (1.5,-0.15) .. controls (1.65,0) .. (1.5,0.15) -- (0.15,0.15) -- (0.15,1.5) .. controls (0,1.65) .. (-0.15,1.5) -- (-0.15,0.15) -- (-1.5,0.15) .. controls (-1.65,0) .. (-1.5,-0.15) -- (-0.15,-0.15) -- (-0.15,-1.5) .. controls (0,-1.65) .. (0.15,-1.5) -- (0.15,-0.15) -- (1.5,-0.15);
\draw [line width=0.8,smooth,rounded corners=2pt,postaction = decorate, decoration = {markings, mark = at position .1 with {\arrow[black,line width=0.2]{stealth};}}] (1.5,-0.15) .. controls (1.65,0) .. (1.5,0.15) -- (0.15,0.15) -- (0.15,1.5) .. controls (0,1.65) .. (-0.15,1.5) -- (-0.15,0.15) -- (-1.5,0.15) .. controls (-1.65,0) .. (-1.5,-0.15) -- (-0.15,-0.15) -- (-0.15,-1.5) .. controls (0,-1.65) .. (0.15,-1.5) -- (0.15,-0.15) -- (1.5,-0.15);
\draw (0.7,-0.7) node[scale=.9]{$\gamma_{1}$};
\draw [line width=1.1, xshift=6cm] (-1.5,0) -- (1.5,0);
\draw [line width=1.1, xshift=6cm] (0,-1.5) -- (0,1.5);
\draw [line width=1.1, xshift=6cm] (-2.5,-2.5) -- (2.5,2.5);
\draw [line width=1.1, xshift=6cm] (-2.5,2.5) -- (2.5,-2.5);
\draw [line width=0.8,xshift=6cm,smooth,rounded corners=2pt,postaction = decorate, decoration = {markings, mark = at position .6 with {\arrow[black,line width=0.2]{stealth};}}] (2.5,-2.3) -- (0.35,-0.15) -- (1.5,-0.15) .. controls (1.65,0) .. (1.5,0.15) -- (0.35,0.15) -- (2.5,2.3);
\draw [line width=0.8,xshift=6cm,smooth,rounded corners=2pt,postaction = decorate, decoration = {markings, mark = at position .2 with {\arrow[black,line width=0.2]{stealth};}}] (2.5,-2.3) -- (0.35,-0.15) -- (1.5,-0.15) .. controls (1.65,0) .. (1.5,0.15) -- (0.35,0.15) -- (2.5,2.3);
\draw [line width=0.8,xshift=6cm,smooth,rounded corners=2pt,postaction = decorate, decoration = {markings, mark = at position .8 with {\arrow[black,line width=0.2]{stealth};}}] (2.5,-2.3) -- (0.35,-0.15) -- (1.5,-0.15) .. controls (1.65,0) .. (1.5,0.15) -- (0.35,0.15) -- (2.5,2.3);
\draw [line width=0.8,xshift=6cm,rotate=90,smooth,rounded corners=2pt,postaction = decorate, decoration = {markings, mark = at position .6 with {\arrow[black,line width=0.2]{stealth};}}] (2.5,-2.3) -- (0.35,-0.15) -- (1.5,-0.15) .. controls (1.65,0) .. (1.5,0.15) -- (0.35,0.15) -- (2.5,2.3);
\draw [line width=0.8,xshift=6cm,rotate=90,smooth,rounded corners=2pt,postaction = decorate, decoration = {markings, mark = at position .22 with {\arrow[black,line width=0.2]{stealth};}}] (2.5,-2.3) -- (0.35,-0.15) -- (1.5,-0.15) .. controls (1.65,0) .. (1.5,0.15) -- (0.35,0.15) -- (2.5,2.3);
\draw [line width=0.8,xshift=6cm,rotate=90,smooth,rounded corners=2pt,postaction = decorate, decoration = {markings, mark = at position .81 with {\arrow[black,line width=0.2]{stealth};}}] (2.5,-2.3) -- (0.35,-0.15) -- (1.5,-0.15) .. controls (1.65,0) .. (1.5,0.15) -- (0.35,0.15) -- (2.5,2.3);
\draw [line width=0.8,xshift=6cm,rotate=180,smooth,rounded corners=2pt,postaction = decorate, decoration = {markings, mark = at position .2 with {\arrow[black,line width=0.2]{stealth};}}] (2.5,-2.3) -- (0.35,-0.15) -- (1.5,-0.15) .. controls (1.65,0) .. (1.5,0.15) -- (0.35,0.15) -- (2.5,2.3);
\draw [line width=0.8,xshift=6cm,rotate=180,smooth,rounded corners=2pt,postaction = decorate, decoration = {markings, mark = at position .6 with {\arrow[black,line width=0.2]{stealth};}}] (2.5,-2.3) -- (0.35,-0.15) -- (1.5,-0.15) .. controls (1.65,0) .. (1.5,0.15) -- (0.35,0.15) -- (2.5,2.3);
\draw [line width=0.8,xshift=6cm,rotate=180,smooth,rounded corners=2pt,postaction = decorate, decoration = {markings, mark = at position .8 with {\arrow[black,line width=0.2]{stealth};}}] (2.5,-2.3) -- (0.35,-0.15) -- (1.5,-0.15) .. controls (1.65,0) .. (1.5,0.15) -- (0.35,0.15) -- (2.5,2.3);
\draw [line width=0.8,xshift=6cm,rotate=270,smooth,rounded corners=2pt,postaction = decorate, decoration = {markings, mark = at position .81 with {\arrow[black,line width=0.2]{stealth};}}] (2.5,-2.3) -- (0.35,-0.15) -- (1.5,-0.15) .. controls (1.65,0) .. (1.5,0.15) -- (0.35,0.15) -- (2.5,2.3);
\draw [line width=0.8,xshift=6cm,rotate=270,smooth,rounded corners=2pt,postaction = decorate, decoration = {markings, mark = at position .6 with {\arrow[black,line width=0.2]{stealth};}}] (2.5,-2.3) -- (0.35,-0.15) -- (1.5,-0.15) .. controls (1.65,0) .. (1.5,0.15) -- (0.35,0.15) -- (2.5,2.3);
\draw [line width=0.8,xshift=6cm,rotate=270,smooth,rounded corners=2pt,postaction = decorate, decoration = {markings, mark = at position .22 with {\arrow[black,line width=0.2]{stealth};}}] (2.5,-2.3) -- (0.35,-0.15) -- (1.5,-0.15) .. controls (1.65,0) .. (1.5,0.15) -- (0.35,0.15) -- (2.5,2.3);
\draw [xshift=6.8cm] (0.7,-0.7) node[scale=.9]{$\gamma_{2}$};
\draw [xshift=6.5cm] (0.1,1.5) node[scale=.9]{$\gamma_{2}$};
\draw [xshift=4.5cm] (0,0.6) node[scale=.9]{$\gamma_{2}$};
\draw [xshift=5.3cm] (0,-1.5) node[scale=.9]{$\gamma_{2}$};
\end{tikzpicture}
\caption{The contour $\gamma_{1}$ around $\Sigma_{1}^*$ (on the left) and the system of unbounded contours $\gamma_{2}$ around 
$\Sigma_{1}^*\cup\Sigma_{2}$ (on the right).}
\label{fig:curves}
\end{center}
\end{figure}

In the case $k=2$, we let $\gamma_{2}$ be a system of unbounded contours around $\Sigma_{1}^* \cup\Sigma_{2}$ 
as shown in  the right part of Figure~\ref{fig:curves}, with orientation as also indicated in Figure~\ref{fig:curves}. 

Noting that $\eta_2(z) = -\frac{t_0}{d} z^{-1} + O(z^{-2-1/d})$ as $z \to \infty$, see \eqref{asymp:xik:2} and \eqref{def:etak},
we find by deforming the contour to infinity in $\mathbb C \setminus (\Sigma_1^* \cup \Sigma_2)$ that
\begin{equation} \label{masscomput:mu2:1} 
	\int_{\gamma_2} \eta_2(z) \ud z = - 2\pi \ir\, \frac{t_0}{d}. 
	\end{equation}
On the other hand, we have by deforming $\gamma_2$ towards $\Sigma_1^* \cup \Sigma_2$ that
\begin{equation} \label{masscomput:mu2:2} 
	\int_{\gamma_2} \eta_2(z)\, \ud z = \int_{\Sigma_2} (\eta_{2,-}(z) - \eta_{2,+}(z)) \ud z
	+ \int_{\Sigma_1^*} (\eta_{2,-}(z)-\eta_{2,+}(z)) \ud z.
	\end{equation}
On $\Sigma_1^*$ we have 
\[ \eta_{2,-}(z) - \eta_{2,+}(z) = \xi_{2,-}(z) - \xi_{2,+}(z), \qquad z \in \Sigma_1^* \]
which follows from \eqref{asymp:xik:2} and \eqref{def:etak} since $\kappa_{2,\ell}^+ = \kappa_{2,\ell}^-$, see \eqref{def:kappa}.
On $\Sigma_1^*$ we also have $\xi_{2,\mp} = \xi_{1,\pm}$, so that by 
\eqref{def:muk}, \eqref{masscomput:mu2:1} and \eqref{masscomput:mu2:2},
\begin{align}  \nonumber
 2 \pi \ir t_0	\int \ud \mu_2^*(z) & = 
	\int_{\Sigma_2} (\eta_{2,-}(z) - \eta_{2,+}(z)) \ud z \\
	& = \nonumber
		- 2\pi \ir \frac{t_0}{d} - \int_{\Sigma_1^*} (\xi_{1,+}(z) - \xi_{1,-}(z)) \ud z  \\
		& = - 2\pi \ir \frac{t_0}{d} + 2\pi \ir t_0  \label{masscomput:mu2:3}
		\end{align}
		where in the last step we used \eqref{masscomput:mu1}.
This proves \eqref{mass:muk} for $k=2$.

The case $k \geq 3$ follows from a similar argument based on induction
and the use of a contour $\gamma_{k}$ around $\Sigma_{k-1}\cup\Sigma_{k}$. Here
it is important that 
\[ \eta_{k,\pm}(z)= \eta_{k-1,\mp}(z), \qquad z\in\Sigma_{k-1}, \]
which holds since $\xi_{k-1,\mp}$ is the analytic continuation of $\xi_{k,\pm}$ across $\Sigma_{k-1}$,
and therefore the leading coeffients in their expansions \eqref{asymp:xik:2} also agree if
we move via one of the half-rays in $\Sigma_{k-1}$ from one sector to the next.  Thus by \eqref{def:etak}
also $\eta_{k-1}$ is the analytic continuation of $\eta_{k}$ across $\Sigma_{k-1}$.

The fact that the leading coefficients agree comes down to the relations 
$\kappa_{k,\ell}^{\pm}  = \kappa_{k-1,\ell}^{\mp}$ if $k$ is even,
and $\kappa_{k, \ell}^{\pm} = \kappa_{k-1,\ell \pm 1}^{\mp}$ if $k$ is odd,
which can also be verified directly from \eqref{def:kappa}.
\end{proof}

\subsubsection{Positivity of $\mu_1^*$}

The next step is to show that $\mu_1^*$ is a positive measure. The proof that
the other measures are positive as well will come later.
We need a lemma about the
function $\psi(w) = rw + t_{d+1} r^d w^{-d}$ introduced in \eqref{param:RS}
that is important later as well.

\begin{lemma} \label{lemma:proppsi} \,

\begin{enumerate}
\item[\rm (a)] The function $\psi$ from \eqref{param:RS} is one-to one 
in the domain $\{w \in \mathbb C \mid |w| \geq \rho \}$ where
\begin{equation} \label{def:rho}
	\rho := (d t_{d+1} r^{d-1})^{1/(d+1)}.  
	\end{equation}
\item[\rm (b)] We have
\begin{equation} \label{ineq:rho}
	\rho \leq 1
\end{equation}
with equality if and only if $t_0 = t_{0,\crit}$.
\item[\rm (c)] $\psi$ satisfies  
\begin{equation}  \label{prop:psi}
	\begin{aligned}
	 \Imag w > 0  \, \Longrightarrow \,  \Imag \psi (w)  > 0, \\ 
	 \Imag w < 0  \, \Longrightarrow \,  \Imag \psi (w)  < 0, \\ 
\end{aligned}
	\quad \text{ for } |w| \geq \rho.
\end{equation}
\item[\rm (d)] For $z = \psi(w)$ with $|w| \geq \rho$, we have
\begin{equation} \label{eq:relxi1psi} 
	\xi_1(z) = \psi \left( \frac{1}{\psi^{-1}(z)}\right). \end{equation}
\item[\rm (e)] We have
\[ \Sigma_1^*  \subset \mathbb C \setminus  \psi(\{w \in \mathbb C \mid |w| > \rho\}). \]
\end{enumerate}
\end{lemma}

\begin{proof}
(a) Let us put $a = t_{d+1} r^d$. Assume to get a contradiction that $\psi(w_1) = \psi(w_2)$
with $w_1 \neq w_2$, $|w_1| \geq |w_2| \geq \rho$ where $\rho$ is given by \eqref{def:rho}.
Thus
\[ r(w_1 - w_2) = a \left( \frac{1}{w_2^d} - \frac{1}{w_1^d} \right) = a \frac{w_1^d - w_2^d}{w_1^d w_2^d}, \]
which since $w_1 \neq w_2$ leads to
\begin{equation} \label{eq:psiproof:1}
	r w_1^d w_2^d = a (w_1^{d-1} + w_1^{d-2} w_2 + \cdots + w_2^{d-1}). 
	\end{equation}
Using $a/r = \frac{1}{d} \rho^{d+1}$, and setting $\lambda_j = w_j/\rho$, for $j=1,2$, 
we can rewrite this further as
\begin{equation} \label{eq:psiproof:2} 
 \lambda_1^d \lambda_2^d =	\frac{1}{d}( \lambda_1^{d-1} + \lambda_1^{d-2} \lambda_2 + \cdots + \lambda_2^{d-1}) 
	\end{equation}
with $|\lambda_1| \geq |\lambda_2| \geq 1$ and $\lambda_1 \neq \lambda_2$.
Taking absolute values on both sides of \eqref{eq:psiproof:2}, we obtain
\begin{equation} \label{eq:psiproof:3} 
	\begin{aligned}
	|\lambda_1|^d  & \leq  |\lambda_1^d \lambda_2^d| = \frac{1}{d} \left| \lambda_1^{d-1} + \lambda_1^{d-2} \lambda_2 + \cdots + \lambda_2^{d-1} \right| \\
	& \leq \frac{1}{d}( |\lambda_1|^{d-1} + |\lambda_1|^{d-2} |\lambda_2| + \cdots + |\lambda_2|^{d-1}|)
		\leq |\lambda_1|^{d-1}, 
		\end{aligned}
		\end{equation}
and from this it follows that $|\lambda_1| = |\lambda_2| = 1$.

It also follows that equality should hold in every inequality in \eqref{eq:psiproof:3}
and in particular
\[ \left| \lambda_1^{d-1} + \lambda_1^{d-2} \lambda_2 + \cdots + \lambda_2^{d-1} \right| 
	=  |\lambda_1|^{d-1} + |\lambda_1|^{d-2} |\lambda_2| + \cdots + |\lambda_2|^{d-1}. \]
This implies that the complex numbers $\lambda_1^{d-1-j} \lambda_2^j$, $j=0, \ldots, d-1$ 
have the same argument, and it is easily seen that this is impossible if $|\lambda_1| = |\lambda_2|$
with $\lambda_1 \neq \lambda_2$. This contradiction proves that $\psi$ is one-to-one on the closed
set $|w| \geq \rho$.
		
\medskip
(b)
For part (b) we note that $r \leq r_{\crit} = (d t_{d+1})^{-1/(d-1)}$, as is easily seen from Figure \ref{graph:f}.
In view of \eqref{def:rho} we get the inequality \eqref{ineq:rho}. Equality holds if and only $r= r_{\crit}$
and this happens if and only if $t_0 = t_{0,\crit}$.

\medskip

(c) Let $|w| \geq \rho$ and $w \in \mathbb C \setminus \mathbb{R}$. Since $\psi(\overline{w}) = \overline{\psi(w)}$ and
$\psi$ is one-to-one, it follows that $\psi(w) \in \mathbb C \setminus \mathbb{R}$, and then part (c) follows.

\medskip

(d) Let $z= \psi(w_0)$ with $|w_0| \geq \rho$. Since $\psi$ is one-to-one for $|w| \geq \rho$, 
it is then clear that $w_0$ is the largest in absolute value solution of \eqref{eq:zw}.  
Thus  $w_0= w_1(z)$ as in \eqref{ordering} and part (d) follows because of the definition \eqref{def:xifunc}. 

\medskip
(e)  If $z = \psi(w_0)$ with $|w_0| > \rho$, then $w_0=w_1(z)$ as in the proof of part (d), and since $\psi$
is one-to-one for $|w| > \rho$, we have $|w_2(z)| \leq \rho$. Then $z \not\in \Sigma_1^*$ because
of the characterization \eqref{eq:Sigmak} of $\Sigma_1^*$. 
\end{proof}

It follows in particular that for $t_0 < t_{0,\crit}$ the curve $\psi( |w| = 1)$ is a closed analytic
curve that surrounds a domain containing $\Sigma_1^*$.

\begin{lemma}\label{lemma:propmu1}
The measure $\mu_1^*$ is positive with $\supp(\mu_1^*) = \Sigma_1^*$.
\end{lemma}

\begin{proof}
By the rotational invariance of $\mu_{1}^{*}$, see Lemma \ref{lemma:massesmuk}, it 
is enough to show that for $0<x<x^{*}$,
\begin{equation}\label{density:mu1}
	\frac{\ud\mu_1^*(x)}{\ud x} = \frac{\xi_{1,-}(x)-\xi_{1,+}(x)}{2\pi\ir\,t_{0}}=\frac{\Imag \xi_{1,-}(x)}{\pi t_{0}}>0,
\end{equation}
see \eqref{def:mu1}. By \eqref{def:xifunc} we have
\begin{equation} \label{eq:xi1min} 
	\xi_{1,-}(x) = \psi \left( \frac{1}{w_{1,-}(x)} \right), 
	\end{equation}
where $w_{1,-}(x)$ is a solution of $\psi(w) = x$ of highest absolute value. We have 
that $\Imag w_{1,-}(x) \neq 0$.  

Since $x \in \Sigma_1^*$ we find by Lemma \ref{lemma:proppsi} (e) that $|w_{1,-}(x)| \leq  \rho$.
Thus $ \frac{1}{|w_{1,-}(x)|} \geq \frac{1}{\rho} \geq \rho$  (since $\rho \leq 1$, see Lemma \ref{lemma:proppsi} (b)).
Applying Lemma \ref{lemma:proppsi} (c) and \eqref{eq:xi1min}, we find $\Imag \xi_{1,-}(x) \neq 0$.
Thus by \eqref{density:mu1} the density of $\mu_1^*$ does not vanish on $(0,x^*)$. By Lemma \ref{lemma:massesmuk}
and the rotational symmetry we already know that
\[ \int_0^{x^*} \ud\mu_1^*(x) = \frac{1}{d+1} > 0. \]
The density is real and continuous, and since it does not vanish we conclude it is positive.
This proves \eqref{density:mu1}.

It is clear that $\supp(\mu_1^*) = \Sigma_1^*$.
\end{proof}

\subsubsection{Identities for Cauchy transforms}

The next steps depend on the properties of the Cauchy transforms of
the measures $\mu_k^*$ that we define as follows.

\begin{definition}\label{def:Cauchytransf}
For $k=1, \ldots, d$, we define $F_k$ as 
\begin{equation}\label{def:Fk}
F_{k}(z):=\int\frac{\ud \mu_{k}^{*}(t)}{z-t},\qquad z\in\mathbb{C}\setminus\Sigma_{k},\qquad k=1,\ldots,d.
\end{equation}
\end{definition}

These functions are closely related to the functions $\xi_k$ as shown in the following
lemma. 

\begin{lemma}\label{lemma:rel:Fkxik}
The following relations hold between the functions $\xi_{k}$ defined in \eqref{def:xifunc} 
and the functions $F_{k}$ defined in \eqref{def:Fk}:
\begin{equation}\label{rel:xi1F1}
\xi_{1}(z)=t_{d+1}\,z^{d}+t_{0}\,F_{1}(z),\qquad z\in\mathbb{C}\setminus\Sigma_{1}^*,
\end{equation}
and for $k=2,\ldots,d+1,$
\begin{equation}\label{rel:xikFk}
\xi_{k}(z)=t_{0}(F_{k}(z)-F_{k-1}(z))+
	\kappa_{k,\ell}^{\pm}\,z^{1/d}, \qquad z\in S_{\ell}^{\pm}, 
	\end{equation}
where $F_{d+1}\equiv 0$, and the coefficients $\kappa_{k,\ell}^{\pm}$ are given in \eqref{def:kappa}.
\end{lemma}
\begin{proof}
With the contour $\gamma_1$ introduced in the proof of Lemma \ref{lemma:massesmuk}, see also Figure \ref{fig:curves},
we have by \eqref{def:Fk} and \eqref{def:mu1} for any $z\in\mathbb{C}\setminus \Sigma_{1}^*$ in the
exterior of $\gamma_1$,
\begin{align*} \nonumber
	t_{0}\,F_{1}(z) & =\frac{1}{2\pi\ir}\int_{\Sigma_{1}^{*}}\frac{\xi_{1,-}(s)-\xi_{1,+}(s)}{z-s}\,\ud s
		=\frac{1}{2\pi\ir}\int_{\gamma_{1}}\frac{\xi_{1}(s)}{z-s}\,\ud s.
		\end{align*}
Moving the contour to infinity we pick up a residue contribution at $s=z$, which is $\xi_1(z)$,
and a contribution at infinity, which because of \eqref{asymp:xi1} is equal to $-t_{d+1} z^d$.
This proves \eqref{rel:xi1F1}.

To prove  \eqref{rel:xikFk} for $k \geq 2$ we use a contour $\gamma_2$ as in Figure \ref{fig:curves}. 
Then for $z \in \mathbb C \setminus (\Sigma_1^* \cup \Sigma_2)$, we find (where we use $\eta_2$ as in \eqref{def:etak})
\begin{align*} 
	2\pi \ir t_0 (F_2(z) - F_1(z)) & = \int_{\Sigma_2} \frac{\eta_{2,-}(s)- \eta_{2,+}(s)}{z-s} \ud s
			- \int_{\Sigma_1^{*}} \frac{\xi_{1,-}(s) - \xi_{1,+}(s)}{z-s} \ud s.
			\end{align*}
We have $-\xi_{1,-} + \xi_{1,+} =-\xi_{2,+} +\xi_{2,-} =-\eta_{2,+} +\eta_{2,-}$ on $\Sigma_1^{*}$,
which implies that
\[ 2 \pi \ir t_0 (F_2(z) - F_1(z)) = \int_{\gamma_2} \frac{\eta_2(s)}{z-s} \ud s, \]
provided that $z$ is in the exterior of $\gamma_2$. Moving $\gamma_2$ to infinity we pick up a residue contribution
$\eta_2(z)$ at $s=z$, and now there is no contribution at infinity, since $\eta_2(s) = O(s^{-1})$ as $s \to \infty$.
This proves \eqref{rel:xikFk}.

The proof for general $k$ is similar.
\end{proof}

\subsubsection{The variational equalities}

\begin{lemma} \label{lemma:varconditions}
The identities \eqref{varcondmu1}, \eqref{varcondmuk} and \eqref{varcondmud} hold
where the measures $\mu_k^*$ are introduced in \eqref{def:mu1} and \eqref{def:muk}.
\end{lemma}
\begin{proof}

Using the identities  \eqref{rel:xi1F1}--\eqref{rel:xikFk} and the fact that 
$\xi_{k,\pm} =\xi_{k+1,\mp}(z)$ on $\Sigma_{k}$, and taking note of the constants \eqref{def:kappa}, 
we easily obtain the following for the Cauchy transforms $F_k$,
\begin{align}
	F_{1,+}(z)+F_{1,-}(z) & = F_{2}(z) + \frac{\omega^{\ell d}}{t_{0}}
	\left(\frac{1}{t_{d+1}^{1/d}}\,|z|^{1/d}-t_{d+1}\,|z|^{d}\right),
	&&  z\in (0, \omega^{\ell} x^{*}] \subset\Sigma_{1},\label{relationF1F2} \\
	F_{k,+}(z)+F_{k,-}(z)&=F_{k-1}(z)+F_{k+1}(z),
		&& z\in\Sigma_{k},\quad k =2, \ldots, d-1, \label{relationFks} \\
	F_{d,+}(z)+F_{d,-}(z)&=F_{d-1}(z),
	  && z\in\Sigma_{d}.\label{relationFdFdm1}
\end{align}
For instance, using \eqref{rel:xi1F1}, \eqref{rel:xikFk} for $k=2$, and $\kappa_{2,0}^{\pm}=t_{d+1}^{-1/d}$, 
the relation $\xi_{1,+} =\xi_{2,-}$ on  $(0,x^{*})$ implies
\[
t_{d+1}\,x^{d}+t_{0}\,F_{1,+}(x)=t_{0}\,(F_{2}(x)-F_{1,-}(x))+\frac{x^{1/d}}{t_{d+1}^{1/d}},\qquad x\in(0,x^{*}),
\]
which gives \eqref{relationF1F2} for $\ell=0$. Then the 
symmetry property $F_{k}(\omega z)=\omega^{d} F_{k}(z)$, $z\in\mathbb{C}\setminus\Sigma_{k}$, 
implies \eqref{relationF1F2} for every $\ell$.

From \eqref{def:logpotential} and \eqref{def:Fk} we get $F_k = - \left(\frac{\partial}{\partial x} -
	i \frac{\partial}{\partial y} \right) U^{\mu_k^*}$. Then in view of the symmetry 
$F_{k}(\overline{z})=\overline{F_{k}(z)}$, we obtain for any $z\in\mathbb{C}\setminus\Sigma_{k}$,
\begin{equation}\label{rel:derivU:F}
-2\frac{\partial}{\partial x}\,U^{\mu_{k}^{*}}(z)=F_{k}(z)+F_{k}(\overline{z}),\qquad 
	-2\ir\,\frac{\partial}{\partial y}\,U^{\mu_{k}^{*}}(z)=F_{k}(\overline{z})-F_{k}(z).
\end{equation}
This shows that for any $x\in(0,x^{*})$,
\[
-2\,\frac{\ud}{\ud x}\,U^{\mu_{1}^{*}}(x)=F_{1,+}(x)+F_{1,-}(x),\qquad -\frac{\ud}{\ud x}\,U^{\mu_{2}^{*}}(x)
=F_{2}(x),
\]
hence from \eqref{relationF1F2} we obtain
\[
-2\,\frac{\ud}{\ud x}\,U^{\mu_{1}^{*}}(x)+\frac{\ud}{\ud x}\,U^{\mu_{2}^{*}}(x)=
\frac{1}{t_{0}}\,\Big(\frac{1}{t_{d+1}^{1/d}}\,x^{1/d}-t_{d+1}\,x^{d}\Big),\qquad x\in[0,x^{*}].
\]
Integrating this relation we get
\[
-2 U^{\mu_{1}^{*}}(x)+U^{\mu_{2}^{*}}(x)-\frac{1}{t_{0}}\,
	\Big(\frac{d}{(d+1)\,t_{d+1}^{1/d}}\,x^{\frac{d+1}{d}}-\frac{t_{d+1}}{d+1}\,x^{d+1}\Big)=\ell_{1},\qquad x\in[0,x^{*}],
\]
for some constant $\ell_{1}$. Using the rotational invariance of the measures $\mu_{k}^{*}$ 
we deduce that \eqref{varcondmu1} holds with equality everywhere on $\Sigma_{1}^*$.

Similarly, using \eqref{relationFks} and the first relation in \eqref{rel:derivU:F}, we argue 
that for $k$ odd, \eqref{varcondmuk} holds for $z\in\mathbb{R}_{+}$ and hence by rotational symmetry it also 
holds for every $z\in\Sigma_{k}$. Here the constant of integration vanishes, as can be seen by inspecting the
behavior as $z \to \infty$.
The same proof is valid for \eqref{varcondmud} if $d$ is odd. The other cases follow in a similar way.
\end{proof}

\subsubsection{Positivity of measures $\mu_k^*$ with $k \geq 2$}

The only piece of information that is still missing is that the measures $\mu_k^*$ are positive for $k \geq 2$.
Recall that the positivity of $\mu_1^*$ was established in Lemma \ref{lemma:propmu1}.

\begin{lemma}  \label{lemma:positivemuk} 
The measures $\mu_k^*$ are positive for $k \geq 2$ with 
$\supp(\mu_k^*) = \Sigma_k^*$.
\end{lemma}

\begin{proof}
 Given the positive measure $\mu_{1}^{*}$ \eqref{def:mu1} on $\Sigma_{1}$, we consider 
the auxiliary vector equilibrium problem consisting of minimizing the energy functional
\[
\sum_{k=2}^{d}I(\mu_{k})-I(\mu_{1}^{*},\mu_{2})-\sum_{k=2}^{d-1} I(\mu_{k},\mu_{k+1})
\]
among all positive Borel measures $\mu_{2},\ldots,\mu_{d}$ satisfying $\supp(\mu_{k})\subset\Sigma_{k}$ 
and $\|\mu_{k}\|=1-\frac{k-1}{d}$ for every $k=2,\ldots,d$.

Let us denote by $(\widehat{\mu}_{2},\ldots,\widehat{\mu}_{d})$ the minimizer to this new problem. 
Since $\mu_{1}^{*}$ is positive, we deduce that $\widehat{\mu}_{2}$ is the balayage of 
$(\mu_{1}^{*}+\widehat{\mu}_{3})/2$ onto $\Sigma_{2}$, therefore we have 
$\supp(\widehat{\mu}_{2})=\Sigma_{2}$ and
\begin{equation}\label{rel:mu2hatmu1}
2 U^{\widehat{\mu}_{2}}(z)=U^{\mu_{1}^{*}}(z)+U^{\widehat{\mu}_{3}}(z),\qquad z\in\Sigma_{2}.
\end{equation}
Similarly we deduce that for each $k=3,\ldots,d,$ the measure $\widehat{\mu}_{k}$ is the balayage 
of $(\widehat{\mu}_{k-1}+\widehat{\mu}_{k+1})/2$ onto $\Sigma_{k}$, and therefore 
$\supp(\widehat{\mu}_{k})=\Sigma_{k}$ and
\begin{equation}\label{rel:mukhat}
2 U^{\widehat{\mu}_{k}}(z)=U^{\widehat{\mu}_{k-1}}(z)+U^{\widehat{\mu}_{k+1}}(z),\qquad z\in\Sigma_{k},
\end{equation}
where $U^{\widehat{\mu}_{d+1}}\equiv 0$.

Using \eqref{varcondmuk} for $k=2$ and \eqref{rel:mu2hatmu1}, we obtain
\begin{equation}\label{rel:pot:mu2mu3}
2(U^{\widehat{\mu}_{2}}(z)-U^{\mu_{2}^{*}}(z))=U^{\widehat{\mu}_{3}}(z)-U^{\mu_{3}^{*}}(z),\qquad z\in\Sigma_{2}.
\end{equation}
Analogously, from \eqref{varcondmuk}--\eqref{varcondmud} and \eqref{rel:mukhat} we obtain 
for each $k=3,\ldots,d-1,$
\begin{equation}\label{rel:pot:muks}
2(U^{\widehat{\mu}_{k}}(z)-U^{\mu_{k}^{*}}(z))=U^{\widehat{\mu}_{k-1}}(z)-U^{\mu_{k-1}^{*}}(z)
+U^{\widehat{\mu}_{k+1}}(z)-U^{\mu_{k+1}^{*}}(z),\qquad z\in\Sigma_{k},
\end{equation}
and
\begin{equation}\label{rel:pot:mudmudm1}
2(U^{\widehat{\mu}_{d}}(z)-U^{\mu_{d}^{*}}(z))=U^{\widehat{\mu}_{d-1}}(z)-U^{\mu_{d-1}^{*}}(z),\qquad z\in\Sigma_{d}.
\end{equation}
Let us now define the  constants
\begin{align} \label{def:mkMk}
m_{k} :=\inf_{z\in\mathbb{C}} \left(U^{\widehat{\mu}_k}(z) - U^{\mu_k^*}(z) \right), \quad
M_{k} :=\sup_{z\in\mathbb{C}} \left(U^{\widehat{\mu}_k}(z) - U^{\mu_k^*}(z) \right), \qquad k=2,\ldots,d.
\end{align}
Since $U^{\widehat{\mu}_k} - U^{\mu_k^*}$ is continuous and vanishes at infinity, 
we have $-\infty < m_k \leq 0 \leq M_{k}<\infty$ for every $k=2,\ldots,d$.

Note that $U^{\widehat{\mu}_k} - U^{\mu_k^*}$ is harmonic in $\mathbb C \setminus \Sigma_k$, which means
by the maximum and minimum principles for harmonic functions, that the maximum and minimum are taken on $\Sigma_k$.
Then \eqref{rel:pot:mu2mu3}--\eqref{rel:pot:mudmudm1}
lead to the inequalities between the numbers $M_k$, namely
\begin{align} \label{rel:Mn:1}
2 M_{2} &\leq M_{3},  \\  \label{rel:Mn:2}
2 M_k & \leq M_{k-1} + M_{k+1},  \text{ for } k =3, \ldots, d-1, \\  \label{rel:Mn:3}
2 M_d & \leq M_{d-1}.
\end{align}
Inductively, we obtain from \eqref{rel:Mn:1} and \eqref{rel:Mn:2} that
\begin{equation} \label{rel:Mn:4} 
	k M_k \leq (k-1) M_{k+1} \qquad \text{ for } k =2, \ldots, d-1. 
	\end{equation}
Taking $k=d-1$ in \eqref{rel:Mn:4} and combining this with \eqref{rel:Mn:3} 
we conclude that $M_d = M_{d-1} = 0$. Then $M_k = 0$ for every $k$ by \eqref{rel:Mn:4}.

In a similar way we prove that $m_k = 0$ for every $k$, and therefore $U^{\mu_{k}^{*}}\equiv U^{\widehat{\mu}_{k}}$.
This implies that $\mu_{k}^{*}=\widehat{\mu}_{k}$ for every $k=2,\ldots,d$ by the uniqueness theorem for logarithmic
potentials, see e.g.\ \cite[Theorem II.2.1]{SaffTotik}. In particular $\mu_k^*$ is a positive measure. 
\end{proof}

\subsection{Conclusion of the proof of Proposition \ref{prop:VEPsol}}

We proved in Lemmas \ref{lemma:propmu1} 
and \ref{lemma:positivemuk} that the measures are real and positive. The total
masses \eqref{mass:muk} are established in Lemma \ref{lemma:massesmuk},
and the variational conditions \eqref{varcondmu1}--\eqref{varcondmud}
are satisfied by Lemma \ref{lemma:varconditions}.

\subsection{Proof of Theorem \ref{theo:subcrit}}

Recall from \eqref{param:RS} and the paragraph that follows it, that 
$\psi'(w^*) = 0$ and $x^* = \psi(w^*)$. Then as $w \to w^*$,
\begin{align} \label{eq:zatbranch} 
	z = \psi(w) = x^* + c_1 (w - w^*)^2 + O(w-w^*)^3 
\end{align}
where $c_1 = \frac{1}{2} \psi''(w^*)> 0$.  
Also from \eqref{param:RS} we have
\begin{align} \label{eq:xiatbranch}	
	\xi = \psi(\frac{1}{w}) = \xi_1(x^*) - c_2 (w-w^*) + O(w-w^*)^2 
\end{align}
with $c_2 = \frac{1}{(w^*)^2} \psi'(\frac{1}{w^*}) > 0$, since $t_0 < t_{0,\crit}$ 
(which implies that $r<r_{\crit}=(d t_{d+1})^{-1/(d-1)}$, see Figure~\ref{graph:f}). 
Inverting \eqref{eq:zatbranch} for $z$ on the first sheet, we get
\begin{equation} \label{eq:zatbranch2} 
	w - w^* =  \frac{1}{\sqrt{c_1}} (z-x^*)^{1/2} + O(z-x^*) \qquad \text{as } z \to x^*, 
	\end{equation}
with the principal branch of the square root since $w > w^*$ for real $z > x^*$.
Thus by \eqref{eq:zatbranch2} and \eqref{eq:xiatbranch} 
\begin{align} \label{eq:xi1atbranch} 
	\xi_1(z) = \xi_1 (x^*) - c_3 (z-x^*)^{1/2} + O(z-x^*) 
	\end{align}
as $z \to x^*$ with $c_3  = \frac{c_2}{\sqrt{c_1}} > 0$.

Then by \eqref{def:mu1} we have for $z \in (0, x^*)$,
\[ \frac{\ud\mu_1^{*}(z)}{\ud z} = \frac{1}{\pi t_0} \Imag \xi_{1,-}(z) = c_4 \sqrt{x^*-z}
	\qquad \text{as } z \nearrow x^*. \]
where $c_4 = \frac{c_3}{\pi t_0} > 0$. Thus the density of $\mu_1^*$ vanishes
as a square root at the endpoint $x^*$. By rotational symmetry the density
vanishes as a square root at each of the endpoints.

Similarly to \eqref{eq:xi1atbranch} we also have
\begin{align} \label{eq:xi2atbranch} 
	\xi_2(z) = \xi_1(x^*) + c_3 (z-x^*)^{1/2} + O(z-x^*),  
	\end{align}
as $z \to x^*$. Then 
\begin{align} \label{eq:xi1minusxi2}
	\xi_1(z) - \xi_2(z) = -2c_3 (z-x^*)^{1/2} + O(z-x^*).  
\end{align}
Now it is easy to calculate that for $x > x^*$,
\begin{align} \nonumber
	\frac{\ud}{\ud x} \left(-2 U^{\mu_1^*}(x) + U^{\mu_2^*}(x) 
	-  \frac{1}{t_0} 		\left( \frac{d}{(d+1) t_{d+1}^{1/d}} x^{\frac{d+1}{d}} - \frac{t_{d+1}}{d+1} x^{d+1} \right)
\right) \\ 
\label{eq:dUnearbranch}
\begin{aligned} 
\qquad \qquad \qquad & = 2F_1(x) - F_2(x) - \frac{1}{t_0} \left( \frac{1}{t_{d+1}^{1/d}} x^{1/d} - t_{d+1} x^d \right) \\
  & = \frac{1}{t_0} (\xi_1(x) - \xi_2(x)),
	\end{aligned}
	\end{align}
where in the last step we used \eqref{rel:xi1F1}, \eqref{rel:xikFk} and \eqref{def:kappa}.
By \eqref{eq:xi1minusxi2} there is  $x^{**} > x^*$ such that $\xi_1(z) - \xi_2(z) < 0$ for $x \in (x^*, x^{**}]$.
Then the variational inequality, see \eqref{varcondmu1}, 
\[ -2 U^{\mu_1^*}(x) + U^{\mu_2^*}(x) - 
		\frac{1}{t_0} \left( \frac{d}{(d+1) t_{d+1}^{1/d}} x^{\frac{d+1}{d}} - \frac{t_{d+1}}{d+1} x^{d+1} \right)
			< \ell_1, \qquad x \in (x^*, x^{**}], \]
follows from \eqref{eq:dUnearbranch} and the fact that we have equality at $x=x^*$.

It now follows that for any choice of $\widehat{x} \in [x^*, x^{**}]$ the measures $(\mu_1^*, \mu_2^*, \ldots, \mu_d^*)$
are the minimizers for the vector equilibrium problem.  Theorem  \ref{theo:subcrit}
is now fully proved.

\subsection{Proof of Theorem \ref{theo:Omega}}

By Lemma \ref{lemma:proppsi} we have that $\psi(|w|=1)$ is a simple closed
curve containing $\Sigma_1^*$ in its interior if $t_0 < t_{0,\crit}$.

We define $\Omega$ as the domain enclosed by the curve $\psi(|w|=1)$.
By parts (b) and (e) of Lemma \ref{lemma:proppsi} we then have that $\Sigma_1^*$
is contained in $\Omega$ if $t_0 < t_{0,\crit}$. 
If $z \in \partial \Omega$ then $z=\psi(w)$ with $|w|=1$. 
By part (d) of Lemma \ref{lemma:proppsi}, we have
$ \xi_1(z) = \psi( \frac{1}{w})$. Using $|w|= 1$ and  the fact that the coefficients of $\psi$ are real,
we then find $\xi_1(z) = \overline{z}$ for $z \in \partial \Omega$.

Using \eqref{def:boundOmega}, we find for an integer $k \geq 0$,
\[
\frac{1}{2\pi\ir}\oint_{\partial\Omega}\frac{\overline{z}}{z^{k}}\,\ud z=\frac{1}{2\pi\ir}\oint_{\partial\Omega}\frac{\xi_{1}(z)}{z^{k}}\,\ud z.
\]
Since $\xi_{1}(z)$ is analytic in the exterior of $\partial\Omega$, by deforming this contour and 
applying \eqref{asymp:xi1} we obtain immediately \eqref{ehm}. 
Note that \eqref{ehm} with $k=0$ implies by Green's theorem that the domain $\Omega$ has area $\pi t_{0}$.

For $z\in\mathbb{C}\setminus\Omega$, we have by Green's theorem and \eqref{def:boundOmega}
\[
\frac{1}{\pi}\iint_{\Omega}\frac{\ud A(\zeta)}{z-\zeta}=\frac{1}{2\pi\ir}\oint_{\partial\Omega}\frac{\overline{\zeta}}{z-\zeta}\,\ud \zeta=\frac{1}{2\pi\ir}\oint_{\partial\Omega}\frac{\xi_{1}(\zeta)}{z-\zeta}\,\ud \zeta.
\]
Using \eqref{asymp:xi1} again and Cauchy's integral formula, the last integral 
is easily seen to be equal to $\xi_{1}(z)-t_{d+1} z^{d}$, and this is $t_{0}\,F_{1}(z)$ by
\eqref{rel:xi1F1}. This proves \eqref{rel:mu1:area}.

\subsection{Proof of Theorem \ref{theo:spectralcurve}}

We write
\begin{equation} \label{def:Pzxi}
P(z,\xi)=\prod_{k=1}^{d+1}(\xi-\xi_{k}(z)),
\end{equation}
where as before we  consider $\xi_k(z)$ as being defined on the $k$th sheet of
the Riemann surface. The equation $P(z,\xi) = 0$ is an equation of the Riemann surface
with rational parametrization \eqref{param:RS}. 


By changing $w \mapsto 1/w$ in the parametrization \eqref{param:RS} we see
that $P(z,\xi) = 0 \implies P(\xi,z) = 0$. This means that $P$ is symmetric in $z$ and $\xi$,
and since $P$ is of degree $d+1$ in $\xi$, we find
\begin{equation} \label{P:expansion2}
P(z,\xi) = \sum_{j,k=0}^{d+1} a_{j,k} z^j \xi^k
\end{equation} 
for certain real constants $a_{j,k}$ that satisfy $a_{k,j} = a_{j,k}$, and so in particular
\[ a_{0,d+1} = a_{d+1,0} = 1. \]
It is also clear by changing $w \mapsto \omega w$ in \eqref{param:RS} where $\omega = \omega_{d+1}$,
that $P(z,\xi) = 0 \implies P(\omega z, \omega^{-1} \xi) = 0$. 
This implies
\[ a_{j,k} \omega^{j-k} = 0, \qquad \text{for all}\,\, j, k= 0, \ldots, d\,\,\text{with}\,\,\omega^{j-k}\neq 1. \]
Hence the only $a_{j,k}$ that are possibly non-zero are those with $j=k$ or $j=0$, $k=d+1$,
or $j=d+1$, $k=0$. It thus follows that $P$ has the form \eqref{algequat}.

Inserting \eqref{param:RS} into \eqref{algequat}, we obtain a Laurent polynomial
in $w$. Setting the coefficients in this expansion equal to $0$, we obtain 
and  recursive relations that determine
the coefficients $c_k$ and $\beta$. In particular we
find \eqref{def:cd}. It remains to show that the coefficients are positive.

In what follows we use the elementary symmetric function
\[ e_k(x_1, \ldots, x_n) = \sum_{j_1 < \ldots < j_k} x_{j_1} x_{j_2} \cdots x_{j_k} \]
and we note that 
\begin{equation} \label{rel:cj:ej}
	c_{d+1-k}z^{d+1-k}  = (-1)^{k-1} e_{k}(\xi_1(z), \ldots, \xi_{d+1}(z)), \qquad k = 1, \ldots, d, 
	\end{equation}
and 
\begin{equation} \label{rel:beta} 
	z^{d+1} + \beta = (-1)^{d+1} e_{d+1}(\xi_1(z), \ldots, \xi_{d+1}(z)) = (-1)^{d+1} \prod_{j=1}^{d+1} \xi_j(z).
	\end{equation}

From \eqref{rel:beta} we get
\[ \beta = (-1)^{d+1} \lim_{z \to 0}  \prod_{j=1}^{d+1} \xi_j(z). \]
By \eqref{eq:zw} we have that $z=0$ corresponds to $w$ values that are solutions of
$rw^{d+1} + a = 0$ where $a = t_{d+1} r^d$.
This implies
\[ \prod_{j=1}^{d+1} w_j(0) = (-1)^{d+1} \frac{a}{r} \]
Also by \eqref{param:RS} we have 
\[ \xi_j(0) = \frac{a w_j(0)^{d+1} + r}{w_j(0)} = \frac{r^2-a^2}{r w_j(0)} \]
and so
\begin{equation} \label{def:beta} 
	\beta = (-1)^{d+1} \prod_{j=1}^{d+1}  \frac{r^2-a^2}{r w_j(0)} = 
	 \left(\frac{r^2-a^2}{r} \right)^{d+1}  \frac{r}{a}. 
	\end{equation}
Since $r  < r_{\crit} = (d t_{d+1})^{-1/(d-1)}$, see Figure \ref{graph:f}, we easily get
$d t_{d+1} r^{d-1} < 1$, and so
\[ a = t_{d+1} r^d < \frac{r}{d} < r. \] 
Then from \eqref{def:beta} we see that $\beta > 0$.

\medskip

In order to show that $c_k > 0$ for $k = 1, \ldots, d-1$, we are going to prove the following claim.
\paragraph{Claim:} For $k = 1, \ldots, d-1$, we have that 
$(-1)^k e_k(\xi_2(z), \ldots, \xi_{d+1}(z))$ has a Laurent expansion at infinity of the form
\begin{equation} \label{claim:1} 
	(-1)^k e_k(\xi_2(z), \ldots, \xi_{d+1}(z)) =  \frac{1}{z^{k}} \sum_{j=0}^{\infty} b_{j,k} z^{-j(d+1)} 
	\end{equation}
with $b_{j,k} > 0$ for all $j$.

\paragraph{Proof of the claim:} We use induction.

For $k=1$, we note that  by \eqref{rel:cj:ej} 
\begin{equation} \label{claim:e1:1} 
	- e_1(\xi_2(z), \ldots, \xi_{d+1}(z)) = \xi_1(z) - c_d z^d =  t_0 F_1(z) 
	\end{equation}
where we used \eqref{rel:xi1F1} and the fact that $c_d = t_{d+1}$. Since $F_1$ is the Cauchy transform of 
$\mu_1^*$ we have
\begin{equation} \label{claim:e1:2} 
	F_1(z) = \frac{1}{z} \left( 1 + \sum_{j=1}^{\infty} s_j z^{-j(d+1)} \right)
		\quad \text{ with } \quad s_j = \int_{\Sigma_1^*} z^{j(d+1)} d\mu_1^*(z).
	\end{equation}
By the  rotational symmetry and the positivity of $\mu_1^*$ we have
\begin{equation} \label{claim:e1:3} 
	s_j = (d+1)\int_0^{x^*}  x^{j(d+1)} d\mu_1^*(x) > 0. 
	\end{equation}
Combining \eqref{claim:e1:1}--\eqref{claim:e1:3} we have proved the claim for $k=1$.

\medskip

Now let $k \geq 2$ with $k \leq d-1$ and assume that the claim is true for $k-1$.
The following basic relation for the elementary symmetric polynomials
\[ e_{k}(\xi_2(z), \ldots, \xi_{d+1}(z)) = e_{k}(\xi_1(z), \xi_2(z), \ldots, \xi_{d+1}(z))
	- \xi_1(z) e_{k-1}(\xi_2(z), \ldots, \xi_{d+1}(z)) \]
leads by \eqref{rel:cj:ej}, \eqref{claim:1} and \eqref{claim:e1:1}  to
\begin{equation} \label{claim:ek:2} 
	(-1)^{k} e_{k}(\xi_2(z), \ldots, \xi_{d+1}(z))  = - c_{d+1-k}z^{d+1-k} + 
	(c_d z^d + t_{0}\, F_1(z))   \frac{1}{z^{k-1}} \sum_{j=0}^{\infty}  b_{j,k-1} z^{-j(d+1)}. 
	\end{equation}
The term with $z^{d+1-k}$ in the right-hand side cancels out, since it is clear from the asymptotic
behavior \eqref{asymp:xik:2} of the $\xi_j$ functions that the elementary symmetric polynomial
cannot grow like $z^{d+1-k}$ as $z \to \infty$. This fact leads to 
\begin{equation} \label{rel:ckb0k} 
	c_{d+1-k} = b_{0,k-1} c_d. 
	\end{equation}
We then obtain from \eqref{claim:e1:2}, \eqref{claim:e1:3} and the induction
hypothesis that the right-hand side of \eqref{claim:ek:2} has a Laurent expansion 
of the form \eqref{claim:1} with positive coefficients $b_{j,k}$, which proves the claim.

\bigskip

Having proved the claim, we know in particular that $b_{0,k} > 0$ for $k=2, \ldots, d-1$, and the relation \eqref{rel:ckb0k} also holds for $k=2,\ldots,d-1$.
These two facts imply that $c_k > 0$ for $k=2, \ldots, d-1$.

The relation \eqref{rel:ckb0k} does not hold for $k=d$, instead we have
\begin{equation}\label{rel:c1cd}
c_{1}=\frac{1}{c_{d}}+b_{0,d-1}\,c_{d}=\frac{1}{t_{d+1}}+b_{0,d-1}\,t_{d+1},
\end{equation}
which implies $c_{1}>0$. So we finish this proof by justifying \eqref{rel:c1cd}.

The relation \eqref{claim:ek:2} also holds for $k=d$, and it implies that
\[
(-1)^{d} e_{d}(\xi_{2}(z),\ldots,\xi_{d+1}(z))=(-1)^{d} \prod_{j=2}^{d+1} \xi_{j}(z)=(-c_{1}+c_{d}\, b_{0,d-1})\,z+O(z^{-d})
\]
as $z\rightarrow\infty$. On the other hand, from \eqref{rel:beta} we deduce 
\[
(-1)^{d} e_{d}(\xi_{2}(z),\ldots,\xi_{d+1}(z))=-\frac{z^{d+1}+\beta}{\xi_{1}(z)}=-(z^{d+1}+\beta)(\frac{1}{t_{d+1}}\,z^{-d}+O(z^{-2d-1})),
\]
as $z\rightarrow\infty$. Identifying the leading coefficients in both expansions we obtain \eqref{rel:c1cd}.

\section{The multiple orthogonal polynomials $P_{n,n}$ and the associated Riemann-Hilbert problem}\label{section:firstRHP}

In this section we start the proof of Theorem \ref{theo:strongasymp}. 
We fix $t_{d+1} > 0$, $0 < t_0 < t_{0,\crit}$ and let $x^*$ be as  \eqref{def:xstar}. 
We also take $\widehat{x} > x^*$ with $\widehat{x} < x^{**}$ as in Theorem \ref{theo:subcrit} and let
$(\mu_1^*, \ldots, \mu_d^*)$ be the minimizer of the vector equilibrium problem
as in Theorem \ref{theo:subcrit}. 
We also use the notions that were developed in Sections \ref{subsection:xik}
and \ref{subsection:construction}, namely
the functions $\xi_k$ given in \eqref{def:xifunc} and the Cauchy transforms $F_k$
given in \eqref{def:Fk} that satisfy the conditions of Lemma \ref{lemma:varconditions}.
These functions will come to play a role in the third transformation of the steepest
descent analysis. 

Throughout we assume that $n$ is a multiple of $d$.

\subsection{The Riemann-Hilbert problem}
We start with the formulation of the Riemann-Hilbert problem for the
polynomials $P_{n,n}$. This is based on the multiple orthogonality stated in Lemma \ref{lemma:multorthog}.
As a result \cite{VAGK} there is a Riemann Hilbert problem of size $(d+1)\times (d+1)$.

Recall that the weight functions $w_{j,n}$ are related to the functions $p_{\ell}$ as in 
\eqref{def:pell}--\eqref{wjnandpell}. The functions $p_{\ell}$ are solutions of \eqref{eq:pODE}.
There are $d+1$ different functions $p_{\ell}$ and any $d$ of them form a basis for the solution space of 
\eqref{eq:pODE}. It is immediate from \eqref{def:pell} that 
\begin{equation}\label{rel:p}
\sum_{\ell=0}^{d} p_{\ell}\equiv 0,
\end{equation}
and
\begin{equation}\label{symmproppl}
p_{\ell}(z)=\omega^{\ell}\, p_{0}(\omega^{\ell}\,z),\qquad \ell=0,\ldots,d,
\end{equation}
recall $\omega = \omega_{d+1}=e^{\frac{2\pi\ir}{d+1}}$.

It will be convenient to forget about the constant prefactors
in \eqref{wjnandpell}, and therefore we introduce the following definition.

\begin{definition}\label{def:vjn}
For $n \in \mathbb{N}$, and $j=0, \ldots, d-1$, we put 
\begin{equation}\label{eq:def:vjn}
	v_{j,n}(z)=e^{\frac{n V(z)}{t_{0}}}\,p_{-\ell}^{(j)}(c_{n}z), \qquad 
		\arg z = \frac{2\pi }{d+1} \ell.
\end{equation}
\end{definition}

By \eqref{wjnandpell} and \eqref{eq:def:vjn}, the weight $v_{j,n}$ is a constant multiple of $w_{j,n}$,
and consequently the polynomial $P_{n,n}$ may be defined alternatively through the multiple orthogonality conditions
\begin{equation}\label{eq:altmultorthog}
	\int_{\Sigma} P_{n,n}(z)\,z^{k}\,v_{j,n}(z)\,\ud z=0,
	\qquad j=0,\ldots,\Big\lceil\frac{n-j}{d}\Big\rceil-1,\qquad j=0,\ldots,d-1.
\end{equation}

We now introduce the RH problem that is associated with the polynomials $P_{n,n}$ and which will be the subject of 
analysis during the rest of the paper. Recall that the star $\Sigma$ is given the outward radial orientation. 
This induces a $+$ and $-$ side on each segment $[0, \omega^j \widehat{x}]$ of $\Sigma$, where the $+$ side ($-$ side) is on the 
left (right) as we traverse the ray according to its orientation. Recall that the star $\Sigma$ is also denoted by $\Sigma_{1}$.

\begin{rhp}\label{RHPforY}
Find a function $\mathbf{Y}:\mathbb{C}\setminus \Sigma \rightarrow\mathbb{C}^{(d+1)\times(d+1)}$  
with  the following properties:
\begin{itemize}
\item[$\bullet$] $\mathbf{Y}$ is analytic in $\mathbb{C}\setminus\Sigma$.
\item[$\bullet$] For every $z\in\Sigma$, the function $\mathbf{Y}$ has boundary values 
$\mathbf{Y}_{+}(z)$ and $\mathbf{Y}_{-}(z)$ and they are related by $\mathbf{Y}_{+}(z)=\mathbf{Y}_{-}(z) \mathbf{J}_{Y}(z)$, where
    \begin{equation}\label{jump:Y}
    \mathbf{J}_{Y}(z)=\begin{pmatrix} 
	1 & v_{0,n}(z) & v_{1,n}(z) & \cdots & v_{d-1,n}(z)\\
  & 1 &   &   & \\
  &   & 1 &   & \\
  &   &   & \ddots  & \\
  &  & & & 1
\end{pmatrix}, \qquad z\in \Sigma. 
    \end{equation}
\item[$\bullet$] As $z\rightarrow\infty$,
\begin{equation}\label{asymp:Y}
\mathbf{Y}(z)=\Big(I+O\Big(\frac{1}{z}\Big)\Big)\,
	\diag \left( z^{n}, z^{-\frac{n}{d}}, z^{-\frac{n}{d}},  \ldots,  z^{-\frac{n}{d}} \right),
\end{equation}
where $I$ denotes the identity matrix of size $(d+1)\times(d+1)$.
\item[$\bullet$] As $z$ approaches one of the endpoints $\omega^{\ell}\,\widehat{x}$ of $\Sigma$,
\begin{equation} \label{endpoint:Y}
\mathbf{Y}(z)=\begin{pmatrix} 
O(1) & O(\log(z-\omega_{d+1}^{\ell}\, \widehat{x})) & \cdots & O(\log(z-\omega_{d+1}^{\ell}\, \widehat{x})) \\[0.3em]
O(1) & O(\log(z-\omega_{d+1}^{\ell}\, \widehat{x})) & \cdots  & O(\log(z-\omega_{d+1}^{\ell}\, \widehat{x})) \\[0.3em]
\vdots  & \vdots  &  & \vdots \\
O(1)  & O(\log(z-\omega_{d+1}^{\ell}\, \widehat{x})) & \cdots & O(\log(z-\omega_{d+1}^{\ell}\, \widehat{x}))
\end{pmatrix},
\end{equation}
and $\mathbf{Y}(z)=O(1)$ as $z\rightarrow 0$.
\end{itemize}
\end{rhp}

The relevance of this RH problem relies on the fact that it characterizes the multiple orthogonal polynomials $P_{n,n}$, 
as it was already indicated. This characterization result is now standard in the theory of RH problems associated with 
multiple orthogonal polynomials, and was first obtained in \cite{VAGK} in the context of multiple
orthogonality on the real line. It generalizes the RH problem for orthogonal polynomials \cite{Deift,FIK}.

\begin{lemma}\label{lemma:RHcharact}
The polynomial $P_{n,n}$ exists and is unique if and only if the RH problem \ref{RHPforY} is solvable. 
In this case we have $\mathbf{Y}_{1,1}=P_{n,n}$.
\end{lemma}

We obtain the large $n$ asymptotics of the RH problem \ref{RHPforY} by a steepest descent analysis.
The analysis is rather involved, since we are dealing with a RH problem of size $(d+1) \times (d+1)$,
where $d \geq 2$ can be any natural number. 
See \cite{DelKui} for an earlier example of a steepest descent analysis for RH problems of arbitrary size.

One of the outcomes
of the steepest descent analysis is that the RH problem \ref{RHPforY} has a solution for $n$ large enough,
and so the polynomials $P_{n,n}$ exist for $n$ large enough. The main outcome is the asymptotic formula
\eqref{strongasympform} for the polynomials $P_{n,n}$, as given in Theorem \ref{theo:strongasymp}.

The steepest descent analysis proceeds via a number of transformations.

The sectors $S_{\ell}$ and their subsectors $S_{\ell}^{\pm}$ were introduced in \eqref{sectorSl} and \eqref{sectorSlpSlm},
see also Figures \ref{fig:sectors}--\ref{fig:sectors4}. There are $d+1$ sectors $S_{\ell}$ and $\ell$ is
considered modulo $d+1$. In many cases it does not matter which value modulo $d+1$ we take, but
in some cases it does matter. The canonical choice is indicated in Figures \ref{fig:sectors}--\ref{fig:sectors4}.
Roughly speaking the index $\ell$ runs from $-\frac{d}{2}$ to $\frac{d}{2}$. To be precise,
we use the following convention
\begin{equation} \label{eq:Sellnumbering}
	\begin{aligned} 
		\text{for }  S_{\ell}^+ : & \quad
				\ell =  - \lceil \frac{d}{2} \rceil, \ldots, \lfloor \frac{d}{2} \rfloor  \\
     \text{for } S_{\ell}^-: & \quad 
			\ell = - \lfloor \frac{d}{2} \rfloor, \ldots, \lceil \frac{d}{2} \rceil.
			\end{aligned}
			\end{equation}

We use $\omega = \omega_{d+1} = e^{\frac{2\pi \ir}{d+1}}$ as before.

We also let $E_{j,k}$ be the elementary matrix with $1$ in
	position $(j,k)$ and $0$ elsewhere. The size of $E_{j,k}$ will be clear from
	the context. It is either size $d\times d$ or $(d+1) \times (d+1)$.

The case $d=2$ is essentially done in \cite{BleherKuij}. In this proof we 
assume $d \geq 3$. Then
\[ \mathbb C \setminus \bigcup_{\ell=0}^d (S_{\ell}^+ \cup S_{\ell}^-) = \Sigma_2 \cup \Sigma_3. \]
For a function that is analytic in $\bigcup_{\ell=0}^d (S_{\ell}^+ \cup S_{\ell}^-)$
we describe its boundary values separately on $\Sigma_2$ and $\Sigma_3$.

\section{First transformation $\mathbf{Y}\mapsto \mathbf{X}$}\label{section:firsttransf}

The first transformation is based on the construction of a $d \times d$-matrix valued
function $\mathbf{F} : \mathbb{C}\setminus(\Sigma_{2}\cup\Sigma_{3}) \to \mathbb C^{d\times d}$ out
of solutions of the ODE \eqref{eq:pODE}.
In each sector we pick a basis $f_1, \ldots, f_d$ of solutions of \eqref{eq:pODE} and
we define $\mathbf{F}$ as the Wronskian matrix
\begin{equation} \label{defF} 
	\mathbf{F} := \begin{pmatrix} f_1 & f_2 & \cdots & f_d \\
		f_1' & f_2' & \cdots & f_d' \\
		\vdots & \vdots &  & \vdots \\
		f_1^{(d-1)} & f_2^{(d-1)} & \cdots & f_d^{(d-1)} \end{pmatrix}.
		\end{equation}

\subsection{Definition and properties of $\mathbf{P}$}
We start from the functions $p_{\ell}$ from \eqref{def:pell}
that are special solutions of \eqref{eq:pODE}. We use cyclic notation here so that
\[ p_{\ell} = p_{d+1+ \ell }. \]
The solution $p_0$ is recessive  in the sector $S_0$.
Indeed a classical steepest descent analysis shows that
\begin{align} \label{asympp0}
	p_0(z) = \frac{1}{\sqrt{2\pi d}} z^{- \frac{d-1}{2d}}  e^{- \frac{d}{d+1} z^{\frac{d+1}{d}}} (1 + O(z^{- \frac{d+1}{d}}))
\end{align}
as $z \to \infty$ with $- \pi < \arg z < \pi$. 
Since
\[ p_{\ell}(z) = \omega^{\ell} p_0(\omega^{\ell} z), \qquad \omega = \omega_{d+1}, \]
we also find
\begin{align} \label{asymppl}
	p_{\ell}(z) = \frac{1}{\sqrt{2\pi d}} e^{\frac{\ell}{d} \pi \ir} z^{- \frac{d-1}{2d}}  
		e^{- \frac{d}{d+1} \omega_d^{\ell} z^{\frac{d+1}{d}}} (1 + O(z^{- \frac{d+1}{d}})),
				\quad \omega_d = e^{\frac{2\pi \ir}{d}},
\end{align}
as $z \to \infty$ with $- \frac{2\ell}{d+1}\pi   - \pi < \arg z < - \frac{2\ell}{d+1}\pi  + \pi$.
Thus $p_{-\ell}$ is recessive as $z \to \infty$ in sector $S_{\ell}$.

The asymptotic behaviors \eqref{asympp0}-\eqref{asymppl} show that in each sector $S_{\ell}^{\pm}$
all possible asymptotic behaviors occur among the functions $p_0, \ldots, p_d$.
We order them  according to their  absolute value as $z \to \infty$ in that sector. 
Ordering them from smallest to largest,  we obtain
\begin{equation}
\begin{aligned} \label{pellordering}
	\text{in } S_\ell^+:  & \quad p_{-\ell}, p_{-\ell-1}, p_{-\ell+1}, p_{-\ell-2}, \ldots, 
		p_{-\ell - (-1)^d  \lfloor \frac{d}{2} \rfloor},  p_{-\ell  + (-1)^d  \lceil \frac{d}{2} \rceil }\\
	\text{in } S_\ell^-:  & \quad p_{-\ell}, p_{-\ell+1}, p_{-\ell-1}, p_{-\ell+2}, \ldots, 
		p_{-\ell + (-1)^d  \lfloor \frac{d}{2} \rfloor}, p_{-\ell  - (-1)^d  \lceil \frac{d}{2} \rceil }.
\end{aligned}
\end{equation}
There are $d+1$ entries in each list, but for the construction that follows we do not use the last one.

\begin{definition}
We define $\mathbf{P} : \mathbb C \setminus (\Sigma_2 \cup \Sigma_3) \to \mathbb C^{d\times d} $ as the piecewise analytic
matrix valued function defined by
\begin{equation} \label{def:P} 
	\mathbf{P} = \begin{pmatrix} p_{-\ell} & p_{-\ell \mp 1} & p_{-\ell \pm 1} & p_{-\ell \mp 2} & \cdots \\  
	p'_{-\ell} & p'_{-\ell \mp 1} & p'_{-\ell \pm 1} & p'_{-\ell \mp 2} & \cdots \\
		\vdots & \vdots & \vdots & \vdots &  \\
		\vdots & \vdots & \vdots & \vdots & \\
		p^{(d-1)}_{-\ell} & p^{(d-1)}_{-\ell \mp 1} & p^{(d-1)}_{-\ell \pm 1} & p^{(d-1)}_{-\ell \mp 2} & \cdots 
	\end{pmatrix}_{d \times d}  \qquad \text{in } S_{\ell}^{\pm}. 
	\end{equation}
The ordering of the columns in \eqref{def:P} is according to \eqref{pellordering}. The last entries
in \eqref{pellordering} do not appear as a column in the matrix $\mathbf{P}$.
\end{definition}
Note that $\mathbf{P}$ has size $d \times d$ and so the last column of $\mathbf{P}$ is
\begin{equation} \label{P:last:d even} 
	\begin{pmatrix} p_{-\ell \mp \frac{d}{2}} & \cdots & p^{(d-1)}_{-\ell \mp \frac{d}{2}} \end{pmatrix}^t \quad \text{ if $d$ is even} 
	\end{equation}
and
\begin{equation} \label{P:last:d odd} 
	\begin{pmatrix} p_{-\ell \pm \frac{d-1}{2}} & \cdots & p^{(d-1)}_{-\ell \pm \frac{d-1}{2}} \end{pmatrix}^t \quad \text{ if $d$ is odd}.
	\end{equation}

Then $\mathbf{P}$ satisfies a RH problem that we state next.
\begin{rhp} \label{rhpforP}
$\mathbf{P} : \mathbb C \setminus (\Sigma_2 \cup \Sigma_3) \to \mathbb C^{d \times d}$ satisfies
\begin{itemize}
\item $\mathbf{P}$ is analytic in $\mathbb C \setminus (\Sigma_2 \cup \Sigma_3)$,
\item $\mathbf{P}_+ = \mathbf{P}_- \mathbf{J}_P$ with jump matrices 
\begin{align}  \label{jump:P:Sigma3}
	\mathbf{J}_P & = 
		\begin{cases} \diag(1, \sigma_1, \ldots, \sigma_1,0) - \sum_{j=1}^d E_{j,d} &  \text{ if $d$ is even,} \\
		\diag(1, \sigma_1, \ldots, \sigma_1) &  \text{ if $d$ is odd,} \end{cases} &&
		\text{ on } \Sigma_3, \\
		\label{jump:P:Sigma2}
	\mathbf{J}_P & = \begin{cases} \diag(\sigma_1, \ldots, \sigma_1) &  \quad \text{ if $d$ is even,} \\
		 \diag( \sigma_1, \ldots, \sigma_1,0) - \sum_{j=1}^d E_{j,d} & \quad  \text{ if $d$ is odd,} 
		\end{cases} && 		\text{ on } \Sigma_2, 
		\end{align}
where $\sigma_1 = \begin{pmatrix} 0 & 1 \\ 1 & 0 \end{pmatrix}$.
\item As $z \to \infty$ in $S_{\ell}^{\pm}$ we have
\begin{equation} \label{asymp:P} 
	\mathbf{P}(z) = \frac{1}{\sqrt{2\pi d}}  \left( I_d + O(z^{-\frac{2}{d}})\right)
	\mathbf{D}(z)^{-1} B_{\ell}^{\pm} \exp(-\mathbf{\Theta}^{\pm}( \omega^{-\ell} z) ) 
	\end{equation}
where 
\begin{align} \label{def:Dz} 
	\mathbf{D}(z) & = \diag\left(z^{\frac{d-1}{2d}}, z^{\frac{d-3}{2d}}, \ldots, z^{-\frac{d-1}{2d}} \right) \\
	\mathbf{\Theta}^{\pm}(z) & =  \label{def:Thetaz}
	\diag \left(1, \omega_d^{\mp 1}, \omega_d^{\pm 1}, \omega_d^{\mp 2}, \ldots \right) \frac{d}{d+1} z^{\frac{d+1}{d}},
	\end{align}
and constant matrices $B_{\ell}^{\pm}$ given by
\begin{align} \label{def:B0} 
	B_0^{\pm} = \diag(1, -1, 1, -1, \ldots) VDM(1, \omega_d^{\mp 1}, \omega_d^{\pm 1}, \omega_d^{\mp 2}, \ldots)
	\diag(1, \omega_{2d}^{\mp 1}, \omega_{2d}^{\pm 1}, \omega_{2d}^{\mp 2}, \ldots) 
	\end{align}
where $VDM$ denotes a Vandermonde matrix 
\[ VDM(x_1, \ldots, x_d) = \begin{pmatrix} x_k^{j-1} \end{pmatrix}_{j,k=1,\ldots, d}, \]
$\omega_d = e^{\frac{2\pi \ir}{d}}$, $\omega_{2d} = e^{\frac{\pi \ir}{d}}$, and 
\begin{align} \label{def:Bl}
	B_{\ell}^{\pm} = 
	\diag(\omega_{2d}^{-\ell}, \omega_{2d}^{-3\ell}, \ldots,  \omega_{2d}^{-(2d-1)\ell}) B_0^{\pm}. 
	\end{align}
	\end{itemize}
\end{rhp} 
\begin{proof}
The jump matrix of $\mathbf{P}$ on $\Sigma_3$ is $\mathbf{J}_P = \mathbf{P}_-^{-1} \mathbf{P}_+$ where $\mathbf{P}_-$ is the limit of $\mathbf{P}$ 
taken from $S_{\ell}^-$ and $\mathbf{P}_+$ is the limit from $S_{\ell}^+$ onto $\Sigma_3$, for some $\ell$. 
Then the jump matrix
\eqref{jump:P:Sigma3} is immediate from the definition \eqref{def:P} in case $d$ is odd.
If $d$ is even then the last column \eqref{P:last:d even} in $S_{\ell}^+$ 
is associated with $p_{-\ell - \frac{d}{2}}$ and this column does not appear
in $\mathbf{P}$ in $S_{\ell}^-$. Because of the relation \eqref{rel:p} the last column
in $S_{\ell}^+$ is equal to minus the sum of all columns in $S_{\ell}^-$,
and this accounts for the term $- \sum_{j=1}^d E_{j,d}$ in the jump matrix \eqref{jump:P:Sigma3}
in case $d$ is even.

The proof of \eqref{jump:P:Sigma2} is similar.

To prove \eqref{asymp:P} we use the asymptotic behavior \eqref{asymppl} of the functions $p_{\ell}$
and their derivatives, which can be obtained by taking the derivatives of the leading term in \eqref{asymppl}. 
Using this in \eqref{def:P} in the sector $S_0^{\pm}$ it then follows that
\begin{equation} \label{asymp:P0} 
	\mathbf{P}(z) = \frac{1}{\sqrt{2\pi d}} \mathbf{D}(z)^{-1} B_0^{\pm} \left( I_d + O(z^{-\frac{d+1}{d}})\right) \exp(-\mathbf{\Theta}^{\pm}(z)) 
	\end{equation}
as $z \to \infty$ in the sectors $S_0^{\pm}$, with $\mathbf{D}(z)$, $\mathbf{\Theta}^{\pm}(z)$ and $B_0^{\pm}$ as in \eqref{def:Dz}, 
\eqref{def:Thetaz}, \eqref{def:B0}.
Moving the $O$-term to the left we obtain \eqref{asymp:P} for $\ell=0$ since
\[ \mathbf{D}(z)^{-1} O(z^{-\frac{d+1}{d}}) = O(z^{-\frac{2}{d}}) \mathbf{D}(z)^{-1} \]
as $\mathbf{D}(z) = O(z^\frac{d-1}{2d})$ and $\mathbf{D}(z)^{-1} = O(z^\frac{d-1}{2d})$ as $z \to \infty$.

We finally use the identity
\[ \mathbf{P}(z) = \diag(\omega, \omega^2, \ldots, \omega^{d}) \mathbf{P}(\omega z), \qquad \omega = \omega_{d+1} \]
to obtain the asymptotic behavior \eqref{asymp:P} with matrices $B_{\ell}^{\pm}$ as in \eqref{def:Bl} 
in the other sectors.
\end{proof}

The jump matrices in \eqref{jump:P:Sigma3} and \eqref{jump:P:Sigma2} have a block diagonal 
form, except partially for the last column. The term $- \sum_{j=1}^d E_{j,d}$ means that the last column
is filled with $-1$'s.

\begin{lemma} \label{jump:Bz}
Let 
\begin{equation} \label{def:Bz}
	\mathbf{B}(z) = \mathbf{D}(z)^{-1} B_{\ell}^{\pm}, \qquad \text{ for } z \in S_{\ell}^{\pm}.
	\end{equation}
where $\mathbf{D}(z)$ is as in \eqref{def:Dz} and the constant matrices $B_{\ell}^{\pm}$
are given in \eqref{def:B0}--\eqref{def:Bl}.
	Then $\mathbf{B}$ is analytic in
$\mathbb C \setminus (\Sigma_2 \cup \Sigma_3)$ and $\mathbf{B}_+ = \mathbf{B}_- \mathbf{J}_B$ with
\begin{align} \label{jump:B:Sigma3} 
	\mathbf{J}_B & = \begin{cases} 
	\diag(1, \sigma_1, \ldots, \sigma_1, -1) & \text{ if $d$ is even}  \\
	\diag(1, \sigma_1,  \ldots, \sigma_1) & \text{ if $d$ is odd} 
	\end{cases} && \text{ on } \Sigma_3, \\
	\mathbf{J}_B & = \begin{cases}
	\diag(\sigma_1, \sigma_1, \ldots, \sigma_1) & \text{ if $d$ is even} \\
	\diag(\sigma_1, \sigma_1, \ldots, \sigma_1, -1) & \text{ if $d$ is odd} \\
	\end{cases} && \text{ on } \Sigma_2. 
	\label{jump:B:Sigma2}
	\end{align}
\end{lemma}
\begin{proof}
It is a straightforward calculation from \eqref{def:B0} to show that
\begin{equation} \label{jump:B0pm}
	(\mathbf{B}_0^-)^{-1} \mathbf{B}_0^+ = \begin{cases} \diag(1, \sigma_1, \ldots, \sigma_1, -1) & \text{ if $d$ is even} \\
		\diag(1, \sigma_1, \ldots, \sigma_1) & \text{ if $d$ is odd}
		\end{cases}
	\end{equation}
and using also \eqref{def:Bl} that
\begin{equation} \label{jump:B01} 
	(\mathbf{B}_0^+)^{-1} \mathbf{B}_1^- = \begin{cases}
	\diag(\sigma_1, \ldots, \sigma_1) & \text{ if $d$ is even} \\
	\diag(\sigma_1, \ldots, \sigma_1, -1) & \text{ if $d$ is odd}.
	\end{cases} 
	\end{equation}
It is then easy to see from the definitions \eqref{def:Bl} and \eqref{def:Bz} that we 
immediately obtain the jump matrices \eqref{jump:B:Sigma3} on $\Sigma_3 \setminus \mathbb R_-$
and \eqref{jump:B:Sigma2} on $\Sigma_2 \setminus \mathbb R_-$.

To compute the jump on the negative real axis (which is oriented from right to left) 
we note that 
\begin{align} \nonumber
\mathbf{J}_B & = \mathbf{B}_-^{-1}(z) \mathbf{B}_+(z) \\ \label{jump:B:Rmin}
		 & = \begin{cases} 
			\left(B_{\frac{d}{2}}^+\right)^{-1} \mathbf{D}_-(z) \mathbf{D}_+(z)^{-1}  B_{-\frac{d}{2}}^-  & \text{ if $d$ is even} \\
			\left(B_{\frac{d+1}{2}}^- \right)^{-1} \mathbf{D}_-(z) \mathbf{D}_+(z)^{-1}  B_{-\frac{d+1}{2}}^+  
			& \text{ if $d$ is odd }
			\end{cases}  \qquad z \in \mathbb R_-
		\end{align}
where due to \eqref{def:Dz} and the choice of principle branches of the fractional powers
\begin{equation} \label{jump:D} 
	\mathbf{D}_-(z) \mathbf{D}_+(z)^{-1} = \diag( \omega_{2d}^{d-1}, \omega_{2d}^{d-3}, \ldots, \omega_{2d}^{-(d-1)} ), \qquad z \in \mathbb R_-.
	\end{equation}

Using \eqref{def:Bl} and \eqref{jump:D} in \eqref{jump:B:Rmin} we find
\begin{align*}
\mathbf{J}_B & = (B_0^+)^{-1} \diag(\omega_{2d}^{\frac{d}{2}}, \omega_{2d}^{3\frac{d}{2}}, \ldots)  
				\diag( \omega_{2d}^{d-1}, \omega_{2d}^{d-3}, \ldots, \omega_{2d}^{-(d-1)} )   
			\diag(\omega_{2d}^{\frac{d}{2}}, \omega_{2d}^{3\frac{d}{2}}, \ldots)	B_0^- \\
			 & = (B_0^+)^{-1} \diag(\omega_{2d}^{-1}, \omega_{2d}^{-3}, \ldots) B_0^- \\
			&	= (B_0^+)^{-1} B_1^- 
			\end{align*}
if $d$ is even, while 
\begin{align*}
	\mathbf{J}_B & =			(B_0^-)^{-1} \diag(\omega_{2d}^{\frac{d+1}{2}}, \omega_{2d}^{3\frac{d+1}{2}}, \ldots) 
			 \diag( \omega_{2d}^{d-1}, \omega_{2d}^{d-3}, \ldots, \omega_{2d}^{-(d-1)} )   
			 \diag(\omega_{2d}^{\frac{d+1}{2}}, \omega_{2d}^{3\frac{d+1}{2}}, \ldots)   
			 B_0^+ \\
			& =  (B_0^-)^{-1} B_0^+ 
			\end{align*}
if $d$ is odd.			
In view of \eqref{jump:B0pm} and \eqref{jump:B01} we then have
\[ \mathbf{J}_B = \begin{cases}
				\diag(\sigma_1, \sigma_1, \ldots, \sigma_1) & \text{ if $d$ is even} \\
				\diag(1, \sigma_1, \sigma_1, \ldots, \sigma_1) & \text{ if $d$ is odd}
				\end{cases} \qquad \text{ on } \mathbb R_-. \]
Since $\mathbb R_- \subset \Sigma_2$ if $d$ is even, and $\mathbb R_- \subset \Sigma_3$ if
$d$ is odd, we verified the jump matrix \eqref{jump:B:Sigma3} and \eqref{jump:B:Sigma2} also
for the part $\mathbb R_-$.
\end{proof}

\subsection{Definition and properties of $\mathbf{F}$}
	
The jumps \eqref{jump:P:Sigma3}--\eqref{jump:P:Sigma2} are not convenient for us. 
We modify $\mathbf{P}$ by multiplying it on the right by certain constant matrices $R^{\pm}$
that we define first.
We use the elementary symmetric functions on $n$ variables:
\begin{align} 
	e_{0,n}(x_1, \ldots, x_n) & = 1 \label{elementary:1} \\
	e_{k,n}(x_1, \ldots, x_n) & = \sum_{1 \leq j_1 < j_2 < \cdots < j_k \leq n} x_{j_1} \cdots x_{j_k}, \qquad 1 \leq k \leq n,\label{elementary:2}  
	\end{align}
and define numbers
\begin{equation} \label{ekn} 
	e_{k,n}^+   := e_{k,n}(\omega, \omega^{-1}, \omega^2, \omega^{-2}, \omega^3, \ldots),  \quad
  e_{k,n}^- := e_{k,n}(\omega^{-1}, \omega, \omega^{-2}, \omega^2, \omega^{-3},\ldots), \qquad \omega = \omega_{d+1},
	\end{equation}
with the understanding that $e^+_{0,0} = e^-_{0,0} = 1$ and $e^+_{k,n} = e^-_{k,n} = 0$ for $k < 0$ or $k > n$.
Then we put for $j,k = 1, \ldots, d$,
\begin{align} \label{def:Rplus}
	(R^+)_{j,k} & :=  (-1)^{1+ \lfloor \frac{j}{2} \rfloor + \lceil \frac{k}{2} \rceil} \times
	\begin{cases}
		e^+_{\lfloor \frac{k}{2} \rfloor + \frac{j-1}{2}, k-1} & \text{ if $j$ is odd} \\
		e^+_{\lfloor \frac{k}{2} \rfloor - \frac{j}{2}, k-1} & \text{ if $j$ is even} 
		\end{cases} \\ \label{def:Rminus}
	(R^-)_{j,k} & := (-1)^{\lfloor \frac{j}{2} \rfloor + \lfloor \frac{k}{2} \rfloor} \times 
\begin{cases}
		e^-_{\lfloor \frac{k}{2} \rfloor + \frac{j-1}{2}, k-1} & \text{ if $j$ is odd} \\
		e^-_{\lfloor \frac{k}{2} \rfloor - \frac{j}{2}, k-1} & \text{ if $j$ is even.} 
		\end{cases} 
		\end{align}
		
For example, for $d=4$ we have $\omega = \omega_5 = \exp(2\pi \ir/5)$ and
\begin{align*}  
	R^+ & = \begin{pmatrix} 1 & e_{1,1}(\omega) & - e_{1,2}(\omega,\omega^{-1}) & -e_{2,3}(\omega, \omega^{-1}, \omega^2) \\
	0 & -1 & 1 & e_{1,3}(\omega, \omega^{-1},\omega^2) \\
	0 & 0 & e_{2,2}(\omega, \omega^{-1}) & e_{3,3}(\omega,\omega^{-1}, \omega^2) \\
	0 & 0 & 0 & -1 \end{pmatrix} \\
	R^- & = \begin{pmatrix} 1 & -e_{1,1}(\omega^{-1}) & - e_{1,2}(\omega^{-1},\omega) & e_{2,3}(\omega^{-1}, \omega, \omega^{-2}) \\
	0 & 1 & 1 & -e_{1,3}(\omega^{-1}, \omega,\omega^{-2}) \\
	0 & 0 & e_{2,2}(\omega^{-1}, \omega) & -e_{3,3}(\omega^{-1},\omega, \omega^{-2}) \\
	0 & 0 & 0 & 1 \end{pmatrix}. 
	\end{align*}
It can be readily verified that $R^{\pm}$ are upper triangular matrices with diagonal
entries 
\begin{equation} \label{Rdiagonal} 
	(R^+)_{k,k} = (-1)^{k-1}, \qquad (R^-)_{k,k} = 1.
	\end{equation}

\begin{definition}
We define $\mathbf{F}: \mathbb C \setminus (\Sigma_2 \cup \Sigma_3) \to \mathbb C^{d \times d}$ by
\begin{equation} \label{def:F} 
	\mathbf{F} = \mathbf{P}  R^{\pm}  \qquad \text{in } S_{\ell}^{\pm}
	\end{equation}
where $\mathbf{P}$ is defined in \eqref{def:P} and $R^{\pm}$ are the right 
upper triangular matrices defined above.
\end{definition}
Then $\mathbf{F}$ has the form \eqref{defF} with functions $f_1, \ldots f_d$ 
where  $f_k$ in $S_{\ell}^{\pm}$ is a linear combination of
the first $k$ functions in 
$p_{-\ell}, p_{-\ell \mp 1}, p_{-\ell \pm 1}, \ldots$.
It turns out that $\mathbf{F}$ is the solution of a RH problem with
lower triangular jump matrices with a simple $2 \times 2$ block structure.
\begin{rhp} \label{RHPforF}
$\mathbf{F}$ is the solution of the following  RH problem:
\begin{enumerate}
\item[$\bullet$] $\mathbf{F}$ is analytic in $\mathbb C \setminus (\Sigma_2 \cup \Sigma_3)$.
\item[$\bullet$] $\mathbf{F}_+ = \mathbf{F}_- \mathbf{J}_F$ on $\Sigma_2 \cup \Sigma_3$
where
\begin{align} \label{JFonSigma3}
	 \mathbf{J}_F & = I - \sum_{j=1}^{\lfloor \frac{d-1}{2} \rfloor} E_{2j+1, 2j} && \text{ on } \Sigma_3, \\
	\mathbf{J}_F & = \diag(\omega, \omega^{-1}, \omega^2, \omega^{-2}, \ldots ) - \sum_{j=1}^{\lfloor \frac{d}{2} \rfloor} E_{2j, 2j-1}
		&& \text{ on } \Sigma_2.  \label{JFonSigma2}
		\end{align}
\item[$\bullet$] 
$\mathbf{F}$ has the following asymptotic behavior as $z \to \infty$,
\begin{multline} \label{asymp:F} 
	\mathbf{F}(z) = \frac{1}{\sqrt{2\pi d}} \left( I_d + O(z^{-\frac{2}{d}}) \right) \mathbf{D}(z)^{-1} \\
	\times \begin{cases} B_{\ell}^{+} \diag(1, -1, 1, -1, \ldots ) \exp(- \mathbf{\Theta}^{+}(\omega^{-\ell} z)), & 
		\text{in } S_{\ell}^+ \text{ with } - \lceil \frac{d}{2} \rceil \leq \ell \leq \lfloor \frac{d}{2} \rfloor, \\
		B_{\ell}^{-} \exp(- \mathbf{\Theta}^{-}(\omega^{-\ell} z)), & 
		\text{in } S_{\ell}^- \text{ with } - \lfloor \frac{d}{2} \rfloor \leq \ell \leq \lceil \frac{d}{2} \rceil,
		\end{cases}
		\end{multline}
with $\mathbf{D}(z)$, $\mathbf{\Theta}^{\pm}(z)$ and $B_{\ell}^{\pm}$ as defined in \eqref{def:Dz}, \eqref{def:Thetaz}
and \eqref{def:B0}--\eqref{def:Bl}.
\end{enumerate}
\end{rhp}
\begin{proof}
The jump matrix is given by $\mathbf{J}_F = (R^-)^{-1} \mathbf{J}_P R^+$ on $ \Sigma_3$ and $\mathbf{J}_F = (R^+)^{-1} \mathbf{J}_P R^-$ 
on $ \Sigma_2$.
Here we insert $\mathbf{J}_P$ as given by \eqref{jump:P:Sigma3}--\eqref{jump:P:Sigma2} and the definitions
\eqref{def:Rplus} and \eqref{def:Rminus} of $R^+$ and $R^-$. Then after straightforward 
calculations we find that this indeed leads to \eqref{JFonSigma3}--\eqref{JFonSigma2}.

The asymptotic condition comes from \eqref{asymp:P} and the fact that $f_k$ is a linear combination of the first
$k$ of the $p_j$ functions. In each sector these are ordered according to their asymptotic behavior at infinity
in that sector, see \eqref{pellordering}. Thus the asymptotic behavior of $f_k$ is given by that of the $k$th member times
a constant coming from the diagonal entry of $R^+$ or $R^-$. The diagonal entries are $\pm 1$ according 
to \eqref{Rdiagonal} and this leads to \eqref{asymp:F} because of \eqref{asymp:P} and \eqref{Rdiagonal}.
\end{proof}

\subsection{Definition and properties of $\mathbf{Q}$}

Let us introduce the matrix $\mathbf{Q}$ that is 
employed in the first transformation of the RH problem \ref{RHPforY}.

\begin{definition} We define $\mathbf{Q}$ as
\begin{equation}\label{def:Q}
\mathbf{Q}(z):= \mathbf{F}(z)^{-t},
\end{equation}
that is, the inverse transpose of \eqref{def:F}.
\end{definition}

From the RH problem \ref{RHPforF} we obtain the following RH problem for $\mathbf{Q}$.
\begin{rhp} \label{RHPforQ}
$\mathbf{Q}$ is the solution of the following  RH problem:
\begin{enumerate}
\item[$\bullet$] $\mathbf{Q}$ is analytic in $\mathbb C \setminus (\Sigma_2 \cup \Sigma_3)$. 
\item[$\bullet$] $\mathbf{Q}_+ = \mathbf{Q}_- \mathbf{J}_Q$ on $\Sigma_2 \cup \Sigma_3$ where
\begin{align} \label{JQonSigma3}
	 \mathbf{J}_Q & = I_d + \sum_{j=1}^{\lfloor \frac{d-1}{2} \rfloor} E_{2j, 2j+1} && \text{ on } \Sigma_3, \\
	\mathbf{J}_Q & = \diag(\omega^{-1}, \omega, \omega^{-2}, \omega^{2}, \cdots ) + 
		\sum_{j=1}^{\lfloor \frac{d}{2} \rfloor} E_{2j-1, 2j}
		&& \text{ on } \Sigma_2  \label{JQonSigma2}.
		\end{align}
\item[$\bullet$] 
$\mathbf{Q}$ has the asymptotic behavior
\begin{align} \label{asymp:Q} 
	\mathbf{Q}(z) = \sqrt{2\pi d} \left( I_d + O(z^{-\frac{2}{d}}) \right) \mathbf{D}(z) A_{\ell}^{\pm} \exp(\mathbf{\Theta}^{\pm}(\omega^{-\ell} z))
	\end{align}
as $z \to \infty$ in $S_{\ell}^{\pm}$,
where
\begin{equation} \label{def:Aell}
\begin{aligned} 
		A_{\ell}^+ & = (B_{\ell}^{+})^{-t} \diag(1, -1, 1, -1, \ldots ), \\
		A_{\ell}^- & = (B_{\ell}^-)^{-t}. 
		\end{aligned}
		\end{equation}
\end{enumerate}
\end{rhp}
Note that the jump matrices for $\mathbf{Q}$ are upper triangular with a $2 \times 2$ block diagonal
structure.

For later use we define
\begin{equation} \label{def:Az} 
	\mathbf{A}(z) = \mathbf{D}(z) A_{\ell}^{\pm} \qquad \text{ for } z \in S_{\ell}^{\pm}. 
	\end{equation}
The following is immediate from Lemma \ref{jump:Bz} and \eqref{def:Aell}-\eqref{def:Az}.
\begin{lemma} \label{jump:Az}
$\mathbf{A}$ is analytic in $\mathbb C \setminus (\Sigma_2 \cup \Sigma_3)$  and $\mathbf{A}_+ = \mathbf{A}_- \mathbf{J}_A$ with
\begin{align} \label{jump:A:Sigma3} 
	\mathbf{J}_A & = \begin{cases} 
	\diag(1, \ir\sigma_2, \ldots, \ir\sigma_2, 1) & \text{ if $d$ is even}  \\
	\diag(1, \ir\sigma_2,  \ldots, \ir\sigma_2) & \text{ if $d$ is odd} 
	\end{cases} && \text{ on } \Sigma_3, \\
	\mathbf{J}_A & = \begin{cases}
	\diag(\ir\sigma_2, \ir\sigma_2, \ldots, \ir\sigma_2) & \text{ if $d$ is even} \\
	\diag(\ir\sigma_2, \ir\sigma_2, \ldots, \ir\sigma_2, -1) & \text{ if $d$ is odd} \\
	\end{cases} && \text{ on } \Sigma_2. 
	\label{jump:A:Sigma2}
	\end{align}
where  $\ir \sigma_2 = \begin{pmatrix} 0 & 1 \\ -1 & 0 \end{pmatrix}$.
\end{lemma}
\begin{proof}
We have $\mathbf{J}_A = (\mathbf{J}_B)^{-t} \diag(1,-1,1,-1, \ldots)$  on $\Sigma_3$ and 
$\mathbf{J}_A = \diag(1,-1,1, -1, \ldots) (\mathbf{J}_B)^{-t}$ on $\Sigma_2$. Then use \eqref{jump:B:Sigma3} and \eqref{jump:B:Sigma2}.
\end{proof}

\subsection{Transformation $\mathbf{Y} \mapsto \mathbf{X}$}

\begin{definition}\label{def:X} Let $\Sigma_X = \Sigma_2 \cup \Sigma_3$ and 
define 
$\mathbf{X} : \mathbb C \setminus \Sigma_X \to \mathbb C^{(d+1) \times (d+1)}$ by 
\begin{equation}\label{eq:def:X}
\mathbf{X}(z):=\begin{pmatrix} 
1 & 0 \\
0  & \frac{1}{\sqrt{2\pi d}}\,\mathbf{D}(c_{n})^{-1}
\end{pmatrix} \mathbf{Y}(z)\begin{pmatrix} 
1 & 0 \\
0  & \mathbf{Q}(c_{n}z)
\end{pmatrix}
\end{equation}
where $c_{n}$ is given in \eqref{def:dn:cn}, $\mathbf{D}$ is the diagonal matrix \eqref{def:Dz},
and $\mathbf{Q}$ is the matrix \eqref{def:Q}.
\end{definition}

Note that $\mathbf{D}$ and $\mathbf{Q}$ are matrices of size $d \times d$, and \eqref{eq:def:X}
is written in a block form with diagonal blocks of size $1 \times 1$ and $d \times d$.

\begin{rhp}\label{RHPforX}
The matrix $\mathbf{X}(z)$ is the solution to the following RH problem:
\begin{itemize}
\item[$\bullet$] $\mathbf{X}$ is analytic in $\mathbb{C}\setminus \Sigma_X$.
\item[$\bullet$] $\mathbf{X}_{+}=\mathbf{X}_{-} \mathbf{J}_{X}$ on $\Sigma_{2}\cup \Sigma_{3}$ with jump matrix 
\begin{align} \label{jump:X:Sigma1}
	\mathbf{J}_{X} &= I_{d+1} + e^{\frac{n}{t_0} V} E_{1,2} + \sum_{j=2}^{\lceil \frac{d}{2} \rceil} E_{2j-1,2j} && \text{ on } \Sigma_1 \\
		\label{jump:X:Sigma3}
	\mathbf{J}_X & = I_{d+1} + \sum_{j=2}^{\lceil \frac{d}{2} \rceil} E_{2j-1,2j} && \text{ on } \Sigma_3 \setminus \Sigma_1 \\
	\label{jump:X:Sigma2}
	\mathbf{J}_X & = \diag(1, \omega^{-1}, \omega, \omega^{-2}, \omega^2, \ldots ) + \sum_{j=1}^{\lfloor \frac{d}{2} \rfloor} E_{2j,2j+1} 
		&& \text{ on } \Sigma_2.
		\end{align} 
\item[$\bullet$] As $z$ tends to infinity in the sector $S_{\ell}^{\pm}$, 
\begin{equation}\label{asymp:X}
\mathbf{X}(z)=\Big(I_{d+1}+O\left(z^{-\frac{2}{d}} \right)\Big) \begin{pmatrix} 1 & 0 \\ 0 & \mathbf{A}(z) \end{pmatrix} 
  \begin{pmatrix} z^n & 0 \\ 0 & z^{-\frac{n}{d}} \exp(\mathbf{\Theta}^{\pm}(\omega^{-\ell} c_n z)) \end{pmatrix},
\end{equation}
where $\mathbf{A}(z)$ is given by \eqref{def:Az}.
\item[$\bullet$] $\mathbf{X}$ satisfies the same endpoint condition \eqref{endpoint:Y} as $\mathbf{Y}$.
\end{itemize}
\end{rhp}
\begin{proof}
For $z\in\Sigma_{1}$, using \eqref{jump:Y} and \eqref{eq:def:X} we find after simple computations that
\begin{equation} \label{Xjump1}
\mathbf{X}_{-}^{-1}(z)\, \mathbf{X}_{+}(z)= \begin{pmatrix} 1 &  
	\begin{pmatrix} v_{0,n}(z) & \cdots & v_{d-1,n}(z) \end{pmatrix} \mathbf{Q}_{+}(c_{n}z) \\ 
	0  & \mathbf{Q}_{-}^{-1}(c_{n}z)\,\mathbf{Q}_{+}(c_{n}z)
\end{pmatrix}.
\end{equation}
By the jump \eqref{JQonSigma3} of $\mathbf{Q}$ we have that the right lower block is 
\[ \mathbf{Q}_{-}^{-1}(c_{n}z)\,\mathbf{Q}_{+}(c_{n}z)=\mathbf{J}_{Q} =  I_d + \sum_{j=1}^{\lfloor \frac{d-1}{2} \rfloor} E_{2j,2j+1}. \] 

From the definitions \eqref{def:F} and \eqref{def:P}, and the fact that $R^{\pm}$ is upper triangular with $(1,1)$ entry equal to $1$,
see \eqref{Rdiagonal}, we obtain that
\[ \mathbf{F}(z)  \begin{pmatrix} 1 \\ 0 \\ \vdots \\ 0 \end{pmatrix} = \mathbf{P}(z) \begin{pmatrix} 1 \\ 0 \\ \vdots \\ 0 \end{pmatrix} 
	= \begin{pmatrix} p_{-\ell}(z) \\ p_{-\ell}'(z) \\ \vdots \\ p_{-\ell}^{(d-1)}(z) \end{pmatrix},
	\qquad z \in S_{\ell}. \]
Then by \eqref{eq:def:vjn}
\[ \mathbf{F}_+(c_n z) \begin{pmatrix} 1 \\ 0 \\ \vdots \\ 0 \end{pmatrix} = 
	e^{-\frac{n}{t_0} V(z)} \begin{pmatrix} v_{0,n}(z) \\ v_{1,n}(z) \\ \vdots \\ v_{d-1,n}(z) \end{pmatrix},
	\qquad z \in \Sigma_1, \]
and therefore, since $\mathbf{Q} =  \mathbf{F}^{-t}$, see \eqref{def:Q},
\begin{align*}
	\begin{pmatrix} v_{0,n}(z) &  \cdots &  v_{d-1,n}(z) \end{pmatrix} 
	\mathbf{Q}_+(c_n z)  = e^{\frac{n}{t_0} V(z)} \begin{pmatrix} 1 & 0 & \cdots & 0 \end{pmatrix},
\end{align*}
which gives the vector in the first row of \eqref{Xjump1}. This proves \eqref{jump:X:Sigma1}.

The computation of $\mathbf{J}_X$ on  $\Sigma_{3}\setminus\Sigma_{1}$ and $\Sigma_2$ is
straightforward, since $\mathbf{Y}$ is analytic on these sets, and we just get the jump $\mathbf{J}_X = \diag(1, \mathbf{J}_Q)$ with
$\mathbf{J}_Q$ given by \eqref{JQonSigma3}--\eqref{JQonSigma2}.

The asymptotic condition \eqref{asymp:X} comes from combining the asymptotic conditons \eqref{asymp:Y}
and \eqref{asymp:Q} in the RH problems for $\mathbf{Y}$ and $\mathbf{Q}$. 

Finally, it is clear that $\mathbf{X}$ satisfies the same endpoint conditions as $\mathbf{Y}$.
\end{proof}

\section{Second transformation $\mathbf{X}\mapsto \mathbf{U}$}

The goal of the next transformation is to simplify the asymptotic condition \eqref{asymp:X}. 
\begin{definition}
We define $\mathbf{U}(z)$ by 
\begin{equation}\label{def:U}
\mathbf{U}(z):=\mathbf{X}(z) \begin{pmatrix} 
1 & 0 \\
0 & \exp(- \mathbf{\Theta}^{\pm}(\omega^{-\ell} c_{n}z)) \end{pmatrix}, \qquad z\in S_{\ell}^{\pm}.  
\end{equation}
\end{definition}

In order to formulate the jump conditions for $\mathbf{U}$ it is  convenient to introduce here 
the following functions:
\begin{equation}\label{def:lambdafunc}
\lambda_{k,n}(z)=\exp\Big(\frac{2\, \ir\, n d}{(d+1)\,t_{0}\, t_{d+1}^{1/d}}
\sin\Big(\frac{\pi k}{d}\Big)|z|^{\frac{d+1}{d}}\Big),\qquad k\in\mathbb{N}.
\end{equation}
We will use $\lambda_{2k,n}$ on $\Sigma_{3}$ and $\lambda_{2k-1,n}$ on $\Sigma_2$. 
Note that $|\lambda_{k,n}(z)|=1$ for all $z\in\mathbb{C}$, but when considered on $\Sigma_2$
or $\Sigma_3$ they allow for analytic continuation into a neighborhood.

\begin{rhp}\label{RHPforU}
The matrix valued function $\mathbf{U}(z)$ is the solution to the following RH problem:
\begin{itemize}
\item[$\bullet$] $\mathbf{U}$ is analytic in $\mathbb{C}\setminus \Sigma_U$ where $\Sigma_U = \Sigma_X = \Sigma_{2}\cup\Sigma_{3}$.
\item[$\bullet$] $\mathbf{U}_{+}=\mathbf{U}_{-} \mathbf{J}_{U}$ on $\Sigma_U$ with jump matrix 
\begin{equation} \label{jump:U:Sigma1}
\mathbf{J}_{U} = \diag(1,1, \lambda_{2,n}, \lambda_{2,n}^{-1}, \lambda_{4,n},\ldots)  + e^{-\frac{n}{t_0} V_1} E_{1,2}
	+ \sum_{j=2}^{\lceil \frac{d}{2} \rceil} E_{2j-1,2j} \quad \text{ on } \Sigma_1 
	\end{equation}
where
\begin{equation} \label{def:V1}
V_1(z) = \frac{d}{(d+1) t_{d+1}^{1/d}} |z|^{\frac{d+1}{d}} - \frac{t_{d+1}}{d+1} z^{d+1},
\end{equation}
and
\begin{align} \label{jump:U:Sigma3}
	\mathbf{J}_U & = \diag(1,1, \lambda_{2,n}, \lambda_{2,n}^{-1}, \lambda_{4,n}, \ldots)  + 
	\sum_{j=2}^{\lceil \frac{d}{2} \rceil} E_{2j-1,2j} && \text{ on } \Sigma_3 \setminus \Sigma_1 \\
	\mathbf{J}_U & = \diag(1, \omega^{-1} \lambda_{1,n}, \omega \lambda_{1,n}^{-1}, \omega^{-2} \lambda_{3,n}, \ldots ) 
	+ \sum_{j=1}^{\lfloor \frac{d}{2}\rfloor} E_{2j,2j+1} && \text{ on } \Sigma_2. \label{jump:U:Sigma2}
	\end{align}
\item[$\bullet$] As $z \to \infty$,
\begin{equation}\label{asymp:U}
\mathbf{U}(z)=\Big(I_{d+1}+O\left(z^{-\frac{2}{d}}\right) \Big)\begin{pmatrix} 1 & 0 \\ 0 & \mathbf{A}(z) \end{pmatrix} 
	\begin{pmatrix} z^n & 0 \\ 0 & z^{-\frac{n}{d}} I_d \end{pmatrix}.
\end{equation}
\item[$\bullet$] $\mathbf{U}$ satisfies the same endpoint condition \eqref{endpoint:Y} as $\mathbf{Y}$.  
\end{itemize}
\end{rhp}
\begin{remark}
The jump matrices \eqref{jump:U:Sigma1}, \eqref{jump:U:Sigma3} and \eqref{jump:U:Sigma2} have size $(d+1) \times (d+1)$.
It means that the pattern on the diagonal repeats until we reach size $(d+1) \times (d+1)$. If $d$ is even then
the diagonals in \eqref{jump:U:Sigma1} and \eqref{jump:U:Sigma3} end with
$\lambda_{d,n} \equiv 1$, see \eqref{def:lambdafunc}. If $d$ is odd, then the diagonal in \eqref{jump:U:Sigma2} ends with
$\omega^{-(d+1)/2} \lambda_{d,n} \equiv -1$. 
\end{remark}

\begin{proof}
Since $\exp(-\mathbf{\Theta}^{\pm}(\omega^{-\ell} c_nz))$ is a diagonal matrix it follows 
from \eqref{def:U} that the non-zero entries in $\mathbf{J}_U$
are in the same positions as those of $\mathbf{J}_X$.

Let $z \in \Sigma_3$ with $z \in S_{\ell}$. Then $(\omega^{-\ell} z)^{\frac{d+1}{d}} = |z|^{\frac{d+1}{d}}$,
which means that by \eqref{def:Thetaz} and \eqref{def:dn:cn}
\begin{equation} \label{ThetaonSigma3} 
	\exp(- \mathbf{\Theta}^{\pm}(\omega^{-\ell} c_n z)) = 
		\exp\left(-\frac{ n d}{(d+1) t_0 t_{d+1}^{1/d}}  \diag(1, \omega_d^{\mp 1}, \omega_d^{\pm 1},\omega_{d}^{\mp 2}, \ldots)
	 |z|^{\frac{d+1}{d}}\right). 
	\end{equation}
It then follows from \eqref{def:U} that the $(1,2)$ entry in the jump
matrix \eqref{jump:X:Sigma1} on $\Sigma_1$ is multiplied by the first entry of 
\eqref{ThetaonSigma3} which by  \eqref{def:V1} results
in
\[ (\mathbf{J}_U(z))_{1,2} = e^{\frac{n}{t_0}V(z)}  e^{-\frac{nd}{(d+1) t_0 t_{d+1}^{1/d}} |z|^{\frac{d+1}{d}}}
	= e^{- \frac{n}{t_0} V_1(z)},
	\qquad z \in \Sigma_1. \]

For $j \geq 2$, and $z \in \Sigma_3 \cap S_{\ell}$, we have that the $(2j-1,2j)$ entry of $\mathbf{J}_X$ is multiplied by the $2j-1$ entry of 
$\exp(-\mathbf{\Theta}^+(\omega^{-\ell} c_nz))$ and by the inverse of the $2j-2$ entry of 
$\exp(-\mathbf{\Theta}^-(\omega^{-\ell} c_nz))$. 
These factors cancel out because
of the way the exponents $\omega_d^{\mp j}, \omega_d^{\pm j}$ of $\omega_d$ appear in \eqref{def:Thetaz}. Thus
\[ (\mathbf{J}_U(z))_{2j-1,2j} = (\mathbf{J}_X(z))_{2j-1,2j} = 1, \qquad z \in \Sigma_3, \]
see \eqref{jump:X:Sigma1}--\eqref{jump:X:Sigma3}.

The diagonal entry $(2j+1,2j+1)$ of $\mathbf{J}_X$ is multiplied by the $2j$ entry of $\exp(-\mathbf{\Theta}^+(\omega^{-\ell}c_nz))$
and by the inverse of the $2j$ entry of $\exp(-\mathbf{\Theta}^-(\omega^{\ell}c_nz))$.  This leads to a combination
\[ - \omega_d^{-j} + \omega_d^{j} = 2\ir \sin \frac{2 \pi j}{d}  \]
in the exponentials. Then by \eqref{jump:X:Sigma1}--\eqref{jump:X:Sigma3} and the definition \eqref{def:lambdafunc} we get 
\[ (\mathbf{J}_U(z))_{2j+1,2j+1} = \lambda_{2j,n}(z), \qquad  z \in \Sigma_3,  \]
and similarly
\[ (\mathbf{J}_U(z))_{2j+2,2j+2} = \lambda_{2j,n}^{-1}(z), \qquad  z \in \Sigma_3.  \]
This establishes \eqref{jump:U:Sigma1} and \eqref{jump:U:Sigma3}.

The jump condition \eqref{jump:U:Sigma2} on $\Sigma_2$ follows from similar considerations. 

The other conditions in the RH problem \ref{RHPforU} are immediate from \eqref{asymp:U} and the
RH problem \ref{RHPforX} for $\mathbf{X}$.
\end{proof}

\section{The third transformation $\mathbf{U}\mapsto \mathbf{T}$}

\subsection{The $g$-functions}

In the third transformation we make use of the equilibrium measures described in Section~\ref{section:vep}. 
Let $(\mu_{1}^{*},\ldots,\mu_{d}^{*})$ be the vector of measures solving the vector equilibrium problem 
\eqref{energyfunc}--\eqref{supp:muk}. We associate to these measures the so-called $g$-functions.

\begin{definition} \label{define:gk} 
	We define for $k=1, \ldots, d$,
\begin{equation}\label{def:gk}
	g_{k}(z):=\int\log(z-t)\,\ud\mu_{k}^{*}(t)=-U^{\mu_{k}^{*}}(z)+\ir\int \arg(z-t)\,\ud\mu_{k}^{*}(t),
		\quad z \in \mathbb C \setminus (\Sigma_k^* \cup \mathbb R^-),
\end{equation}
where for each $t \in \Sigma_k^*$, $z \mapsto \arg(z-t)$ is defined with a branch cut along
$\mathbb R^- \cup [0, t]$.
\end{definition}

It follows from this definition that $g_k$ is analytic  in $\mathbb C \setminus (\Sigma_k^* \cup \mathbb R^-)$,
with symmetry relations
\begin{align}
	g_k(\omega^{\ell} z) &=g_k(z)+\frac{2\pi \ell \, \ir}{d+1} \| \mu_k^*\|, 
		\qquad && z\in S_{0},\quad \ell =\pm 1, \pm 2,\ldots, \pm \lfloor \frac{d}{2} \rfloor,\label{symm:gk}
		\end{align}
where we recall that $\| \mu_k^* \| = 1 - \frac{k-1}{d}$.

Note also that $g_k' = F_k$ where $F_k$ is the Cauchy transform of the measure $\mu_k^*$
as defined in \eqref{def:Fk}. The $g$-functions therefore satisfy certain jump relations,
which when differentiated give rise to \eqref{relationF1F2}--\eqref{relationFdFdm1}.
In the next lemma we state the relations for the $g$-function, where the main issue
is the determination of the constants of integration.

\begin{lemma} \label{lemma:jumpsgk}
The following relations hold:
\begin{itemize}
\item[\rm (a)] For $z\in [0, \omega^{\ell} x^{*}]\subset \Sigma_{1}$, $\ell =0,\pm 1,\ldots,\pm\lfloor \frac{d}{2} \rfloor$,
\begin{equation}\label{rel:g1g2:1}
g_{1,+}(z)+g_{1,-}(z)-g_{2}(z)=\frac{1}{t_{0}}\Big(\frac{d}{(d+1)\, t_{d+1}^{1/d}}|z|^{\frac{d+1}{d}}-\frac{t_{d+1}}{d+1}z^{d+1}\Big)+\frac{2\pi \ell \,\ir}{d}+\ell_{1},
\end{equation}
where $\ell_1$ is the variational constant from \eqref{varcondmu1}.

In addition, in case $d$ is odd, and $z\in[-x^{*},0]\subset\Sigma_{1}$,
\begin{equation}\label{rel:g1g2:2}
g_{1,+}(z)+g_{1,-}(z)-g_{2,\pm}(z)=\frac{1}{t_{0}}\Big(\frac{d}{(d+1)\, t_{d+1}^{1/d}}|z|^{\frac{d+1}{d}}-\frac{t_{d+1}}{d+1}z^{d+1}\Big)\pm\frac{\pi \ir\,(d-1)}{d}+\ell_{1}.
\end{equation}
\item[\rm (b)] For $k \geq 2$, 
\begin{equation}\label{rel:gfunc:1}
g_{k,+}(z)+g_{k,-}(z)=g_{k+1}(z)+g_{k-1}(z), \qquad z\in\Sigma_{k}\setminus \mathbb{R}^{-},
\end{equation}
where $\Sigma_{k}$ is defined in \eqref{def:Sigmak}, and in case that $k=d$, we understand $g_{d+1}\equiv 0$.
\item[\rm (c)] For $k \geq 2$, and $k \equiv d (\textrm{mod } 2)$.
\begin{equation}\label{rel:gfunc:2}
g_{k,+}(z)+g_{k,-}(z)=g_{k+1,\pm}(z)+g_{k-1,\pm}(z)\pm 2\pi\ir \|\mu_{k}^{*}\|,\qquad z\in\mathbb{R}^{-}.
\end{equation}
In case $k=d$ we again understand $g_{d+1} \equiv 0$.
\end{itemize}
\end{lemma}
\begin{proof}
(a) From the definition of $g_{1}$ we deduce that for $z\in [0,x^{*}]$,
\[
g_{1,+}(z)+g_{1,-}(z)=-2 U^{\mu_{1}^{*}}(z).
\]
Since $g_{2}$ is real-valued on $\mathbb{R}^{+}$, we have $g_{2}(z)=-U^{\mu_{2}^{*}}(z)$ 
for $z\in (0,x^{*}]$. Applying now \eqref{varcondmu1} and the fact  
$\supp(\mu_{1}^{*})=\Sigma_{1}^*$, we obtain \eqref{rel:g1g2:1} for $z \in [0, x^*]$. 
The remaining cases in \eqref{rel:g1g2:1} and \eqref{rel:g1g2:2} are then obtained  
by using the symmetry relations \eqref{symm:gk} of $g_1$ and $g_2$.
\medskip

(b) and (c)
In order to show \eqref{rel:gfunc:1} and \eqref{rel:gfunc:2} it is convenient to 
consider the four cases $k$ and $d$ even/odd separately. We consider the case
where $k$ is even and $d$ is odd. The other cases are analyzed in 
a similar manner. 

So we assume that $k$ is even and $d$ is odd. We first check that \eqref{rel:gfunc:1} is valid for 
$x\in \Sigma_{2}$ with $\arg x= \frac{\pi}{d+1}$. By definition of $g_{k}$, we have
\begin{equation}\label{eq:gk}
g_{k,+}(x)+g_{k,-}(x)=-2 U^{\mu_{k}^{*}}(x)+\lim_{z\rightarrow x^{-}}\ir\int_{\Sigma_{2}}\arg (z-t)\,\ud \mu_{k}^{*}(t)
+\lim_{z\rightarrow x^{+}}\ir\int_{\Sigma_{2}}\arg (z-t)\,\ud \mu_{k}^{*}(t).
\end{equation}
Observe that since $d$ is odd, $e^{\frac{\pi\ir}{d+1}}\mathbb{R}\subset\Sigma_{2}$ and we can write
\begin{multline*}
\lim_{z\rightarrow x^{-}}\int_{\Sigma_{2}}\arg (z-t)\,\ud \mu_{k}^{*}(t)
=\int_{\Sigma_{2}\setminus e^{\frac{\pi\ir}{d+1}}\mathbb{R}}\arg(x-t)\,\ud \mu_{k}^{*}(t)\\
+\frac{\pi}{d+1}\,\mu_{k}^{*} \left(e^{\frac{\pi\ir}{d+1}}\mathbb{R}\setminus (x,e^{\frac{\pi\ir}{d+1}}\infty) \right)+
\left(-\pi+\frac{\pi}{d+1}\right)\,\mu_{k}^{*}((x,e^{\frac{\pi\ir}{d+1}}\infty)),
\end{multline*}
where $-\pi<\arg(x-t)<\pi$. By the rotational invariance of $\mu_{k}^{*}$,
\[
\int_{\Sigma_{2}\setminus e^{\frac{\pi\ir}{d+1}}\mathbb{R}}\arg(x-t)\,\ud \mu_{k}^{*}(t)
=\frac{(d-1)\pi}{(d+1)^{2}}\,\|\mu_{k}^{*}\|.
\]
Since $\mu_{k}^{*}(e^{\frac{\pi\ir}{d+1}}\mathbb{R})=2\|\mu_{k}^{*}\|/(d+1)$, from the previous two identities we get 
\begin{equation}\label{boundlimitarg:1}
\lim_{z\rightarrow x^{-}}\int_{\Sigma_{2}}\arg (z-t)\,\ud \mu_{k}^{*}(t)=\frac{\pi \|\mu_{k}^{*}\|}{d+1}-\pi\mu_{k}^{*}((x,e^{\frac{\pi\ir}{d+1}}\infty)).
\end{equation}
On the other hand, it is easy to see that
\begin{equation}\label{boundlimitarg:2}
\lim_{z\rightarrow x^{+}}\int_{\Sigma_{2}}\arg (z-t)\,\ud \mu_{k}^{*}(t)=
\lim_{z\rightarrow x^{-}}\int_{\Sigma_{2}}\arg (z-t)\,\ud \mu_{k}^{*}(t)
+2\pi\,\mu_{k}^{*}\left((x,e^{\frac{\pi\ir}{d+1}}\infty)\right).
\end{equation}
Therefore from \eqref{eq:gk}, \eqref{boundlimitarg:1} and \eqref{boundlimitarg:2} we obtain
\begin{equation}\label{eq:gk:2}
g_{k,+}(x)+g_{k,-}(x)=-2 U^{\mu_{k}^{*}}(x)+\frac{2\pi\ir\, \|\mu_{k}^{*}\|}{d+1}.
\end{equation}
Now, the functions $g_{k-1}$ and $g_{k+1}$ are analytic on $e^{\frac{\pi\ir}{d+1}}\mathbb{R}^{+}$, 
and reasoning as before we deduce that
\begin{equation}\label{eq:gj}
g_{j}(x)=-U^{\mu_{j}^{*}}(x)+\frac{\pi\ir\,\|\mu_{j}^{*}\|}{d+1},\qquad j=k-1,k+1.
\end{equation}
It then follows from \eqref{varcondmuk} and \eqref{eq:gk:2}--\eqref{eq:gj} that \eqref{rel:gfunc:1} 
holds for $x\in\Sigma_{2}$ with $\arg x=\pi/(d+1)$. By applying the symmetry relations 
\eqref{symm:gk} we deduce from this that \eqref{rel:gfunc:1} holds everywhere in $\Sigma_{2}\setminus\mathbb{R}^{-}=\Sigma_{2}$.
\end{proof}

\subsection{The $\varphi$-functions}

The jump matrices in the RH problem that follows can be expressed 
in a very convenient way in terms of certain functions that we now introduce.

\begin{definition}
Let $\varphi_{1}$ be the analytic function defined on $\mathbb{C}\setminus(\Sigma_{1}^* \cup\Sigma_{2})$ by
\begin{equation}\label{def:phi1}
\varphi_{1}(z)=\frac{1}{2 t_{0}}\int_{\omega^{\ell}\,x^{*}}^{z}(\xi_{1}(s)-\xi_{2}(s))\,\ud s,\qquad 
	z\in S_{\ell}\setminus [0, \omega^{\ell} x^{*}],
\end{equation}
where for $z\in S_{\ell}\setminus [0, \omega^{\ell} x^{*}]$, 
integration is carried out in \eqref{def:phi1} along a path in $S_{\ell}\setminus [0, \omega^{\ell} x^{*}]$.

For $k \geq 2$, let  $\varphi_{k}$ be the analytic function 
on $\mathbb{C}\setminus(\Sigma_{2}\cup\Sigma_{3})$ given by
\begin{equation}\label{def:phik}
\varphi_{k}(z)=\frac{1}{2 t_{0}}\int_{0}^{z}(\xi_{k}(s)-\xi_{k+1}(s))\,\ud s\, 
	\mp (-1)^k \frac{\pi\ir}{(d+1)} \|\mu_{k}^*\|,
\qquad z\in S_{\ell}^{\pm},
\end{equation}
where for $z\in S_{\ell}^{\pm}$, integration in \eqref{def:phik} takes place 
along a path contained in $S_{\ell}^{\pm}$.
\end{definition}

\begin{lemma}\label{lemma:phifunc:1}
The following relations hold:
\begin{itemize}
\item[\rm (a)] For every $k=1,\ldots,d$,
\begin{equation}\label{symm:phifunc}
\varphi_{k}(\omega z)=\varphi_{k}(z),\qquad z\in\mathbb{C}\setminus(\Sigma_{2}\cup\Sigma_{3}).
\end{equation}
\item[\rm (b)] For $z\in\Sigma_{1}^* \setminus[-x^{*},0]$,
\begin{equation}\label{rel:g1phi1:1}
g_{1,+}(z)-g_{1,-}(z)=\pm 2 \varphi_{1,\pm}(z),
\end{equation}
and,  in case $d$ is odd, for $z\in[-x^{*},0]\subset\Sigma_{1}^*$,
\begin{equation}\label{rel:g1phi1:2}
g_{1,+}(z)-g_{1,-}(z)=\pm 2\varphi_{1,\pm}(z)-2\pi\ir.
\end{equation}
\item[\rm (c)] Let $2\leq k\leq d$. Then for $z\in\Sigma_{k}\setminus\mathbb{R}^{-}$,
\begin{equation}\label{rel:gkphik:1}
g_{k,+}(z)-g_{k,-}(z)=\pm 2\varphi_{k,\pm}(z)+\frac{2\ir d\, \sin((k-1)\pi/d)}{(d+1)\,t_{0}\,t_{d+1}^{1/d}}\,|z|^{\frac{d+1}{d}},
\end{equation}
and for $z\in\mathbb{R}^{-}\subset\Sigma_{k}$,
\begin{equation}\label{rel:gkphik:2}
g_{k,+}(z)-g_{k,-}(z)=\pm 2\varphi_{k,\pm}(z)+\frac{2\ir d\, \sin((k-1)\pi/d)}{(d+1)\,t_{0}\,t_{d+1}^{1/d}}\,|z|^{\frac{d+1}{d}}
-\frac{2\pi\ir (d-k+1)}{d}.
\end{equation}
\item[\rm (d)] For $z\in S_{\ell}\setminus [0, \omega^{\ell} x^{*}]$, $\ell=0,\pm 1,\ldots,\pm\lfloor d/2\rfloor$,
\begin{equation}\label{rel:g1g2phi1:1}
2g_{1}(z)-g_{2}(z)-\ell_{1}=2\varphi_{1}(z)+
\frac{1}{t_{0}}\left(\frac{d\, \omega_{d}^{-\ell}}{(d+1)\,t_{d+1}^{1/d}}\,z^{\frac{d+1}{d}}-
\frac{t_{d+1}}{d+1}\,z^{d+1}\right)+\frac{2\pi \ell \ir}{d},
\end{equation}
and in case $d$ is odd and for $z\in S_{\frac{d+1}{2}}^{-} \cup S_{-\frac{d+1}{2}}^+$,
\begin{equation}\label{rel:g1g2phi1:2}
2g_{1}(z)-g_{2}(z)-\ell_{1}=2\varphi_{1}(z)+
\frac{1}{t_{0}}\left(\frac{d\, \omega_{d}^{\pm\frac{d+1}{2}}}{(d+1)\,t_{d+1}^{1/d}}\,z^{\frac{d+1}{d}}-
\frac{t_{d+1}}{d+1}\,z^{d+1}\right)\mp\frac{\pi\ir(d+1)}{d}.
\end{equation}
\item[\rm (e)] For $z\in S_{0}^{\pm}$ and $2\leq k\leq d$,
\begin{equation}\label{rel:gksphik}
2g_{k}(z)-g_{k-1}(z)-g_{k+1}(z)=2\varphi_{k}(z)+(-1)^{k}\,\frac{d\,(\omega_{d}^{\mp\lceil\frac{k-1}{2}\rceil}
-\omega_{d}^{\pm\lfloor\frac{k-1}{2}\rfloor})}{(d+1)\,t_{0}\, t_{d+1}^{1/d}}\,z^{\frac{d+1}{d}},
\end{equation}
with the convention $g_{d+1}\equiv 0$.
\end{itemize}
\end{lemma}
\begin{proof}
(a) The symmetry property \eqref{symm:phifunc} is immediate from \eqref{symm:xik:2}
and \eqref{def:phi1}--\eqref{def:phik}. 

\medskip

(b)
We first check that \eqref{rel:g1phi1:1} holds for $x\in[0,x^{*}]$. Observe that for $s\in S_{0}\setminus[0,x^{*}]$,
\begin{equation}\label{eq:aux:0}
\xi_{1}(s)-\xi_{2}(s)=2 t_{0} F_{1}(s)-t_{0} F_{2}(s)-\frac{s^{1/d}}{t_{d+1}^{1/d}}+t_{d+1} s^{d},
\end{equation}
cf.~\eqref{rel:xi1F1} and \eqref{rel:xikFk}. Therefore applying \eqref{relationF1F2} we obtain from \eqref{def:phi1}
and \eqref{def:gk}
\begin{align*}
2\varphi_{1,\pm}(x) & =-\frac{1}{t_{0}}\int_{x}^{x^{*}}\Big(2 t_{0} F_{1,\pm}(s)-t_{0} F_{2}(s)-
	\frac{s^{1/d}}{t_{d+1}^{1/d}}+t_{d+1} s^{d}\Big)\ud s\\
& =-\int_{x}^{x^{*}}(F_{1,\pm}(s)-F_{1,\mp}(s)) \ud s \\
& =\pm\,2\pi\ir\,\mu_{1}^{*}([x,x^{*}])=\pm\, (g_{1,+}(x)-g_{1,-}(x)),
\end{align*}
and \eqref{rel:g1phi1:1} follows for $z=x \in [0,x^*]$. Because of \eqref{symm:phifunc}, we have for every $z\in\Sigma_{1}^{*}$,
\[
\varphi_{1,+}(z)=\varphi_{1,+}(|z|)=-\varphi_{1,-}(|z|)=-\varphi_{1,-}(z).
\]
Using now the symmetry property \eqref{symm:gk} we obtain \eqref{rel:g1phi1:1}--\eqref{rel:g1phi1:2}.

\medskip

(c)
The proof of \eqref{rel:gkphik:1}--\eqref{rel:gkphik:2} is analogous. For instance, if we let
\[
\widehat{\varphi}_{2j}(z)=\frac{1}{2 t_{0}}\int_{0}^{z}(\xi_{2j}(s)-\xi_{2j+1}(s))\,\ud s,
	\qquad z\in\mathbb{C}\setminus(\Sigma_{2}\cup\Sigma_{3}),
\]
it follows from the definitions of $\xi_{2j}$ and $\xi_{2j+1}$ on the sectors $S_{0}^{+}$ and $S_{1}^{-}$, 
and relations \eqref{relationFks} and \eqref{relationFdFdm1}, that for $x\in\Sigma_{2}$ with $\arg x=\pi/(d+1)$,
\begin{align}
2\widehat{\varphi}_{2j,\pm}(x) & =\frac{1}{t_{0}}\int_{0}^{|x|}
\Big(t_{0}\,(F_{2j,\pm}(e^{\frac{\pi\ir}{d+1}}t)-F_{2j,\mp}(e^{\frac{\pi\ir}{d+1}}t))\pm 
	\frac{\omega_{d}^{-j}-\omega_{d}^{j-1}}{t_{d+1}^{1/d}}\,e^{\frac{\pi\ir}{d(d+1)}}\,t^{1/d}\Big)\, 
		e^{\frac{\pi\ir}{d+1}}\,\ud t \notag\\
& =\mp\,2\pi\ir\,\mu_{2j}^{*}([0,x])\mp\frac{2\ir d\, \sin((2j-1)\pi/d)}{(d+1)\,t_{0}\,t_{d+1}^{1/d}}\,|x|^{\frac{d+1}{d}}.\label{eq:aux:1}
\end{align}
Since
\begin{equation}\label{eq:aux:2}
g_{2j,+}(x)-g_{2j,-}(x)=2\pi\ir\,\mu_{2j}^{*}([x,e^{\frac{\pi\ir}{d+1}}\infty))
=\frac{2\pi\ir\,(d-2j+1)}{d(d+1)}-2\pi\ir\,\mu_{2j}^{*}([0,x]),
\end{equation}
we deduce from \eqref{eq:aux:1} and \eqref{eq:aux:2} that \eqref{rel:gkphik:1} holds for $k = 2j$ even. We use 
again the symmetry properties of the $g$ and $\varphi$ functions to check that the relations 
\eqref{rel:g1phi1:1}--\eqref{rel:g1phi1:2} are valid on the remaining rays of $\Sigma_{2}$ in case $k$ is
even. The case $k$ odd can be proved in a similar way. 

\medskip
(d)
Using $g_{k}'=F_{k}$ and $2\varphi_{1}'=(\xi_{1}-\xi_{2})/t_{0}$, we can rewrite \eqref{eq:aux:0} as
\[
2 g_{1}'(z)-g_{2}'(z)=2\varphi_{1}'(z)-\frac{t_{d+1}}{t_{0}}\,z^{d}+\frac{z^{1/d}}{t_{0}\,t_{d+1}^{1/d}},
\qquad z\in S_{0}\setminus[0,x^{*}],
\]
and integration yields
\[
2g_{1}(z)-g_{2}(z)-\ell_{1}=2\varphi_{1}(z)+\frac{1}{t_{0}}\left(\frac{d}{(d+1)\,t_{d+1}^{1/d}}\,z^{\frac{d+1}{d}}
-\frac{t_{d+1}}{d+1}\,z^{d+1}\right)+C,
\]
for some constant $C$. It follows from \eqref{rel:g1g2:1} that $C=0$, and this proves \eqref{rel:g1g2phi1:1} 
for $\ell=0$. The remaining cases in \eqref{rel:g1g2phi1:1} and \eqref{rel:g1g2phi1:2} follow again from 
symmetry considerations.
\medskip

(e)
The proof of \eqref{rel:gksphik} is analogous to the proof of \eqref{rel:g1g2phi1:1}--\eqref{rel:g1g2phi1:2} 
and uses the relations \eqref{rel:gfunc:1} and \eqref{rel:gkphik:1}. 
\end{proof}

\begin{lemma}\label{lemma:phifunc:2}
The following relations hold:
\begin{itemize}
\item[\rm (a)] For $z\in\Sigma_{1}^*$,
\begin{equation}\label{rel:psi1:phi1}
e^{n(g_{1,-}(z)-g_{1,+}(z))} = e^{-2n\varphi_{1,+}(z)}=e^{2n\varphi_{1,-}(z)}.
\end{equation}
\item[\rm (b)] For $z\in\Sigma_{2}$,
\begin{equation}\label{rel:psi2k:phi2k}
	e^{n(g_{2k,-}(z) - g_{2k,+}(z))} \,\lambda_{2k-1,n}(z)=e^{-2n\varphi_{2k,+}(z)}=e^{2n\varphi_{2k,-}(z)},
		\qquad 1\leq k\leq \lfloor \tfrac{d}{2} \rfloor.
\end{equation}
\item[\rm (c)] For $z\in\Sigma_{3}$,
\begin{equation}\label{rel:psi2kp1:phi2kp1}
	e^{n(g_{2k+1,-}(z) - g_{2k+1,+}(z))} \,\lambda_{2k,n}(z)=e^{-2n\varphi_{2k+1,+}(z)}=e^{2n\varphi_{2k+1,-}(z)},
		\qquad 1\leq k\leq \lceil \frac{d}{2} \rceil -1.
\end{equation}
The functions $\lambda_{k,n}(z)$ are defined in \eqref{def:lambdafunc}.
\end{itemize}
\end{lemma}
\begin{proof}
The relations \eqref{rel:psi1:phi1}--\eqref{rel:psi2kp1:phi2kp1} are immediate from 
\eqref{rel:g1phi1:1}--\eqref{rel:gkphik:2}, since $n$ is a multiple of~$d$.
\end{proof}

\begin{lemma}
The following relations hold:
\begin{itemize}
\item[\rm(a)] For $z\in\Sigma_{1}^*$ and $k$ even, $2\leq k\leq d$,
\begin{equation}\label{rel:phikphikm1phikp1Sigma1}
\varphi_{k,+}(z)-\varphi_{k,-}(z)=\varphi_{k+1,-}(z)+\varphi_{k-1,-}(z),
\end{equation}
where $\varphi_{d+1}\equiv 0$.
\item[\rm (b)] For $z\in\Sigma_{2}$ and $k$ odd, we have
\begin{equation}\label{rel:phikphikm1phikp1Sigma2}
\varphi_{k,+}(z)-\varphi_{k,-}(z)= \begin{cases} 
\varphi_{2,-}(z)-\frac{\pi\ir}{d}, 	& \text{ if } k =1, \\
\varphi_{k+1,-}(z)+\varphi_{k-1,-}(z), & \text{ if } k \geq 3,
\end{cases}
\end{equation}
where $\varphi_{d+1}\equiv 0$.
\end{itemize}
\end{lemma}
\begin{proof}
(a) By \eqref{symm:phifunc} it is enough to prove \eqref{rel:phikphikm1phikp1Sigma1} on $(0,x^{*})$. 
Assume that $k$ is even with $k\geq 4$; the same proof with a slight variation can be used for $k=2$. Applying \eqref{rel:gksphik} for  $z \in S_{0}^{+}$,
\begin{equation}\label{rel:gks:1}
	g_{k-1}(z)+g_{k+1}(z)-g_{k-2}(z)-g_{k+2}(z)
	=2\sum_{i=k-1}^{k+1}\varphi_{i}(z)+\frac{d(\omega_{d}^{\frac{k}{2}}-\omega_{d}^{-(\frac{k}{2}-1)})z^{\frac{d+1}{d}}}{(d+1)\,
t_{0}\,t_{d+1}^{1/d}},
\end{equation}
and similarly for $z\in S_{0}^{-}$,
\begin{equation}\label{rel:gks:2}
	g_{k-1}(z)+g_{k+1}(z)-g_{k-2}(z)-g_{k+2}(z) =2\sum_{i=k-1}^{k+1}\varphi_{i}(z)
+\frac{d(\omega_{d}^{-\frac{k}{2}}-\omega_{d}^{(\frac{k}{2}-1)}) z^{\frac{d+1}{d}}}{(d+1)\,t_{0}\,t_{d+1}^{1/d}}.
\end{equation}
We let $z$ tend to the interval $(0,x^*)$ in both \eqref{rel:gks:1}--\eqref{rel:gks:2}, which
leads to a $+$-boundary value in \eqref{rel:gks:1} and a $-$-boundary value in \eqref{rel:gks:2}.
We subtract the results and  we get for $z \in (0,x^*)$,
\begin{multline*}
	g_{k-1,+}(z)-g_{k-1,-}(z)+g_{k+1,+}(z)-g_{k+1,-}(z)\\
	=2\sum_{i=k-1}^{k+1}(\varphi_{i,+}(z)-\varphi_{i,-}(z))+\frac{2\ir d\, (\sin(\pi k/d)+\sin(\pi(k-2)/d))}{(d+1)\,t_{0}\,t_{d+1}^{1/d}},
\end{multline*}
where we used the fact that for $i=k-2, k+2,$ the boundary values $g_{i,\pm}(z)$ are the same. Now 
\eqref{rel:phikphikm1phikp1Sigma1} follows immediately from \eqref{rel:gkphik:1}. In the 
case $k=2$ one also uses \eqref{rel:g1g2phi1:1}.

\medskip
(b) The proof of \eqref{rel:phikphikm1phikp1Sigma2} is completely analogous, 
and again by \eqref{symm:phifunc} it is sufficient to prove these relations for $z\in\Sigma_{2}$ with 
$\arg z=\pi/(d+1)$. From \eqref{rel:gksphik} and the symmetry properties of the $g$ and $\varphi$ functions 
one deduces that for $z \in S_{1}^{-}$ and $2\leq k\leq d$,
\[
2g_{k}(z)-g_{k-1}(z)-g_{k+1}(z)=2\varphi_{k}(z)
+(-1)^{k}\,\frac{d(\omega_{d}^{\lceil\frac{k-3}{2}\rceil}-\omega_{d}^{-\lfloor\frac{k+1}{2}\rfloor})}{(d+1)\,
t_{0}\,t_{d+1}^{1/d}}\,z^{\frac{d+1}{d}}.
\]
The details are left to the reader.
\end{proof}

\subsection{Transformation $\mathbf{U} \mapsto \mathbf{T}$}

We now define the third transformation of the steepest descent analysis. 

\begin{definition}
The matrix valued function $\mathbf{T}(z)$ is defined by 
\begin{equation}\label{def:T}
\mathbf{T}(z) =L^{n}\, \mathbf{U}(z)\, \mathbf{G}(z)^{n}\, L^{-n},\qquad z\in\mathbb{C}\setminus(\Sigma_{2}\cup\Sigma_{3}),
\end{equation}
where $\mathbf{G}(z)$ and $L$ are the diagonal matrices of size $(d+1)\times(d+1)$ given by
\begin{align*}
\mathbf{G}(z) & =\diag (e^{-g_{1}(z)}, e^{g_{1}(z)-g_{2}(z)}, e^{g_{2}(z)-g_{3}(z)}, \ldots, e^{g_{d-1}(z)-g_{d}(z)}, e^{g_{d}(z)}),\\
L &=\diag(e^{-\ell_{1}}, 1, 1, \ldots, 1),
\end{align*}
where $\ell_{1}$ is the constant in \eqref{varcondmu1}.
\end{definition}

\begin{rhp}\label{RHPforT}
The matrix $\mathbf{T}$ is the solution to the following RH problem:
\begin{itemize}
\item[$\bullet$] $\mathbf{T}$ is analytic in $\mathbb{C}\setminus \Sigma_T $, where $\Sigma_T = \Sigma_U = \Sigma_2 \cup \Sigma_3$.
\item[$\bullet$] $\mathbf{T}_{+}=\mathbf{T}_{-} \mathbf{J}_{T} $ on $\Sigma_{2}\cup\Sigma_{3}$ with jump matrix  (where we put $\varphi_{d+1} \equiv 0$)
 \begin{align} \label{jump:T:Sigma1}
    \mathbf{J}_T & = \diag(e^{-2n \varphi_{1,+}}, e^{-2n \varphi_{1,-}}, e^{-2n \varphi_{3,+}}, e^{-2n \varphi_{3,-}}, \ldots ) 
			+ \sum_{j=1}^{\lceil \frac{d}{2} \rceil} E_{2j-1,2j} && \text{ on } \Sigma_1^*, \\
			\label{jump:T:Sigma1hat}
    \mathbf{J}_T & = \diag(1,1, e^{-2n \varphi_{3,+}}, e^{-2n \varphi_{3,-}}, \ldots) + e^{2n \varphi_1} E_{1,2} + 
			\sum_{j=2}^{\lceil \frac{d}{2} \rceil} E_{2j-1,2j} && \text{ on }
			\Sigma_1 \setminus \Sigma_1^*, \\
			\label{jump:T:Sigma3}
		\mathbf{J}_T & = \diag(1,1, e^{-2n \varphi_{3,+}}, e^{-2n \varphi_{3,-}}, \ldots) + 
			\sum_{j=2}^{\lceil \frac{d}{2} \rceil} E_{2j-1,2j} && \text{ on }
			\Sigma_3 \setminus \Sigma_1, \\
			\label{jump:T:Sigma2}
		\mathbf{J}_T & = \diag(1, \omega^{-1} e^{-2n \varphi_{2,+}}, \omega e^{-2n \varphi_{2,-}}, \omega^{-2} e^{-2n \varphi_{4,+}}, \ldots)
			+ \sum_{j=1}^{\lfloor \frac{d}{2} \rfloor} E_{2j,2j+1} && \text{ on } \Sigma_2.
 \end{align}
\item[$\bullet$] As $z\rightarrow\infty$ 
\begin{equation}\label{asymp:T}
	\mathbf{T}(z)=\Big(I_{d+1} +O\left( z^{-\frac{2}{d}} \right) \Big) 
		\begin{pmatrix} 1 & 0 \\ 0 & \mathbf{A}(z) \end{pmatrix}.
\end{equation}
\item[$\bullet$] $\mathbf{T}$ satisfies the same endpoint condition \eqref{endpoint:Y} as $\mathbf{Y}$.
\end{itemize}
\end{rhp}

\begin{remark}
The jump matrices are of size $(d+1) \times (d+1)$ and so the pattern on the diagonal in \eqref{jump:T:Sigma1}--\eqref{jump:T:Sigma2}
continues until we reach size $(d+1) \times (d+1)$. If $d$ is even, then the last diagonal entry in \eqref{jump:T:Sigma1}--\eqref{jump:T:Sigma3}
is $e^{-2n \varphi_{d+1,+}} \equiv 1$, because of our convention that $\varphi_{d+1} \equiv 0$.
If $d$ is odd, then the last diagonal entry in \eqref{jump:T:Sigma2} is 
$\omega^{d+1} e^{-2n \varphi_{d+1,+}} \equiv -1$.
\end{remark}

\begin{proof}
By definition \eqref{def:T}, the matrix $\mathbf{J}_{T}(z)$ is given by
\[
\mathbf{J}_{T}(z)=\mathbf{T}_{-}^{-1}(z) \mathbf{T}_{+}(z)=L^{n}\,\mathbf{G}_{-}(z)^{-n}\,\mathbf{J}_{U}(z)\,\mathbf{G}_{+}(z)^{n}\,L^{-n},
\]
and recall that $\mathbf{G}(z)$ and $L$ are diagonal matrices.

The $k$ th diagonal entry of $\mathbf{J}_U$ is multiplied by 
\[ e^{n(g_{k-1,+}-g_{k,+})} e^{-n(g_{k-1,-}-g_{k,-})} \]
where we put $g_0 \equiv 0$ and $g_{d+1} \equiv 0$. On $\Sigma_k$ we have
that $e^{ng_{k-1,+}} = e^{n g_{k-1,-}}$ and then it reduces to multiplication
by $e^{n(g_{k,-}-g_{k,+})}$. Lemma \ref{lemma:phifunc:2} expresses this in terms of
$e^{-2n \varphi_{k,+}}$ and $e^{-2n \varphi_{k,-}}$. It all combines nicely
with the entries $\lambda_{k,n}$ and $\lambda_{k,n}^{-1}$ and this
gives the diagonal entries in the jump matrices \eqref{jump:T:Sigma1}--\eqref{jump:T:Sigma2}.

The $(1,2)$ entry in the jump matrix on $\Sigma_1$ is multiplied by
$e^{n(g_{1,+}(z) + g_{1,-}(z) - g_{2,+}(z) - \ell_1)}$, 
and so we obtain
\[ e^{n(- \frac{1}{t_0} V_1(z) + g_{1,+}(z) + g_{1,-}(z) - g_{2,+}(z) - \ell_1)} 
	= \begin{cases} 1 & \text{ on } \Sigma_1^*,  \\
		e^{2n \varphi_1(z)} & \text{ on } \Sigma_1 \setminus \Sigma_1^*,
		\end{cases} \]
by \eqref{rel:g1g2:1}  and \eqref{rel:g1g2phi1:1}.
We also use here that $n$ is a multiple of $d$.
This gives the $(1,2)$ entries in the jumps in \eqref{jump:T:Sigma1} and \eqref{jump:T:Sigma1hat}.

The entry $(k,k+1)$ with $k \geq 2$ is multiplied by
$e^{n(g_{k,+}(z) + g_{k,-}(z) - g_{k+1,+}(z) - g_{k-1,-}(z))}$.
The relations \eqref{rel:gfunc:1}--\eqref{rel:gfunc:2}  show that this is $1$
for $z \in \Sigma_k$. The expressions with even $k$ are relevant on $\Sigma_2$ and
those with odd $k$ on $\Sigma_3$. All together this implies that 
$(\mathbf{J}_U)_{k,k+1} = 1$ on $\Sigma_2$ if $k$ is even and
$(\mathbf{J}_U)_{k,k+1} = 1$ on $\Sigma_3$ is $k \geq 3$ is odd. 
This proves the jumps in the RH problem.

The asymptotic condition in the RH problem for $\mathbf{T}$ follows from \eqref{asymp:U} and 
the fact that
\[ \mathbf{G}(z)^n = \diag(z^{-n}, z^{\frac{n}{d}} I_d) \times (I_{d+1} + O(z^{-\frac{d+1}{d}})) 
	\qquad \text{ as } z \to \infty \]
because of the normalizations of the measures $\|\mu_k^*\| = 1 - \frac{k-1}{d}$,
and 
\begin{align*} g_1(z) & =  \log z + O(z^{-d-1}) \\
		g_k(z) & = \|\mu_k^*\| \log z + O(z^{- \frac{d+1}{d}}) \qquad \text{for } k \geq 2
		\end{align*}
		as $z \to \infty$.
\end{proof}

\section{Fourth transformation $\mathbf{T}\mapsto \mathbf{S}$}

In the fourth transformation of the RH steepest descent analysis we construct lenses 
around the stars $\Sigma_1^*$, $\Sigma_2$ and $\Sigma_3$. 
Before describing this construction we analyze certain properties of the $\varphi$-functions.

Recall that $\varphi_{1}$ is an analytic function on $\mathbb{C}\setminus(\Sigma_{1}^* \cup\Sigma_{2})$. 
As a consequence of \eqref{rel:g1phi1:1}--\eqref{rel:g1phi1:2} we know that $\Real \varphi_{1,\pm}=0$ on $\Sigma_{1}^*$. 
Moreover, since $\varphi_{2,\pm}$ is purely imaginary on $\Sigma_{2}$ (cf. \eqref{rel:gkphik:1}--\eqref{rel:gkphik:2}), 
it follows from \eqref{rel:phikphikm1phikp1Sigma2}  that $\Real \varphi_{1,+}=\Real \varphi_{1,-}$ on $\Sigma_{2}$. Therefore, 
the harmonic function $\Real \varphi_{1}$ on $\mathbb{C}\setminus(\Sigma_{1}^*\cup\Sigma_{2})$ can be extended 
continuously to the whole complex plane.

Consider now the set
\[
D:=\{z\in\mathbb{C}: \Real \varphi_{1}(z)<0\}.
\]
By the maximum principle for harmonic functions, the boundary $\partial D$ of $D$ consists of a finite union of 
analytic arcs that start and end on $\Sigma_{1}^*\cup \Sigma_{2}$ or at infinity. Since 
$\varphi_{1,\pm}(z)=\pm \pi \ir\, \mu_{1}^{*}([z, \omega^{\ell} x^{*}])$ for 
$z\in [0, \omega^{\ell} x^{*}]\subset\Sigma_{1}^*$, $\ell=0,\ldots,d$, the Cauchy-Riemann equations allow 
us to conclude that for every $z\in [0, \omega^{\ell} x^{*})$, there exists a neighborhood $U$ of 
$z$ such that $\Real\varphi_{1}>0$ on $U\setminus[0,\omega^{\ell} x^{*})$. 
Here we use that the density of $\mu_{1}^{*}$ is strictly positive at every point $z\in[0,\omega^{\ell} x^{*})$. 
Hence the closed set $\mathbb{C}\setminus D$ contains in its interior the segments 
$[0, \omega^{\ell} x^{*})$.

Recall that $-\ir (\xi_2 - \xi_1)_+ > 0$ on $[0, x^*)$, since this is the density of $\mu_1^*$ up
to a positive constant, see \eqref{def:mu1} with $\xi_{1,-}=\xi_{2,+}$ on $\Sigma_{1}^{*}$. 
It also vanishes as a square root at $x^*$. Then it follows
that $\xi_2(x) - \xi_1(x) > 0$ for $x \in (x^*, \widehat{x}]$ provided we take $\widehat{x}$
sufficiently close to $x^*$. We deduce from \eqref{def:phi1} that $(x^{*},\widehat{x}]\subset D$ 
and $x^{*}\in\partial D$. Note also that
\[
\varphi_{1}(z)=-c(z-x^{*})^{3/2}+O((z-x^{*})^{5/3}), \qquad \text{as } z\rightarrow x^{*},
\]
where $c>0$,  
which implies that $\partial D$ makes angles $\pm \pi/3$ with $[0,x^{*}]$ at $x^{*}$. 
By symmetry, the same properties hold for the other segments $(\omega^{\ell}x^{*}, \omega^{\ell} \widehat{x}]$ 
and $[0, \omega^{\ell} x^{*}]$.

On the other hand, observe that
\begin{equation}\label{asymp:phi1}
2\varphi_{1}(z)=\frac{t_{d+1}}{(d+1)\,t_{0}}\, z^{d+1}+O(z^{\frac{d+1}{d}}),\qquad z\rightarrow\infty,
\end{equation}
which follows from \eqref{rel:g1g2phi1:1}--\eqref{rel:g1g2phi1:2}. From \eqref{asymp:phi1} 
it is easy to deduce that the analytic arcs of $\partial D$ that start at the points $\omega^{\ell} x^{*}$ 
cannot end at infinity and therefore necessarily end at points on $\Sigma_{2}$.

As a consequence of \eqref{rel:gkphik:1}--\eqref{rel:gkphik:2} we know that for $z\in\Sigma_{k}$,
\[
\ir\varphi_{k,+}(z)=-\ir\varphi_{k,-}(z)=\pi\mu_{k}^{*}([0,z])+
\frac{d\,\sin( \frac{(k-1)\pi}{d})}{(d+1)\,t_{0}\,t_{d+1}^{1/d}}\,|z|^{\frac{d+1}{d}}-
	\frac{\pi}{(d+1)} \| \mu_k^*\|,
\]
hence $\pm\ir\varphi_{k,\pm}$ is strictly increasing along each ray of $\Sigma_{k}$. 
Again from the Cauchy-Riemann equations it follows that $\Real \varphi_{k}>0$ on both sides of $\Sigma_{k}$.

We now describe the construction of the lenses around the stars $\Sigma_1^*$, $\Sigma_2$ and $\Sigma_3$. They are shown in 
Figure~\ref{lensesaroundstars} in the case $d=3$. We use 
\[
L_{1}=L_{1}^{+}\cup L_{1}^{-} \cup \Sigma_1^*
\]
to denote the lens around $\Sigma_{1}^*$, where $L_{1}^{+}$ ($L_{1}^{-}$) is the part of 
$L_{1}$ that lies on the $+$ side ($-$ side) of $\Sigma_{1}^*$. The boundary $\partial L_{1}$ of $L_{1}$ 
intersects $\Sigma_{3}$ at the points $\omega^{\ell}\,x^{*}$, $\ell=0,\ldots,d$ and intersects 
$\Sigma_{2}$ at points at a positive distance $\delta_{1}>0$ from the origin. By the properties of the 
$\varphi$-functions discussed above, it is possible to take $L_{1}$ so that 
$\partial L_{1}\setminus\{\omega^{\ell}\,x^{*}\}_{\ell=0}^{d}$ is contained in the regions where 
$\Real \varphi_{k}>0$, for every odd $k\geq 1$.

Analogously we construct the lens
\[
L_{2}=L_{2}^{+}\cup L_{2}^{-} \cup \Sigma_2
\]
around $\Sigma_{2}$. It is chosen so that $\partial L_{2}$ lies in the regions where $\Real \varphi_{k}>0$, 
for every even $k\geq 2$. The boundary $\partial L_{2}$ consists of $2(d+1)$ rays that meet $\Sigma_{1}$ at 
points at a positive distance $\delta_{2}>0$ from the origin, and have asymptotic angles 
$\frac{(2\ell+1)\pi}{d+1}\pm \varepsilon$, for some small $\varepsilon>0$.

Finally, we construct a similar lens
\[
L_{3}=L_{3}^{+}\cup L_{3}^{-} \cup \Sigma_3
\]
around $\Sigma_{3}$. We choose it such that $L_1 \subset L_3$, and the boundary $\partial L_3$ 
consists of $\partial L_1 \cap L_2$ and infinite rays that emanate from the intersection points 
of $\partial L_{1}$ and $\partial L_{2}$, see Figure~\ref{lensesaroundstars}. 
These rays are chosen so that they are contained in the regions where $\Real \varphi_{k}>0$, 
for every odd $k\geq 3$, and with asymptotic angles $\frac{2\pi\ell}{d+1}\pm\varepsilon$.

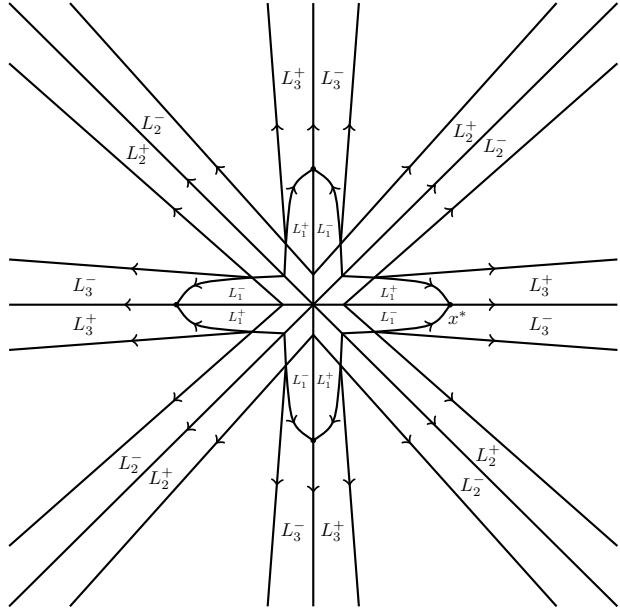
\begin{figure}[t]
\begin{center}
\begin{tikzpicture}[scale = 2,line width = .5]
\draw [thick, postaction = decorate, decoration = {markings, mark = at position .8 with {\arrow[black]{>};}}, decoration = {markings, mark = at position .2 with {\arrow[black]{<};}}] (-2,0) -- (2,0);
\draw [thick,postaction = decorate, decoration = {markings, mark = at position .8 with {\arrow[black]{>};}}, decoration = {markings, mark = at position .2 with {\arrow[black]{<};}}] (0,-2) -- (0,2);
\draw [thick, postaction = decorate, decoration = {markings, mark = at position .7 with {\arrow[black]{>};}}, decoration = {markings, mark = at position .3 with {\arrow[black]{<};}}] (-2,-2) -- (2,2);
\draw [thick, postaction = decorate, decoration = {markings, mark = at position .7 with {\arrow[black]{>};}}, decoration = {markings, mark = at position .3 with {\arrow[black]{<};}}] (-2,2) -- (2,-2);
\filldraw [black] (0.9,0) circle (0.4pt);
\filldraw [black] (0,0.9) circle (0.4pt);
\filldraw [black] (-0.9,0) circle (0.4pt);
\filldraw [black] (0,-0.9) circle (0.4pt);
\draw [thick,postaction = decorate, decoration = {markings, mark = at position .4 with {\arrow[black]{>};}}] (0.2,0) -- (2,1.6);
\draw [thick,postaction = decorate, decoration = {markings, mark = at position .4 with {\arrow[black]{>};}}] (0.2,0) -- (2,-1.6);
\draw [rotate=90,thick,postaction = decorate, decoration = {markings, mark = at position .4 with {\arrow[black]{>};}}] (0.2,0) -- (2,1.6);
\draw [rotate=90,thick,postaction = decorate, decoration = {markings, mark = at position .4 with {\arrow[black]{>};}}] (0.2,0) -- (2,-1.6);
\draw [rotate=180,thick,postaction = decorate, decoration = {markings, mark = at position .4 with {\arrow[black]{>};}}] (0.2,0) -- (2,1.6);
\draw [rotate=180,thick,postaction = decorate, decoration = {markings, mark = at position .4 with {\arrow[black]{>};}}] (0.2,0) -- (2,-1.6);
\draw [rotate=270,thick,postaction = decorate, decoration = {markings, mark = at position .4 with {\arrow[black]{>};}}] (0.2,0) -- (2,1.6);
\draw [rotate=270,thick,postaction = decorate, decoration = {markings, mark = at position .4 with {\arrow[black]{>};}}] (0.2,0) -- (2,-1.6);
\draw (0.95,0) node[scale=0.7, below]{$x^{*}$};
\draw [thick, postaction = decorate, decoration = {markings, mark = at position .5 with {\arrow[black]{>};}}] (0.4,0.18) -- (2,0.3);
\draw [thick, smooth, postaction = decorate, decoration = {markings, mark = at position .7 with {\arrow[black]{>};}}] (0.4,0.18).. controls (0.8,0.15) .. (0.9,0);
\draw [thick, smooth, postaction=decorate, decoration = {markings, mark = at position .7 with {\arrow[black]{>};}}] (0.4,-0.18) .. controls (0.8,-0.15) .. (0.9,0);
\draw [thick, smooth, postaction = decorate, decoration = {markings, mark = at position .7 with {\arrow[black]{>};}}, rotate=90] (0.4,0.18).. controls (0.8,0.15) .. (0.9,0);
\draw [thick, smooth, postaction=decorate, decoration = {markings, mark = at position .7 with {\arrow[black]{>};}}, rotate=90] (0.4,-0.18) .. controls (0.8,-0.15) .. (0.9,0);
\draw [thick, smooth, postaction = decorate, decoration = {markings, mark = at position .7 with {\arrow[black]{>};}},rotate=180] (0.4,0.18).. controls (0.8,0.15) .. (0.9,0);
\draw [thick, smooth, postaction=decorate, decoration = {markings, mark = at position .7 with {\arrow[black]{>};}},rotate=180] (0.4,-0.18) .. controls (0.8,-0.15) .. (0.9,0);
\draw [thick, smooth, postaction = decorate, decoration = {markings, mark = at position .7 with {\arrow[black]{>};}},rotate=270] (0.4,0.18).. controls (0.8,0.15) .. (0.9,0);
\draw [thick, smooth, postaction=decorate, decoration = {markings, mark = at position .7 with {\arrow[black]{>};}},rotate=270] (0.4,-0.18) .. controls (0.8,-0.15) .. (0.9,0);
\draw [thick,smooth] (0.19,0.19) -- (0.4,0.18);
\draw [thick,smooth] (0.19,-0.19) -- (0.4,-0.18);
\draw [thick,smooth, rotate=90] (0.19,0.19) -- (0.4,0.18);
\draw [thick,smooth, rotate=90] (0.19,-0.19) -- (0.4,-0.18);
\draw [thick,smooth,rotate=180] (0.19,0.19) -- (0.4,0.18);
\draw [thick,smooth,rotate=180] (0.19,-0.19) -- (0.4,-0.18);
\draw [thick,smooth,rotate=270] (0.19,0.19) -- (0.4,0.18);
\draw [thick,smooth,rotate=270] (0.19,-0.19) -- (0.4,-0.18);
\draw [thick,postaction = decorate, decoration = {markings, mark = at position .5 with {\arrow[black]{>};}}] (0.4,-0.18) -- (2,-0.3);
\draw [thick, rotate=90,postaction = decorate, decoration = {markings, mark = at position .5 with {\arrow[black]{>};}}] (0.4,0.18) -- (2,0.3);
\draw [thick, rotate=90,postaction = decorate, decoration = {markings, mark = at position .5 with {\arrow[black]{>};}}] (0.4,-0.18) -- (2,-0.3);
\draw [thick, rotate=180,postaction = decorate, decoration = {markings, mark = at position .5 with {\arrow[black]{>};}}] (0.4,0.18) -- (2,0.3);
\draw [thick, rotate=180,postaction = decorate, decoration = {markings, mark = at position .5 with {\arrow[black]{>};}}] (0.4,-0.18) -- (2,-0.3);
\draw [thick, rotate=270,postaction = decorate, decoration = {markings, mark = at position .5 with {\arrow[black]{>};}}] (0.4,0.18) -- (2,0.3);
\draw [thick, rotate=270,postaction = decorate, decoration = {markings, mark = at position .5 with {\arrow[black]{>};}}] (0.4,-0.18) -- (2,-0.3);
\draw (1.5,0.13) node[scale=0.7]{$L_{3}^{+}$};
\draw (1.5,-0.13) node[scale=0.7]{$L_{3}^{-}$};
\draw [rotate=90] (1.5,0.13) node[scale=0.7]{$L_{3}^{+}$};
\draw [rotate=90] (1.5,-0.13) node[scale=0.7]{$L_{3}^{-}$};
\draw [rotate=180] (1.5,0.13) node[scale=0.7]{$L_{3}^{+}$};
\draw [rotate=180] (1.5,-0.13) node[scale=0.7]{$L_{3}^{-}$};
\draw [rotate=270] (1.5,0.13) node[scale=0.7]{$L_{3}^{+}$};
\draw [rotate=270] (1.5,-0.13) node[scale=0.7]{$L_{3}^{-}$};
\draw (0.5,0.08) node[scale=0.5]{$L_{1}^{+}$};
\draw (0.5,-0.08) node[scale=0.5]{$L_{1}^{-}$};
\draw [rotate=90](0.5,0.08) node[scale=0.5]{$L_{1}^{+}$};
\draw [rotate=90](0.5,-0.08) node[scale=0.5]{$L_{1}^{-}$};
\draw [rotate=180](0.5,0.08) node[scale=0.5]{$L_{1}^{+}$};
\draw [rotate=180](0.5,-0.08) node[scale=0.5]{$L_{1}^{-}$};
\draw [rotate=270](0.5,0.08) node[scale=0.5]{$L_{1}^{+}$};
\draw [rotate=270](0.5,-0.08) node[scale=0.5]{$L_{1}^{-}$};
\draw (1,1.15) node[scale=0.7]{$L_{2}^{+}$};
\draw [rotate=90](1,1.15) node[scale=0.7]{$L_{2}^{+}$};
\draw [rotate=180](1,1.15) node[scale=0.7]{$L_{2}^{+}$};
\draw [rotate=270](1,1.15) node[scale=0.7]{$L_{2}^{+}$};
\draw (1.2,1.05) node[scale=0.7]{$L_{2}^{-}$};
\draw [rotate=90](1.2,1.05) node[scale=0.7]{$L_{2}^{-}$};
\draw [rotate=180](1.2,1.05) node[scale=0.7]{$L_{2}^{-}$};
\draw [rotate=270](1.2,1.05) node[scale=0.7]{$L_{2}^{-}$};
\end{tikzpicture}
\caption{The lenses $L_{1}$, $L_{2}$ and $L_{3}$ around the stars $\Sigma_{1}^*$, $\Sigma_{2}$ and $\Sigma_{3}$, in the case $d=3$.}
\label{lensesaroundstars}
\end{center}
\end{figure}

\begin{definition} \label{def:Sz}
Let $\mathbf{S}(z)$ be the matrix-valued function defined as follows. 
We set $\mathbf{S}(z)=\mathbf{T}(z)$ for $z$ outside the lenses $L_2$ and $L_3$,
and
\begin{align}\label{def:S:L3pm:1}
	\mathbf{S} & =\mathbf{T} \left(I_{d+1} \mp \sum_{j=1}^{\lceil \frac{d}{2} \rceil} e^{-2n \varphi_{2j-1}} E_{2j,2j-1} \right)
	&& \text{ in } L_1^{\pm} \setminus L_2, \\
	\label{def:S:L3pm:2}
	\mathbf{S} & = \mathbf{T} \left(I_{d+1} \mp \sum_{j=2}^{\lceil \frac{d}{2} \rceil} e^{-2n \varphi_{2j-1}} E_{2j,2j-1} \right)
	&& \text{ in } L_3^{\pm} \setminus L_1, \\
	\label{def:S:L3pm:3}
	\mathbf{S} & = \mathbf{T} \left(I_{d+1} \mp \sum_{j=1}^{\lfloor \frac{d}{2} \rfloor} \omega^{\mp j} e^{-2n \varphi_{2j}} E_{2j+1,2j} \right)
	&& \text{ in } L_2^{\pm} \setminus L_1. \\
	\label{def:S:L3pm:4}
	\mathbf{S} & = \mathbf{T} \left(I_{d+1} + \sum_{j=2}^{d+1} \sum_{k=1}^{j-1} c_{j,k}^{\pm} e^{-2n(\varphi_k + \cdots + \varphi_{j-1})} E_{j,k}\right)
	&& \text{ in } L_1^{\pm} \cap L_2^{\mp}.
	\end{align}
\end{definition}

The definition of $\mathbf{S}$ in the intersection of the lenses, see \eqref{def:S:L3pm:4}, involves
certain constants $c_{j,k}^{\pm}$ with $1 \leq k < j \leq d+1$, that are not yet defined. We collect these numbers in two lower
triangular matrices
\begin{equation} \label{Cpm} 
	C^\pm  = \begin{pmatrix} 1 & 0 & 0 & \cdots & 0 \\
	c_{2,1}^\pm & 1 & 0 & \cdots & 0 \\
	c_{3,1}^\pm & c_{3,2}^\pm & 1 & \cdots & 0 \\
	\vdots & \vdots & \vdots & \ddots & \vdots \\
	c_{d+1,1}^\pm & c_{d+1,2}^\pm & c_{d+1,3}^\pm & \cdots & 1 \end{pmatrix}.
	\end{equation}
	
Recall that $\sigma_2$ is the Pauli  matrix $\sigma_2 = \begin{pmatrix} 0 & -\ir \\ \ir & 0 \end{pmatrix}$ so that
\[ \ir \sigma_2 = \begin{pmatrix} 0 & 1 \\ -1 & 0\end{pmatrix}. \]
This matrix will appear on the diagonal in the block diagonal matrices \eqref{eq:Cpm1}-\eqref{eq:Cpm2}.
\begin{lemma} \label{lemma:Cpmexist}
There exist two lower triangular matrices $C^+$ and $C^-$ of the form \eqref{Cpm}
such that
\begin{equation} \label{eq:Cpm1} 
	\left(I_{d+1} + \sum_{j=1}^{\lceil \frac{d}{2} \rceil} E_{2j-1,2j} \right) C^+ = C^-
	\times \begin{cases}
	\diag(\ir\sigma_2, \ldots, \ir\sigma_2) & \quad \text{ if $d$ is odd,} \\
	\diag(\ir\sigma_2, \ldots, \ir\sigma_2,1) & \quad \text{ if $d$ is even,} 
	\end{cases} \end{equation}
and
\begin{equation} \label{eq:Cpm2} 
	\left( \diag(1, \omega^{-1}, \omega, \omega^{-2}, \cdots) + \sum_{j=1}^{\lfloor \frac{d}{2} \rfloor} E_{2j,2j+1}\right)
	C^- = C^+ \times \begin{cases}
	\diag(1, \ir \sigma_2, \ldots, \ir \sigma_2,-1) & \quad \text{ if $d$ is odd,} \\
	\diag(1, \ir \sigma_2, \ldots, \ir \sigma_2) & \quad \text{ if $d$ is even.}
	\end{cases} 	
	\end{equation}
\end{lemma}

\begin{proof}
In order to define the matrices $C^{\pm }$ we use the elementary symmetric 
polynomials $e_{k,n}(x_1, \ldots, x_n)$ as in \eqref{elementary:1}--\eqref{elementary:2} and  we put
\begin{align} 
	e_{k,n} & = e_{k,n}(1, \omega^{-1}, \omega, \omega^{-2}, \omega^2, \omega^{-3}, \ldots), \label{eknandprime1}\\
	e_{k,n}' & = e_{k,n}(1, \omega, \omega^{-1}, \omega^2, \omega^{-2}, \omega^3, \ldots).\label{eknandprime2}
\end{align}
Recall that $n$ denotes the number of variables appearing in the symmetric function. 

Then we construct lower triangular matrices $C^{\pm}$ as in \eqref{Cpm} with entries
\begin{align*} 
	c^+_{j+1,j-2k}   & = (-1)^{j-k} e'_{j-k,j} && \text{for } j = 1, \ldots, d \text{ and }  k = 0, 1, \ldots, \lfloor \tfrac{j-1}{2} \rfloor, \\
	c^+_{j+1,j+1-2k} & = (-1)^{k} e'_{k,j} &&  \text{for } j=1, \ldots, d \text{ and } k = 0, 1, \ldots, \lfloor \tfrac{j}{2} \rfloor, \\
  c^-_{2j+1,2j+1-2k} & = (-1)^k e_{k, 2j} && \text{for } j=1, \ldots \lfloor \tfrac{d}{2} \rfloor \text{ and } k = 1, \ldots, j, \\
	c^-_{2j+1,2j-2k}  & =  (-1)^{k+1} e_{2j-k, 2j} && \text{for } j= 1, \ldots, \lfloor \tfrac{d}{2} \rfloor \text{ and } k = 0, 1, \ldots, j-1, \\
	c^-_{2j,2j-1-2k } & =  (-1)^k e_{k  , 2j-1} && \text{for } j=1, \ldots, \lceil \tfrac{d}{2} \rceil \text{ and }  k = 0, 1, \ldots, j-1, \\
	c^-_{2j,2j-2k } & =  (-1)^{k} e_{2j-1-k  , 2j-1} && \text{for } j=1, \ldots, \lceil \tfrac{d}{2} \rceil \text{ and } k=1,2, \ldots, j-1.
\end{align*}
For example, for $d=5$ we have
\begin{align*} 
	C^+ & = \begin{pmatrix} 1 & 0 & 0 & 0 & 0 & 0\\
-e'_{1,1} & 1 & 0 & 0 & 0 & 0\\
-e'_{1,2} & e'_{2,2} & 1 & 0 & 0 & 0 \\
e'_{2,3} & -e'_{1,3} & -e'_{3,3} & 1 & 0 & 0 \\
e'_{2,4} & -e'_{3,4} & -e'_{1,4} & e'_{4,4} & 1 & 0 \\
-e'_{3,5} & e'_{2,5} & e'_{4,5} & -e'_{1,5} & -e'_{5,5} & 1
\end{pmatrix} \\ 
	C^- & =  \begin{pmatrix} 1 & 0 & 0 & 0 & 0 & 0 \\
e_{0,1} & 1 & 0 & 0 & 0 & 0 \\
-e_{1,2} & -e_{2,2} & 1 & 0 & 0 & 0 \\
-e_{1,3} & -e_{2,3} & e_{0,3} & 1 & 0 & 0 \\
e_{2,4} & e_{3,4} & -e_{1,4} & -e_{4,4} & 1 & 0 \\
e_{2,5} & e_{3,5} & -e_{1,5} & -e_{4,5} & e_{0,5} & 1 
\end{pmatrix}. 
\end{align*}
Then it is a tedious task to verify that the identities \eqref{eq:Cpm1} and \eqref{eq:Cpm2}
are indeed satisfied. The verification depends on the following properties of the quantities
$e_{k,n}$ and $e_{k,n}'$ defined in \eqref{eknandprime1}--\eqref{eknandprime2}:
\begin{enumerate}
\item[(i)] If $n$ is odd then $e_{k,n} = e'_{k,n}$  and $e_{n-k,n} = e_{k,n}$.
\item[(ii)] If $n$ is even then $e_{k,n} = \omega^{-n/2} e'_{n-k,n}$ and $e'_{k,n} = \omega^{n/2} e_{n-k,n}$.
\item[(iii)] For every $n$,
\begin{equation} \label{eknprop3} 
	e_{k,n+1} = e_{k,n} + \begin{cases} \omega^{n/2} e_{k-1,n}  &  \text{ if $n$ is even}, \\
					\omega^{-(n+1)/2} e_{k-1,n} & \text{ if $n$ is odd}. \end{cases}
	\end{equation}
\end{enumerate}
These identities are easy consequences of the definitions \eqref{eknandprime1} and \eqref{eknandprime2}.
For example, the identity in \eqref{eknprop3} relies on the fact that the elementary symmetric polynomial
$e_{k,n+1}(x_1, \ldots, x_{n+1})$ on $n+1$ variables is the sum of $e_{k,n}(x_1, \ldots, x_n)$
which has the terms without $x_{n+1}$, and $x_{n+1} e_{k-1,n}(x_1, \ldots, x_n)$
which has the terms containing $x_{n+1}$.

Now we are going to verify the identity \eqref{eq:Cpm1}. The identity \eqref{eq:Cpm1} for the
entries in rows $2j-1, 2j$ and columns $2k-1, 2k$
(where $1 \leq k \leq j \leq \lfloor \frac{d+1}{2} \rfloor$) is exactly
\[ \begin{pmatrix} 1 & 1 \\ 0 & 1 \end{pmatrix}
\begin{pmatrix}
c_{2j-1,2k-1}^{+} & c_{2j-1,2k}^{+} \\
c_{2j,2k-1}^{+} & c_{2j,2k}^{+}
\end{pmatrix}=\begin{pmatrix}
c_{2j-1,2k-1}^{-} & c_{2j-1,2k}^{-} \\
c_{2j,2k-1}^{-} & c_{2j,2k}^{-}
\end{pmatrix}
\begin{pmatrix} 0 & 1 \\ -1 & 0 \end{pmatrix},
\]
which when written in terms of the elementary symmetric polynomials is equivalent to
\begin{equation} \label{Cplusjk} 
	(-1)^{j+k} \begin{pmatrix} 1 & 1 \\ 0 & 1 \end{pmatrix}
	\begin{pmatrix} 	e'_{j-k,2j-2} & - e'_{j+k-1,2j-2} \\ - e'_{j+k-1,2j-1} & e'_{j-k,2j-1} \end{pmatrix}  
		= (-1)^{j+k} \begin{pmatrix} e_{j-k,2j-2} & e_{j+k-1,2j-2} \\ e_{j-k,2j-1} & e_{j+k-1,2j-1} \end{pmatrix}
			\begin{pmatrix} 0 & 1 \\ -1 & 0 \end{pmatrix}. 
			\end{equation}
The equality for the entries in the second row of \eqref{Cplusjk} follows from item (i) above, while
the equality for the entries in the first row relies on all three items. For example, the equality
for the first entry follows from (where  $n = 2j-2$ is even)
\begin{align*} 
	e'_{j-k,2j-2} - e'_{j+k-1,2j-1} & = e'_{j-k,n} - e'_{j+k-1,n+1} \\
	& = \omega^{n/2} e_{j+k-2,n} - e_{j+k-1,n+1}  && \text{by (ii) and (i)} \\
		& = - e_{j+k-1,n}  && \text{by (iii)}.
		\end{align*}
				
This argument establishes \eqref{eq:Cpm1} completely if $d$ is odd. If $d$ is even, then the last row needs
separate verification. The equality follows from 
\begin{equation} \label{eknprop4} 
 - e'_{d+1-k, d} = e_{k,d}, \qquad \text{for } k = 1, \ldots, d \text{ and $d$ is even.}
\end{equation}
To see that \eqref{eknprop4} indeed holds we calculate
\begin{align*} 
	e'_{d+1-k,d} + e_{k,d} & = \omega^{d/2} e_{k-1,d} + e_{k,d} && \text{by item (ii)} \\
	& = e_{k, d+1} && \text{by (iii)}.
\end{align*}
Now we recall that $\omega$ is a $(d+1)$-st root of unity, and by the definition \eqref{eknandprime1} of $e_{k,d+1}$,
\[ e_{k,d+1} = e_{k,d+1}(1, \omega, \omega^2, \ldots, \omega^{d}) = 0, \qquad k = 1, \ldots, d. \]
Thus \eqref{eknprop4} holds, and this gives the identity of the last row of \eqref{eq:Cpm1} in case
$d$ is even.

The proof of \eqref{eq:Cpm2} follows along similar lines.
\end{proof}

We return to the matrix-valued function $\mathbf{S}$ defined in \eqref{def:S:L3pm:1}--\eqref{def:S:L3pm:4}.
The coefficients $c_{j,k}^{\pm}$ are taken from the matrices \eqref{Cpm} whose existence
is proved in Lemma \ref{lemma:Cpmexist}. 

\begin{rhp}
The matrix $\mathbf{S}(z)$ is the solution to the following RH problem:
\begin{itemize}
\item[$\bullet$] $\mathbf{S}$ is analytic in $\mathbb{C}\setminus \Sigma_{S}$, where 
$\Sigma_{S}=\Sigma_{2}\cup\Sigma_{3}\cup\bigcup_{j=1}^{3}\partial L_{j}$.
\item[$\bullet$] $\mathbf{S}_{+}=\mathbf{S}_{-} \mathbf{J}_{S}$ on  $\Sigma_{S}$ with jump matrix 
\begin{align} \label{jump:S:Sigma1}
	\mathbf{J}_S & = \begin{cases} \diag(\ir \sigma_2, \ir \sigma_2, \ldots, \ir\sigma_2, 1) & \text{if $d$ is even} \\
		\diag(\ir \sigma_2, \ir \sigma_2, \ldots,  \ir \sigma_2) & \text{if $d$ is odd}
		\end{cases} && \text{ on } \Sigma_1^*, \\
		\label{jump:S:Sigma1hat}
	\mathbf{J}_S & = e^{2n \varphi_1} E_{1,2} + \begin{cases} \diag(1,1, \ir \sigma_2, \ldots, \ir\sigma_2, 1) & \text{if $d$ is even} \\
		\diag(1,1, \ir \sigma_2, \ldots,  \ir \sigma_2) & \text{if $d$ is odd}
		\end{cases} && \text{ on } \Sigma_1 \setminus \Sigma_1^*, \\
		\label{jump:S:Sigma3}
	\mathbf{J}_S & = \begin{cases} \diag(1,1, \ir \sigma_2, \ldots, \ir\sigma_2, 1) & \text{if $d$ is even} \\
		\diag(1,1, \ir \sigma_2, \ldots,  \ir \sigma_2) & \text{if $d$ is odd}
		\end{cases} && \text{ on } \Sigma_3 \setminus \Sigma_1, \\
		\label{jump:S:Sigma2}
	\mathbf{J}_S & = \begin{cases} \diag(1,\ir \sigma_2, \ldots, \ir\sigma_2) & \text{if $d$ is even} \\
		\diag(1,\ir \sigma_2, \ldots,  \ir \sigma_2,-1) & \text{if $d$ is odd}
		\end{cases} && \text{ on } \Sigma_2, 
\end{align}
and on the lips of the lenses,
\begin{align} 
	\label{jump:S:L1}
	\mathbf{J}_S & = I_{d+1} + e^{-2n \varphi_1} E_{2,1}  && \text{on } \partial L_1 \setminus L_2 \\
	\label{jump:S:L3}
	\mathbf{J}_S & = I_{d+1} + \sum_{j=2}^{\lceil \frac{d}{2} \rceil} e^{-2n \varphi_{2j-1}} E_{2j,2j-1} && \text{on } \partial L_3 \setminus L_2 \\
	\label{jump:S:L2}
	\mathbf{J}_S & = I_{d+1} + \sum_{j=1}^{\lfloor \frac{d}{2} \rfloor} \omega^{\mp j} e^{-2n \varphi_{2j}} E_{2j+1,2j} && 
	\text{on } \partial L_2^{\pm} \setminus L_1 \\
	\label{jump:S:lipsnear0}
	\mathbf{J}_S & = I_{d+1} + \sum_{j=2}^{d+1} \sum_{k=1}^{j-1} b_{j,k} e^{-2n(\varphi_{k} + \cdots + \varphi_{j-1})} E_{j,k} 
	&&  \begin{array}{l} \text{on } \partial L_1 \cap L_2 \\
		\text{and on }  \partial L_2 \cap L_1, \end{array}
		\end{align}
for certain constants $b_{j,k}$ (different constants on $\partial L_1 \cap L_2$ and $\partial L_2 \cap L_1$)
whose exact values are not important to us in what follows.
\item[$\bullet$] As $z\rightarrow\infty$ in the sector $S_{\ell}^{\pm}$, $\ell=0,\ldots,d,$
\begin{equation}\label{asymp:S}
\mathbf{S}(z)=\Big(I_{d+1} +O\left(z^{-\frac{2}{d}}\right) \Big)\begin{pmatrix} 1 & 0 \\ 0 & \mathbf{A}(z) 
\end{pmatrix}.
\end{equation}
\end{itemize}
\end{rhp}

\begin{proof} 
The jump matrices \eqref{jump:T:Sigma1}--\eqref{jump:T:Sigma3} in the RH problem for $\mathbf{T}$ have a block
diagonal form with $2 \times 2$ blocks
\[ \begin{pmatrix} e^{-2n \varphi_{k,+}} & 1\\  0 & e^{-2n \varphi_{k,-}} \end{pmatrix}. \]
The $2 \times 2$ block has the standard factorization
\[ \begin{pmatrix} 1 & 0 \\ e^{-2n \varphi_{k,-}} & 1 \end{pmatrix} \begin{pmatrix} 0 & 1 \\ -1 & 0 \end{pmatrix} 
\begin{pmatrix} 1 & 0 \\ e^{-2n \varphi_{k,+}} & 1 \end{pmatrix},
\]
since $\varphi_{k,+} = -\varphi_{k,-}$ on $\Sigma_k$, see e.g.\ \eqref{rel:g1phi1:1}--\eqref{rel:gkphik:2}.
The definition of $S$ in $L_3 \setminus L_2$, see \eqref{def:S:L3pm:1}--\eqref{def:S:L3pm:2}, has 
the effect of moving the outer two factors to the jump matrix on the lips
of the lenses while keeping the  inner factor on the original contour. This accounts for the
jump matrices \eqref{jump:S:L1}--\eqref{jump:S:L3} on the lips of the lenses, and for the jump
matrices \eqref{jump:S:Sigma1}--\eqref{jump:S:Sigma3} on the contour $\Sigma_3$ outside of the lens $L_2$. 
The jump matrix \eqref{jump:T:Sigma2} has $2 \times 2$ blocks
\[ \begin{pmatrix} \omega^{-k} e^{-2n \varphi_{2k,+}} & 1 \\ 0 & \omega^{k} e^{-2n \varphi_{2k,-}} \end{pmatrix} \]
which has the factorization
\[ \begin{pmatrix} 1 & 0 \\ \omega^k e^{-2n \varphi_{2k,-}} & 1 \end{pmatrix} \begin{pmatrix} 0 & 1 \\ -1 & 0 \end{pmatrix} 
\begin{pmatrix} 1 & 0 \\ \omega^{-k} e^{-2n \varphi_{2k,+}} & 1 \end{pmatrix}.
\]
Because of the definition \eqref{def:S:L3pm:3} of $S$ in $L_2 \setminus L_1$, we obtain \eqref{jump:S:Sigma2}
and \eqref{jump:S:L2}.

The form of the jump matrix \eqref{jump:S:lipsnear0} is immediate from the definitions.

\medskip
It remains to check the jumps on the parts of $\Sigma_1^*$ and $\Sigma_2$ that 
lie in the intersection of the lenses. 
To that end we write 
\begin{equation} \label{def:psik} 
	\psi_k = - \sum_{j=1}^{k-1} \varphi_j + \sum_{j=k}^{d} \varphi_j, \qquad k = 1, \ldots, d+1, 
	\end{equation}
where the first sum is $0$ if $k=1$ and the last sum is $0$ if $k=d+1$. Also put
\[ \mathbf{\Psi} = \diag(e^{\psi_1}, e^{\psi_2}, \cdots, e^{\psi_{d+1}}). \]
Then \eqref{def:S:L3pm:4} can be written as
\[ \mathbf{S} = \mathbf{T}  \mathbf{\Psi}^n C^{\pm} \mathbf{\Psi}^{-n}, \qquad \text{ in } L_1^{\pm} \cap L_2^{\mp}. \]

The jump matrix on $\Sigma_1^* \cap L_1 \cap L_2$ therefore is
\begin{equation} \label{eq:JS1} 
	\mathbf{J}_S = \mathbf{\Psi}^n_- \left(C^-\right)^{-1} \mathbf{\Psi}^{-n}_- \mathbf{J}_T \mathbf{\Psi}^n_+ C^+ \mathbf{\Psi}^{-n}_+.
	\end{equation}
Recall that $\varphi_{2k-1,+} = - \varphi_{2k-1,-}$ and $\varphi_{2k,+} = \varphi_{2k-1,-} + \varphi_{2k,-}
+ \varphi_{2k+1,-}$ on $\Sigma_1^*$, see \eqref{rel:g1phi1:1}--\eqref{rel:gkphik:2} and \eqref{rel:phikphikm1phikp1Sigma1}. 
Then if $d$ is odd, we can check from
the definition \eqref{def:psik} of $\psi_k$ that
\begin{equation} \label{jumps:psik}
\begin{aligned} 
	\psi_{2k,+} & = \psi_{2k,-} + \varphi_{1,-} +2 \varphi_{2k-1,-} - \varphi_{d,-} \\
	\psi_{2k-1,+} &=\psi_{2k-1,-}+\varphi_{1,-}-2\varphi_{2k-1,-}-\varphi_{d,-}\\
        \psi_{2k-1,+} & = \psi_{2k,-} + \varphi_{1,-} - \varphi_{d,-}  \\
	\psi_{2k,+} & = \psi_{2k-1,-} + \varphi_{1,-} - \varphi_{d,-}. 
	\end{aligned}
	\end{equation} 
It follows from these relations and \eqref{jump:T:Sigma1} that
\begin{equation} \label{eq:JS2} 
		\mathbf{\Psi}^{-n}_- \mathbf{J}_T \mathbf{\Psi}^n_+ =  e^{n(\varphi_{1,-}-\varphi_{d,-})} 
	    \left( I_{d+1} + \sum_{j=1}^{\lceil \frac{d}{2} \rceil} E_{2j-1,2j}  \right) 
			\end{equation}
which is a scalar function times a constant matrix. 
Then by one of the defining properties \eqref{eq:Cpm1} of $C^{\pm}$, we 
obtain from \eqref{eq:JS1} and \eqref{eq:JS2} 
\begin{equation} \label{eq:JS3} 
	\mathbf{J}_S =  e^{n(\varphi_{1,-}-\varphi_{d,-})}  \mathbf{\Psi}^n_- 
		\diag(\ir\sigma_2, \cdots, \ir\sigma_2) \mathbf{\Psi}^{-n}_+.
\end{equation}
Again using the jump properties \eqref{jumps:psik}  of $\psi_k$ we then see that 
\eqref{eq:JS3} reduces to \eqref{jump:S:Sigma1}.
The argument for $d$ even is similar. The main difference is that the term $-\varphi_{d,-}$
in \eqref{jumps:psik} is absent. 

\medskip
The jump matrix $\mathbf{J}_S$ on $\Sigma_2 \cap L_1 \cap L_2$ is calculated in a similar way. Now
we use the other property \eqref{eq:Cpm2} that defines the matrices $C^{\pm}$.

\medskip

The asymptotic condition \eqref{asymp:S} is obvious from \eqref{asymp:T} outside the lenses, since
$\mathbf{S} = \mathbf{T}$ outside the lenses. Furthermore, it is easy to deduce 
from \eqref{rel:gksphik} that the $\varphi$ functions satisfy
\begin{equation} \label{varphi:atinfty}
\begin{aligned} 
	\varphi_{2k}(z) & = d_{2k,\ell,\pm}\,z^{\frac{d+1}{d}}+O(\log z), && z\in L_{2}\cap S_{\ell}^{\pm}, 
		\quad k = 1, \ldots, \lceil \tfrac{d}{2} \rceil, 	\\
	\varphi_{2k-1}(z) & =d_{2k-1,\ell,\pm}\,z^{\frac{d+1}{d}}+O(\log z), && z\in L_{3}\cap S_{\ell}^{\pm},
		\quad k = 2, \ldots, \lfloor \tfrac{d}{2} \rfloor,
\end{aligned}
\end{equation}
as $z \to \infty$, where $d_{k,\ell,\pm}$ are constants such that 
$\Real \varphi_{2k}(z)>0$ ($\Real\varphi_{2k+1}(z)>0$) for $z$ large enough in $L_{2}$ ($L_{3}$). 
Therefore, from \eqref{asymp:T} and the definition \eqref{def:S:L3pm:2}--\eqref{def:S:L3pm:3}  
of $\mathbf{S}$ in the unbounded lenses, we deduce that the asymptotic condition 
is also valid inside the  lenses $L_{2}$ and $L_{3}$ (although not uniformly up to the sets $\Sigma_2$ and $\Sigma_3$).
\end{proof}

\begin{remark}
Note that the jump matrices \eqref{jump:S:Sigma1}, \eqref{jump:S:Sigma3} and \eqref{jump:S:Sigma2}
are constant. The jump matrix \eqref{jump:S:Sigma1hat} is not constant since it contains the
term $e^{2n \varphi_1} E_{1,2}$. However, by the discussion at the beginning of this section we have
\begin{equation} \label{estimate:Rephi:1}
	\Real \varphi_1 < 0 \qquad \text{ on } \Sigma_1 \setminus \Sigma_1^* 
	\end{equation}
so that the non-constant term is exponentially small if $n$ is large.

The jump matrices \eqref{jump:S:L1}--\eqref{jump:S:lipsnear0} on the lips of the lenses 
are exponentially close to the identity matrix, since the lenses where taken so that
\begin{align} \label{estimate:Rephi:2}
	\Real \varphi_1 > 0  & \quad \text{ on } \partial L_1 \\
	\label{estimate:Rephi:3}
	\Real \varphi_{2k-1} > 0 & \quad \text{ on } \partial L_3 \quad \text{for } k=2, \ldots, \lceil \tfrac{d}{2} \rceil \\
	\label{estimate:Rephi:4}
	\Real \varphi_{2k} > 0 & \quad \text{ on } \partial L_2 \quad \text{for } k=1, \ldots, \lfloor \tfrac{d}{2} \rfloor \\
	\label{estimate:Rephi:5}
	\Real \varphi_{k} \geq 0 & \quad \text{ on } L_1 \cap L_2 \quad \text{for  all } k.
	\end{align}
	
Also taking note of \eqref{varphi:atinfty} we may then conclude that
\begin{equation} \label{estimate:jumpS} 
	\mathbf{J}_S(z) = I_{d+1} + O \left(e^{-cn |z|^{\tfrac{d+1}{d}}} \right), 
	\qquad z \in (\Sigma_1 \setminus \Sigma_1^*) \cup \partial L_1 \cup \partial L_2 \cup \partial L_3, 
	\end{equation}
as $n \to \infty$, and the $O$-term is uniform for $z$ if we stay
away from the endpoints $\omega^{\ell} x^*$ of $\Sigma_1^*$.

\end{remark}

Ignoring all exponentially small entries in the jump matrices, we arrive at a new RH problem
which we study next.

\section{Global parametrix}\label{section:globalparam}

\begin{rhp}
Find $\mathbf{M}:\mathbb{C}\setminus \Sigma_M \to \mathbb C^{(d+1)\times(d+1)}$ where 
$\Sigma_M = \Sigma_{2}\cup\Sigma_{3}$ such
that:
\begin{itemize}
\item[$\bullet$] $\mathbf{M}$ is analytic in $\mathbb{C}\setminus \Sigma_M$.
\item[$\bullet$] $\mathbf{M}_{+}=\mathbf{M}_{-}\mathbf{J}_{M}$ on $\Sigma_M$, where
 \begin{align} \label{jump:M:Sigma1}
\mathbf{J}_M & = \begin{cases} \diag(\ir\sigma_2, \ir\sigma_2, \ldots, \ir\sigma_2,1) & \text{for $d$ even} \\
	\diag(\ir\sigma_2, \ir\sigma_2, \ldots, \ir\sigma_2) & \text{for $d$ odd}
	\end{cases} && \text{on } \Sigma_1^*, \\
	\label{jump:M:Sigma3}
	\mathbf{J}_M & =  \begin{cases} \diag(1,1, \ir\sigma_2, \ldots, \ir\sigma_2,1) & \text{for $d$ even} \\
	\diag(1,1, \ir\sigma_2, \ldots, \ir\sigma_2) & \text{for $d$ odd}
	\end{cases} && \text{on } \Sigma_3 \setminus \Sigma_1^*, \\
	\label{jump:M:Sigma2}
	\mathbf{J}_M & =  \begin{cases} \diag(1,\ir\sigma_2, \ir\sigma_2, \ldots, \ir\sigma_2) & \text{for $d$ even} \\
	\diag(1, \ir\sigma_2, \ir\sigma_2, \ldots, \ir\sigma_2,-1) & \text{for $d$ odd}
	\end{cases} && \text{on } \Sigma_2.
	\end{align}
\item[$\bullet$] $\mathbf{M}$ satisfies the asymptotic condition
\begin{equation} \label{asymp:M}
	\mathbf{M}(z)=\Big(I_{d+1}+O\Big(\frac{1}{z}\Big)\Big) \begin{pmatrix} 1 & 0 \\ 0 & \mathbf{A}(z) \end{pmatrix}
 \end{equation}
as $z\rightarrow\infty$ in the sector $S_{\ell}^{\pm}$, $\ell=0,\ldots,d$.
\item[$\bullet$] $\mathbf{M}(z)=O((z-\omega^{j}\,x^{*})^{-1/4})$ as $z\rightarrow\omega^{j}x^{*}$, $j=0,\ldots,d$.
\item[$\bullet$] $\mathbf{M}(z)$ remains bounded as $z\rightarrow 0$.
\end{itemize}
\end{rhp}

We solve this problem with the help of a meromorphic differential and certain meromorphic functions 
defined on the Riemann surface $\mathcal{R}$ given in \eqref{def:RiemannsurfR}. Let $\bm{\eta}$ be the meromorphic 
differential that has simple poles at the branch points $\omega^{j}\,x^{*}$, $j=0,\ldots,d,$ and a 
simple pole at $\infty_{2}$ (see definition of $\mathcal{R}$), with
\begin{equation}\label{residues}
\Res(\bm{\eta}, \omega^{j}\,x^{*})=-\frac{1}{2},\qquad \Res(\bm{\eta},\infty_{2})=\frac{d+1}{2},
\end{equation}
and is holomorphic elsewhere. This meromorphic differential exists and is uniquely determined by 
these conditions, since $\mathcal{R}$ has genus zero.

We now define
\[
u_{j}(z)=\int_{\infty_{1}}^{z}\bm{\eta},\qquad z\in\mathcal{R}_{j},\qquad j=1,\ldots,d+1,
\]
where $\infty_{1}$ is the point at infinity in $\mathcal{R}_{1}$ and the path of integration is taken so that 
it satisfies the following rules:
\begin{itemize}
\item[1)] The path for $u_{1}(z)$ stays on the first sheet.
\item[2)] The path for $u_{2}(z)$ starts on the first sheet, passes once through $\Sigma_{1}$ and stays on the 
second sheet. The passage is made via the $-$-side on $\mathcal{R}_{1}$.
\item[3)] The definition of the path for $u_{j}(z)$, $3\leq j\leq d$, is based on the following observation. 
Suppose that $z\in\mathcal{R}_{j}$, $j=3,\ldots,d,$ and $z$ lies in any of the $2d+2$ sectors $S_{\ell}^{\pm}$. 
Then there is only one neighboring sector in the upper sheet $\mathcal{R}_{j-1}$ to which the point $z$ can be 
connected through a path that crosses the cut connecting the sheets $\mathcal{R}_{j-1}$ and $\mathcal{R}_{j}$ 
only once.

We then require that the path for $u_{j}(z)$, $3\leq j\leq d$, in its passage from $\mathcal{R}_{1}$ to 
$\mathcal{R}_{2}$, crosses $\Sigma_{1}$ only once through the $-$-side in the upper sheet, and for each 
$2\leq k\leq j-1,$ the path should cross the cut $\Sigma_{k}$ connecting the sheets $\mathcal{R}_{k}$ and 
$\mathcal{R}_{k+1}$ only once. The observation above shows that this path is well-defined and clearly any 
two paths satisfying these requirements will give the same value for $u_{j}(z)$.
\item[4)] The path for $u_{d+1}(z)$ is defined as follows. As before, it is required to cross $\Sigma_{1}$ 
only once through the $-$-side in its passage from $\mathcal{R}_{1}$ to $\mathcal{R}_{2}$, and for 
$2\leq k\leq d-1$, it is required to cross the cut $\Sigma_{k}$ only once in its passage from $\mathcal{R}_{k}$ 
to $\mathcal{R}_{k+1}$. Finally, the path from $\mathcal{R}_{d}$ to $\mathcal{R}_{d+1}$ should cross 
$\Sigma_{d}$ only once through the $-$-side on the upper sheet.
\end{itemize}

These rules should be satisfied for any value of $d$, even or odd. It easily follows from these rules 
and the residue assumptions on the branch points $\omega^{j}\,x^{*}$ that the following relations hold 
in the case $d$ odd:
\begin{align}
u_{1,-} & =u_{2,+},  & \text{on } \Sigma_{1}^*,\label{bv:ufunc:1}\\
u_{1,+} & =u_{2,-}\pm\pi\ir, & \text{on } \Sigma_{1}^*,\label{bv:ufunc:2}\\
u_{2j-1,\pm} & =u_{2j,\mp},\quad 2\leq j \leq \tfrac{d-1}{2}, & \text{on } \Sigma_{3},\label{bv:ufunc:3}\\
u_{d+1,+} & = u_{d,-}, & \text{on }  \Sigma_{3},\label{bv:ufunc:4}\\
u_{d+1,-} & = u_{d,+}\pm\pi\ir, & \text{on } \Sigma_{3},\label{bv:ufunc:5}\\
u_{j,+} & = u_{j,-},\quad j=1, d+1, & \text{on } \Sigma_{2},\label{bv:ufunc:6}\\
u_{2j,\pm} & = u_{2j+1,\mp}, \quad 1\leq j\leq \tfrac{d-1}{2}, & \text{on } \Sigma_{2},\label{bv:ufunc:7}
\end{align}
and in the case $d$ even, the relations \eqref{bv:ufunc:1}--\eqref{bv:ufunc:2} still hold, and we now have
\begin{align}
u_{2j-1,\pm} & =u_{2j,\mp},\quad 2\leq j\leq \tfrac{d}{2}, & \text{on } \Sigma_{3},\label{bv:ufunc:8}\\
u_{d+1,+} & = u_{d+1,-}, & \text{on } \Sigma_{3},\label{bv:ufunc:9}\\
u_{1,+} & = u_{1,-}, & \text{on } \Sigma_{2},\label{bv:ufunc:10}\\
u_{2j,\pm} & = u_{2j+1,\mp}, \quad 1\leq j\leq \tfrac{d}{2}-1 & \text{on } \Sigma_{2}, \label{bv:ufunc:11}\\
u_{d+1,+} & = u_{d,-}, & \text{on } \Sigma_{2}, \label{bv:ufunc:12}\\
u_{d+1,-} & = u_{d,+}\pm\pi\ir, & \text{on } \Sigma_{2}.\label{bv:ufunc:13}
\end{align}

Let us define now the functions $v_{j}(z):=e^{u_{j}(z)}$, $j=1,\ldots,d+1,$ and set
\[
\mathbf{v}(z):=(v_{1}(z)\hspace{0.3cm} v_{2}(z) \hspace{0.3cm} \ldots \hspace{0.3cm} v_{d+1}(z)).
\]
Then the above relations \eqref{bv:ufunc:1}--\eqref{bv:ufunc:13} give in the case $d$ is odd,
\begin{align*}
\mathbf{v}_{+}(z) & =\mathbf{v}_{-}(z)\,
\diag(\ir\sigma_2, \sigma_1, \ldots, \sigma_1, \ir\sigma_2), && z \in \Sigma_1^{*}, \\
\mathbf{v}_{+}(z) & =\mathbf{v}_{-}(z)\,
\diag(1,1, \sigma_1, \ldots, \sigma_1, \ir\sigma_2), && z \in \Sigma_3 \setminus \Sigma_1^{*}, \\
\mathbf{v}_{+}(z) & = \mathbf{v}_{-}(z)\,
\diag(1, \sigma_1, \cdots, \sigma_1, 1), && z \in \Sigma_2,
\end{align*}
and in the case $d$ is even,
\begin{align*}
\mathbf{v}_{+}(z) & =\mathbf{v}_{-}(z)\,
\diag(\ir\sigma_2, \sigma_1, \ldots, \sigma_1, 1), && z \in \Sigma_1^{*}, \\
\mathbf{v}_{+}(z) & =\mathbf{v}_{-}(z)\,
\diag(1,1, \sigma_1, \ldots, \sigma_1, 1), && z \in \Sigma_3 \setminus \Sigma_1^{*}, \\
\mathbf{v}_{+}(z) & = \mathbf{v}_{-}(z)\,
\diag(1, \sigma_1, \cdots, \sigma_1, \ir \sigma_2), && z \in \Sigma_2.
\end{align*}
Here $\sigma_1 = \begin{pmatrix} 0 & 1 \\ 1 & 0 \end{pmatrix}$ and $\ir\sigma_2 = \begin{pmatrix} 0 & 1 \\ -1 & 0 \end{pmatrix}$.

We next introduce the functions
\begin{align*}
\widehat{v}_{1} & :=v_{1}, \\
\widehat{v}_{2j} & :=v_{2j}, && 1 \leq j\leq \lfloor \tfrac{d}{2} \rfloor,\\
\widehat{v}_{2j+1}& := \begin{cases} 
    -v_{2j+1} & \text{on } \bigcup_{\ell=0}^{d} S_{\ell}^{+},\\[0.5em]
    v_{2j+1} & \text{on } \bigcup_{\ell=0}^{d} S_{\ell}^{-},
    \end{cases} 
    &&  1\leq j\leq \lfloor \tfrac{d-1}{2} \rfloor, \\
\widehat{v}_{d+1}& :=\begin{cases} 
    v_{d+1} & \text{on } \bigcup_{\ell=0}^{d} S_{\ell}^{+},\\[0.5em]
    -v_{d+1} & \text{on } \bigcup_{\ell=0}^{d} S_{\ell}^{-},
    \end{cases} && \text{if $d$ is odd}, \\
\widehat{v}_{d+1} & := v_{d+1} && \text{if $d$ is even}.		
\end{align*}
If we let
\[
\widehat{\mathbf{v}}(z):=
	\begin{pmatrix} \widehat{v}_{1}(z) & \widehat{v}_{2}(z) &  \cdots & \widehat{v}_{d+1}(z) \end{pmatrix}
\]
then we readily see that $\widehat{\mathbf{v}}$ is analytic in $\mathbb{C}\setminus(\Sigma_{2}\cup\Sigma_{3})$, and satisfies
\begin{equation}\label{jump:vhat}
	\widehat{\mathbf{v}}_{+}=\widehat{\mathbf{v}}_{-}\,\mathbf{J}_{M},\qquad \text{on }  \Sigma_{2}\cup\Sigma_{3}.
\end{equation}

We now explicitly construct the sought function $\mathbf{M}$ in terms of the 
functions $\widehat{v}_{j}$ and certain meromorphic functions on $\mathcal{R}$. Let
\[
f^{(0)}\equiv 1, f^{(1)}, f^{(2)},\ldots, f^{(d)},
\]
be a basis of the vector space of all meromorphic functions that have a pole of order at 
most $d$ at $\infty_{2}$ and are holomorphic elsewhere in $\mathcal{R}$. We choose these 
functions so that $f^{(j)}$ has a pole of order $j$ at $\infty_{2}$. By $f_{i}^{(j)}$ 
we denote the restriction of $f^{(j)}$ to the sheet $\mathcal{R}_{i}$. Then we set
\begin{equation}\label{def:matrixB}
\mathbf{B}:=\begin{pmatrix}
1 & 1 & 1 & \ldots & 1 \\[0.3em]
f_{1}^{(1)} & f^{(1)}_{2} & f_{3}^{(1)}  & \cdots & f_{d+1}^{(1)} \\[0.3em]
f_{1}^{(2)} & f^{(2)}_{2} & f^{(2)}_{3} & \cdots & f_{d+1}^{(2)} \\[0.3em]
\vdots & \vdots & \vdots & \ddots & \vdots \\[0.3em]
f_{1}^{(d)} & f_{2}^{(d)} & f_{3}^{(d)} & \cdots & f_{d+1}^{(d)}
\end{pmatrix} 
	\diag \begin{pmatrix} \widehat{v}_{1} & \widehat{v}_{2} & \cdots & \widehat{v}_{d+1} \end{pmatrix}.
\end{equation}
In virtue of \eqref{jump:vhat} we also obtain
\begin{equation}\label{jump:B}
\mathbf{B}_{+} = \mathbf{B}_{-}\,\mathbf{J}_{M},\qquad \text{on } \Sigma_{2}\cup\Sigma_{3}.
\end{equation}

By comparing \eqref{jump:M:Sigma3}--\eqref{jump:M:Sigma2} with  \eqref{jump:A:Sigma3}--\eqref{jump:A:Sigma2}
we see that  
\[ \mathbf{J}_M = \begin{pmatrix} 1 & 0 \\ 0 & \mathbf{J}_A \end{pmatrix} \qquad 
\text{ on } (\Sigma_3 \setminus \Sigma_1^* ) \cup \Sigma_2. \]
Thus 
\begin{equation} \label{jump:A} 
	\begin{pmatrix} 1 & 0 \\ 0 & \mathbf{A}_+(z) \end{pmatrix} = 
	\begin{pmatrix} 1 & 0 \\ 0 & \mathbf{A}_-(z) \end{pmatrix} \mathbf{J}_M
	\qquad \text{on } (\Sigma_3 \setminus \Sigma_1^* ) \cup \Sigma_2. 
	\end{equation}
From \eqref{jump:B} and \eqref{jump:A} we deduce that the function 
$\mathbf{B}(z) \begin{pmatrix} 1 & 0 \\ 0 & \mathbf{A}^{-1}(z) \end{pmatrix}$ 
extends to an analytic function in $\mathbb{C}\setminus\Sigma_{1}^{*}$, and therefore has a Laurent
expansion at infinity. 

We have $\widehat{v}_{1}(z)=1+O(1/z)$ as $z\rightarrow\infty$, 
and in virtue of \eqref{residues}, we have  $\widehat{v}_{j}(z)=O(z^{-\frac{d+1}{2d}})$ as 
$z\rightarrow\infty$ for $j \geq 2$. 
As $f^{(j)}$ has a pole of order $j$ at $\infty_2$, we obtain
$f^{(j)}_{1}(z)=O(1)$ for all $j$, and $f^{(j)}_{i}(z)=O(z^{\frac{j}{d}})$ for $i \geq 2$. 
Thus
\[ \mathbf{B}(z) = O(z^{\frac{d-1}{2d}}) \]
Also $\mathbf{A}^{-1}(z) = O(z^{\frac{d-1}{2d}})$ as $z \to \infty$, so that
\[ \mathbf{B}(z) \begin{pmatrix} 1 & 0 \\ 0 & \mathbf{A}^{-1}(z) \end{pmatrix} = O(z^{\frac{d-1}{d}})  
	\qquad \text{ as } z \to \infty. \]
Since $\frac{d-1}{d} < 1$ we see that the Laurent expansion does not have any terms with
strictly positive powers of $z$. Hence
\begin{equation}\label{eq:ABC}
\mathbf{B}(z) \begin{pmatrix} 1 & 0 \\ 0 & \mathbf{A}^{-1}(z) \end{pmatrix} = 
	C+O\Big(\frac{1}{z}\Big),\qquad z\rightarrow\infty,
\end{equation}
for some constant matrix $C$.

The matrix $C$ is easily seen to be invertible, as the functions $f^{(j)}$ are linearly
independent. Then
\begin{equation}\label{formulaGP} 
\mathbf{M}(z) = C^{-1} \mathbf{B}(z)
\end{equation}
satisfies all the conditions in the RH problem for $\mathbf{M}$. 

The $(1,1)$ entry of $\mathbf{M}$ is of special interest, since it appears 
in the asymptotic formula \eqref{strongasympform}. This entry is given by
\begin{equation}\label{formulaM11}
\mathbf{M}_{1,1}(z)=\exp(u_{1}(z))=\exp\left(\int_{\infty_{1}}^{z}\bm{\eta}\right),\qquad z\in\mathbb{C}\setminus\Sigma_{1}^{*},
\end{equation}
where $\bm{\eta}$ is the meromorphic differential defined above. To see this, observe that according to \eqref{formulaGP} and \eqref{def:matrixB}, 
\[
\mathbf{M}_{1,1}(z)=\widehat{v}_{1}(z) \sum_{i=1}^{d+1} (C^{-1})_{1,i}\,f^{(i-1)}_{1}(z)=e^{u_{1}(z)} \sum_{i=1}^{d+1} (C^{-1})_{1,i}\,f^{(i-1)}_{1}(z),
\]
where $(C^{-1})_{1,i}$ denotes the $(1,i)$ entry of $C^{-1}$. Since $\widehat{v}_{j}(z)=O(z^{-\frac{d+1}{2d}})$ as $z\rightarrow\infty$ for $j\geq 2$, 
it follows from the expression of $\mathbf{A}^{-1}(z)$ and \eqref{eq:ABC} that the entries $(1,i)$, $i\geq 2$, of $C$ are all zero. 
This implies that $(C^{-1})_{1,i}=0$ for all $i\geq 2$, hence $\mathbf{M}_{1,1}(z)=c\, e^{u_{1}(z)}$ for some constant $c$. By \eqref{asymp:M} we have $\mathbf{M}_{1,1}(z)=1+O(z^{-1})$, so $c=1$ and \eqref{formulaM11} follows.

\section{The final transformation $\mathbf{S}\mapsto \mathbf{R}$ and 
conclusion of the steepest descent analysis}\label{section:conclusion}

In this section we introduce the final transformation of the RH problem 
and give the proof of Theorem \ref{theo:strongasymp}. This final transformation makes
use of the global parametrix from the previous section and of a local parametrix
$\mathbf{P}^{(\textrm{Airy})}$ involving Airy functions that is defined on small disks 
around the endpoints $\omega^{\ell}\,x^{*}$, $\ell = 0, \ldots, d$ of the star $\Sigma_{1}^*$.
We define disks
\[
D(\omega^{\ell}\,x^{*},\delta)=\{z\in\mathbb{C} \mid |z-\omega^{\ell} x^{*}|<\delta\},
\]
where $\delta>0$ is taken sufficiently small. In any case, we want $\delta < \widehat{x}-x^*$
and the disks should be contained in the lense $L_3$.
Recall that  $t_{0}<t_{0,\crit}$, 
and therefore we know by Lemma \ref{lemma:propmu1} that the density of $\mu_{1}^{*}$ vanishes 
as a square root at the endpoints of $\Sigma_{1}^*$. This property and the fact that the 
RH problem is locally of size $2\times 2$ allow us to construct in a standard way a 
function $\mathbf{P}^{(\textrm{Airy})}$ that is the solution to the following problem.

\begin{rhp}\label{RHPlocalparam}
\begin{itemize}
\item $\mathbf{P}^{(\textrm{Airy})}$ is continuous on 
$\Big(\bigcup_{\ell=0}^{d} \overline{D(\omega^{\ell} x^{*},\delta)}\Big)\setminus 
\Sigma_{S}$ and is analytic in its interior $\Big(\bigcup_{\ell=0}^{d} D(\omega^{\ell}\,x^{*},\delta)\Big)\setminus \Sigma_{S}$.

\item $\mathbf{P}^{(\textrm{Airy})}_{+}=\mathbf{P}^{(\textrm{Airy})}_{-}\,\mathbf{J}_{S}$ on $\Sigma_{S}\cap \bigcup_{\ell=0}^{d} D(\omega^{\ell}\,x^{*},\delta)$, where $\mathbf{J}_{S}$ is the jump matrix for $\mathbf{S}$ as given
in \eqref{jump:S:Sigma1}, \eqref{jump:S:Sigma1hat}, \eqref{jump:S:L1}.

\item $\mathbf{P}^{(\textrm{Airy})}$ matches with the global parametrix $\mathbf{M}$ in the sense that
\begin{equation}\label{matching}
\mathbf{P}^{(\textrm{Airy})}(z)=\mathbf{M}(z)(I_{d+1}+O(n^{-1})),
\end{equation}
uniformly for $z\in \bigcup_{\ell=0}^{d} \partial D(\omega^{\ell}\,x^{*},\delta)$.
\end{itemize}
\end{rhp}

We omit here the construction of $\mathbf{P}^{(\textrm{Airy})}$ in terms of Airy functions. 
The details in the $2 \times 2$ case can be found in \cite{Deift}. 

We now define the final transformation $\mathbf{S}\mapsto \mathbf{R}$.

\begin{definition}
We define the matrix-valued function $\mathbf R: \mathbb C \setminus (\Sigma_S \cup \bigcup_{\ell} \partial D(\omega^{\ell} x^*, \delta))
\to \mathbb C^{(d+1) \times (d+1)}$ by
\begin{equation}\label{eq:defR}
\mathbf{R}(z):=\begin{cases}
\mathbf{S}(z)\,(\mathbf{P}^{(\textrm{Airy})})^{-1}(z), & \text{in the disks } D(\omega^{\ell} x^*, \delta), \\
\mathbf{S}(z)\,\mathbf{M}(z)^{-1}, & \text{outside the disks}.
\end{cases}
\end{equation}
\end{definition}

Since the jump matrices of $\mathbf{S}$  and $\mathbf{M}$ agree on $\Sigma_{1}^*$, $\Sigma_{2}$
and on $\Sigma_{3}\setminus\Sigma_{1}$, we see from \eqref{eq:defR} that $\mathbf{R}$ has analytic 
continuation across $\Sigma_{2}$ and $\Sigma_{3}\setminus\Sigma_{1}$ and across the part of $\Sigma_{1}^*$ outside 
the disks. The jump matrices of $\mathbf{S}$ and $\mathbf{P}^{(\textrm{Airy})}$  agree inside
the disks, and therefore $\mathbf{R}$ can be analytically 
continued inside the disks. The result is that $R$ is analytic in $\mathbb{C}\setminus\Sigma_{R}$, where $\Sigma_{R}$ 
is a system of contours as shown in Figure~\ref{fig:contourGammaR} for the case $d=3$. 
Thus $\Sigma_{R}$ consists of the intervals $\omega^{\ell} [x^*+\delta, \widehat{x}]$ for $\ell= 0, \ldots, d$,
the part of the boundary $\partial L_1$ of the lens $L_1$ that is outside of the disks, 
the full boundaries $\partial L_2$ and $\partial L_3$ of the other lenses, and the circles $\partial D(\omega^{\ell} x^*, \delta)$.  
These circles are given the positive orientation in $\Sigma_{R}$.

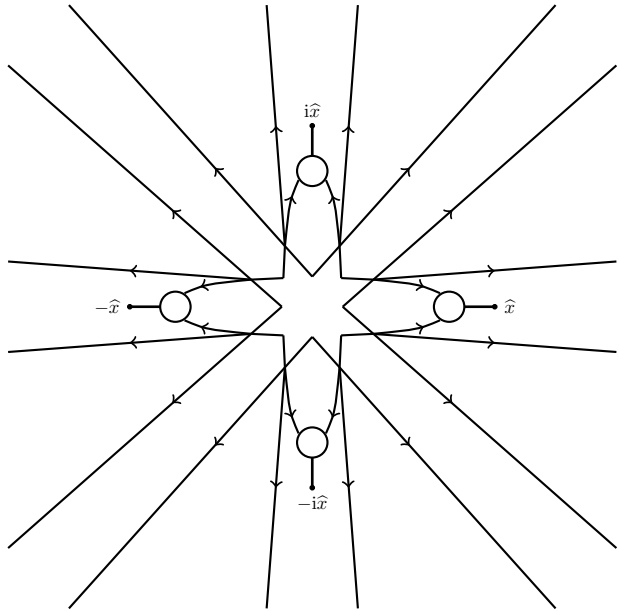
\begin{figure}[h!]
\begin{center}
\begin{tikzpicture}[scale = 2,line width = .5]
\draw [thick,postaction = decorate, decoration = {markings, mark = at position .4 with {\arrow[black]{>};}}] (0.2,0) -- (2,1.6);
\draw [thick,postaction = decorate, decoration = {markings, mark = at position .4 with {\arrow[black]{>};}}] (0.2,0) -- (2,-1.6);
\draw [rotate=90,thick,postaction = decorate, decoration = {markings, mark = at position .4 with {\arrow[black]{>};}}] (0.2,0) -- (2,1.6);
\draw [rotate=90,thick,postaction = decorate, decoration = {markings, mark = at position .4 with {\arrow[black]{>};}}] (0.2,0) -- (2,-1.6);
\draw [rotate=180,thick,postaction = decorate, decoration = {markings, mark = at position .4 with {\arrow[black]{>};}}] (0.2,0) -- (2,1.6);
\draw [rotate=180,thick,postaction = decorate, decoration = {markings, mark = at position .4 with {\arrow[black]{>};}}] (0.2,0) -- (2,-1.6);
\draw [rotate=270,thick,postaction = decorate, decoration = {markings, mark = at position .4 with {\arrow[black]{>};}}] (0.2,0) -- (2,1.6);
\draw [rotate=270,thick,postaction = decorate, decoration = {markings, mark = at position .4 with {\arrow[black]{>};}}] (0.2,0) -- (2,-1.6);
\draw [thick, postaction = decorate, decoration = {markings, mark = at position .5 with {\arrow[black]{>};}}] (0.4,0.18) -- (2,0.3);
\draw [thick, smooth, postaction = decorate, decoration = {markings, mark = at position .75 with {\arrow[black]{>};}}] (0.4,0.18).. controls (0.7,0.15) .. (0.84,0.09);
\draw [thick, smooth, postaction = decorate, decoration = {markings, mark = at position .75 with {\arrow[black]{>};}}] (0.4,-0.18).. controls (0.7,-0.15) .. (0.84,-0.09);
\draw [rotate=90,thick, smooth, postaction = decorate, decoration = {markings, mark = at position .75 with {\arrow[black]{>};}}] (0.4,0.18).. controls (0.7,0.15) .. (0.84,0.09);
\draw [rotate=90,thick, smooth, postaction = decorate, decoration = {markings, mark = at position .75 with {\arrow[black]{>};}}] (0.4,-0.18).. controls (0.7,-0.15) .. (0.84,-0.09);
\draw [rotate=180,thick, smooth, postaction = decorate, decoration = {markings, mark = at position .75 with {\arrow[black]{>};}}] (0.4,0.18).. controls (0.7,0.15) .. (0.84,0.09);
\draw [rotate=180,thick, smooth, postaction = decorate, decoration = {markings, mark = at position .75 with {\arrow[black]{>};}}] (0.4,-0.18).. controls (0.7,-0.15) .. (0.84,-0.09);
\draw [rotate=270,thick, smooth, postaction = decorate, decoration = {markings, mark = at position .75 with {\arrow[black]{>};}}] (0.4,0.18).. controls (0.7,0.15) .. (0.84,0.09);
\draw [rotate=270,thick, smooth, postaction = decorate, decoration = {markings, mark = at position .75 with {\arrow[black]{>};}}] (0.4,-0.18).. controls (0.7,-0.15) .. (0.84,-0.09);
\draw [thick] (0.9,0) circle (0.1);
\draw [thick] (0,0.9) circle (0.1);
\draw [thick] (-0.9,0) circle (0.1);
\draw [thick] (0,-0.9) circle (0.1);
\filldraw [black] (1.2,0) circle (0.4pt);
\filldraw [rotate=90,black] (1.2,0) circle (0.4pt);
\filldraw [rotate=180,black] (1.2,0) circle (0.4pt);
\filldraw [rotate=270,black] (1.2,0) circle (0.4pt);
\draw (1.3,0) node[scale=0.7]{$\widehat{x}$};
\draw[rotate=90] (1.3,0) node[scale=0.7]{$\ir\widehat{x}$};
\draw[rotate=180] (1.35,0) node[scale=0.7]{$-\widehat{x}$};
\draw[rotate=270] (1.3,0) node[scale=0.7]{$-\ir\widehat{x}$};
\draw [line width=1] (1,0) -- (1.2,0);
\draw [line width=1] (0,1) -- (0,1.2);
\draw [line width=1] (-1.2,0) -- (-1,0);
\draw [line width=1] (0,-1) -- (0,-1.2);
\draw [thick,smooth] (0.19,0.19) -- (0.4,0.18);
\draw [thick,smooth] (0.19,-0.19) -- (0.4,-0.18);
\draw [thick,smooth, rotate=90] (0.19,0.19) -- (0.4,0.18);
\draw [thick,smooth, rotate=90] (0.19,-0.19) -- (0.4,-0.18);
\draw [thick,smooth,rotate=180] (0.19,0.19) -- (0.4,0.18);
\draw [thick,smooth,rotate=180] (0.19,-0.19) -- (0.4,-0.18);
\draw [thick,smooth,rotate=270] (0.19,0.19) -- (0.4,0.18);
\draw [thick,smooth,rotate=270] (0.19,-0.19) -- (0.4,-0.18);
\draw [thick,postaction = decorate, decoration = {markings, mark = at position .5 with {\arrow[black]{>};}}] (0.4,-0.18) -- (2,-0.3);
\draw [thick, rotate=90,postaction = decorate, decoration = {markings, mark = at position .5 with {\arrow[black]{>};}}] (0.4,0.18) -- (2,0.3);
\draw [thick, rotate=90,postaction = decorate, decoration = {markings, mark = at position .5 with {\arrow[black]{>};}}] (0.4,-0.18) -- (2,-0.3);
\draw [thick, rotate=180,postaction = decorate, decoration = {markings, mark = at position .5 with {\arrow[black]{>};}}] (0.4,0.18) -- (2,0.3);
\draw [thick, rotate=180,postaction = decorate, decoration = {markings, mark = at position .5 with {\arrow[black]{>};}}] (0.4,-0.18) -- (2,-0.3);
\draw [thick, rotate=270,postaction = decorate, decoration = {markings, mark = at position .5 with {\arrow[black]{>};}}] (0.4,0.18) -- (2,0.3);
\draw [thick, rotate=270,postaction = decorate, decoration = {markings, mark = at position .5 with {\arrow[black]{>};}}] (0.4,-0.18) -- (2,-0.3);
\end{tikzpicture}
\end{center}
\caption{The contour $\Sigma_{R}$ in the case $d=3$.}
\label{fig:contourGammaR}
\end{figure}

We immediately observe from \eqref{eq:defR} that $\mathbf{R}$ satisfies the following RH problem.

\begin{rhp}\label{RHPforR}
\begin{itemize}
\item The matrix $\mathbf{R}$ is analytic in $\mathbb{C}\setminus\Sigma_{R}$.
\item  We have $\mathbf{R}_{+} =\mathbf{R}_{-} \,\mathbf{J}_{R}$ on $\Sigma_{R}$, where
\begin{equation} \label{jump:JR}
\mathbf{J}_{R}(z)=\begin{cases}
\mathbf{M}(z)\,\mathbf{P}^{(\textrm{Airy})}(z)^{-1}, & \text{for } z\in\bigcup\limits_{\ell=0}^{d} \partial D(\omega^{\ell} x^{*},\delta),\\
\mathbf{M}_{-}(z)\,\mathbf{J}_{S}(z)\,\mathbf{M}_{+}^{-1}(z), & \text{for } z\in\Sigma_{R}\setminus 
	\left(\bigcup\limits_{\ell=0}^{d} \partial D(\omega^{\ell} x^{*},\delta)\right).
\end{cases}
\end{equation}
\item $\mathbf{R}(z)=I_{d+1}+O(z^{-1})$ as $z\rightarrow\infty$.
\end{itemize}
\end{rhp}

Note that the asymptotic conditions \eqref{asymp:S} and \eqref{asymp:M} for $\mathbf{S}$ and $\mathbf{M}$ both contain
$\begin{pmatrix} 1 & 0 \\ 0 & A(z) \end{pmatrix}$, and this is cancelled in the asymptotic condition for $\mathbf{R}$.
Then one gets $\mathbf{R}(z)=I_{d+1}+O(z^{-\frac{2}{d}})$ as $z\to \infty$. However, since all jump matrices
in the RH problem for $\mathbf{R}$ are exponentially close to the identity matrix for large $z$, see \eqref{estimate:jumpR:2}, 
it can be shown that the $O$-term improves to $O(z^{-1})$.

The important feature of the RH problem associated with $\mathbf{R}$ is that the jump 
matrix $\mathbf{J}_{R}(z)$ tends to the identity matrix as $n\rightarrow\infty$, provided that the 
point $\widehat{x}>x^{*}$ is chosen so that it lies in the region where $\Real \varphi_{1}(z)<0$. Indeed, 
from \eqref{matching} and \eqref{jump:JR} we get
\begin{equation}\label{estimate:jumpR:1}
\mathbf{J}_{R}(z)=I_{d+1}+O(n^{-1}),\qquad z\in\bigcup_{\ell=0}^{d}\partial D(\omega^{\ell} x^{*},\delta),
\end{equation}
and on the remaining parts of $\Sigma_{R}$ we have convergence to the identity matrix exponentially fast, specifically
\begin{equation}\label{estimate:jumpR:2}
	\mathbf{J}_{R}(z)=I_{d+1}+O \left(e^{-c n |z|^{\frac{d+1}{d}}}\right),
	\qquad \text{for $z$ elsewhere on $\Sigma_{R}$},
\end{equation}
for some constant $c>0$. The estimate \eqref{estimate:jumpR:2} follows from \eqref{jump:JR} and
the corresponding estimate \eqref{estimate:jumpS}  for $\mathbf{J}_S(z)$, which is valid
uniformly on the parts of $\Sigma_R$ that come from $\Sigma_1 \setminus \Sigma_1^*$, and from the lips of the lenses.

\medskip

We can now conclude the steepest descent analysis. From \eqref{estimate:jumpR:1}--\eqref{estimate:jumpR:2} 
we deduce several important consequences. First, it follows that the RH problem \ref{RHPforR} has a unique 
solution if $n$ is large enough. The solution $\mathbf{R}$ can be expressed as a Neumann series. Since the transformations 
performed during the steepest descent analysis
\begin{equation}\label{transformations}
\mathbf{Y}\mapsto \mathbf{X}\mapsto \mathbf{U}\mapsto \mathbf{T}\mapsto \mathbf{S}\mapsto \mathbf{R}
\end{equation}
are all invertible, this implies that a solution $\mathbf{Y}$ to the original RH problem \ref{RHPforY} uniquely exists. 
In particular, the polynomial $P_{n,n}$, which is the $(1,1)$ entry of $\mathbf{Y}$, uniquely exists. We also 
deduce from \eqref{estimate:jumpR:1}--\eqref{estimate:jumpR:2} that the solution $R$ to the RH problem \ref{RHPforR} 
is itself close to the identity as $n\rightarrow\infty$. In fact, it follows from
\eqref{estimate:jumpR:1} and \eqref{estimate:jumpR:2} that $R$ satisfies
\begin{equation}\label{estimate:jumpR:3}
	\mathbf{R}(z)=I_{d+1}+O\Big(\frac{1}{n (1+|z|)}\Big), \qquad \text{as } n \to \infty,
\end{equation}
uniformly for $z\in\mathbb{C}\setminus\Sigma_{R}$. 


\subsection{Proof of Theorem \ref{theo:strongasymp}}

We unravel the transformations \eqref{transformations} in order to express the 
polynomial $P_{n,n}$ in terms of the matrix-valued function $\mathbf{R}$ and then we use \eqref{estimate:jumpR:3}.

We start from the relation $P_{n,n}(z)=\mathbf{Y}_{1,1}(z)$ that we know from Lemma \ref{lemma:RHcharact}. 
From \eqref{eq:def:X}, \eqref{def:U} and \eqref{def:T} we easily obtain
\[
P_{n,n}(z)= \mathbf{X}_{1,1}(z)=\mathbf{U}_{1,1}(z)=\mathbf{T}_{1,1}(z)\,e^{n g_{1}(z)},\qquad z\in\mathbb{C}\setminus\Sigma_{1}^*.
\]
Next, by definition of $\mathbf{S}$ in Definition \ref{def:Sz} we know that $\mathbf{S}(z)=\mathbf{T}(z)$ for 
all $z\in\mathbb{C}\setminus(L_{2}\cup L_{3})$, and so certainly $\mathbf{S}_{1,1}(z) = \mathbf{T}_{1,1}(z)$
for those $z$. But from \eqref{def:S:L3pm:2} and \eqref{def:S:L3pm:3} we see that $\mathbf{S}_{1,1}(z) = \mathbf{T}_{1,1}(z)$
also for $z \in (L_2 \cup L_3) \setminus L_1$. Hence
\begin{equation}\label{rel:P:S}
P_{n,n}(z)=\mathbf{S}_{1,1}(z)\,e^{n g_{1}(z)},\qquad z\in\mathbb{C}\setminus L_1.
\end{equation}
On account of \eqref{eq:defR}, we have $\mathbf{S}=\mathbf{R} \mathbf{M}$ outside of the disks  
$D(\omega^{\ell} x^{*},\delta)$, so using \eqref{estimate:jumpR:3} we obtain
\[
\mathbf{S}_{1,1}(z)=(1+O(n^{-1}))\,\mathbf{M}_{1,1}(z)+O(n^{-1})
\]
for such $z$. Since $\mathbf{M}_{1,1}(z)$ is an analytic function with no zeros 
in $\mathbb{C}\setminus\Sigma_{1}^{*}$, the above estimate can be rewritten as
\begin{equation}\label{estimate:S11}
\mathbf{S}_{1,1}(z)=(1+O(n^{-1}))\,\mathbf{M}_{1,1}(z),\qquad z\in\mathbb{C}\setminus 
\Big(\bigcup_{\ell=0}^{d} D(\omega^{\ell} x^{*},\delta)\Big),
\end{equation}
uniformly for $z$ in the indicated set.

Finally, from \eqref{rel:P:S} and \eqref{estimate:S11} we obtain \eqref{strongasympform} uniformly for 
$z\in\mathbb{C}\setminus \left( L_1 \cup \bigcup\limits_{\ell=0}^{d} D(\omega^{\ell} x^{*},\delta)\right)$. 
Since the lens $L_{1}$ and the disks $D(\omega^{\ell} x^{*},\delta)$ 
can be taken as small as we like, it follows that \eqref{strongasympform} 
is also valid uniformly for $z$ in compact subsets of $\mathbb{C}\setminus\Sigma_{1}^*$.

\section{Acknowledgements}

The first author is supported by KU Leuven Research Grant OT/12/073, the Belgian
Interuniversity Attraction Pole P07/18, and the FWO Flanders projects G.0641.11 and G.0934.13.

The second author was supported by a postdoctoral fellowship from the Fund 
for Scientific Research - Flanders (Belgium)  during his stay at the University of Leuven. 

The authors are grateful to the participants of the AIM workshop 
on Vector Equilibrium problems and their Applications to Random Matrix Models in April 2012
for fruitful and inspiring discussions on the normal matrix model.

\end{document}